\documentclass[a4paper,11pt]{article}
\title {Tesis Szyld}

\usepackage{hyperref}
\usepackage[latin1]{inputenc}
\usepackage[spanish]{babel}
\usepackage{color}
\usepackage{verbatim}
\usepackage{amsmath, amsthm, amsfonts, amssymb}
\usepackage[all]{xy}
\usepackage{graphicx}

	\def \A{\mathbb{A}}
	\def \C{\mathbb{C}}
	\def \T{\mathbb{T}}
	\def \Cat{\mathcal{C}}
	\def \Dat{\mathcal{D}}
	\def \Bat{\mathcal{B}}
	\def \Xat{\mathcal{X}}
	\def \Rat{\mathcal{R}}
	\def \Vat{\mathcal{V}}
	\def \Ens{\mathcal{E}ns}
	\def \G{\mathbb{G}}
	\def \Gr{\mathbb{G}r}

	\def \Z{\mathbb{Z}}
	
	\def \be{\begin{enumerate}}
	\def \en{\end{enumerate}}

\def \de{definition}
\def \te{theorem}
\def \prop{proposition}
\def \ej {example}

\newcommand{\mr}[1]{\buildrel {#1} \over \longrightarrow}

\newcommand{\dr}[1]
          {
           \ar@<4pt>@{-}'+<0pt,-6pt>[d] 
           \ar@<-4pt>@{-}'+<0pt,-6pt>[d]^{#1}
          }

\newcommand{\dd}{\ar@2{-}[d]}

\newcommand{\op}[1]
          {
           \ar@{-}[ld] 
           \ar@{-}[rd] 
           \ar@{}[d]|{#1}  
          }

\newcommand{\cl}[1]
          { 
           \ar@{-}[ur] 
           \ar@{}[u]|{#1} 
           \ar@{-}[ul] 
          }

\newcommand{\ig}[1]
          {
           \ \ \ \ar@{}[d]|{\stackrel{#1}{=}}
          }

\newcommand{\dcell}[1]
          {
           \ar@<4pt>@{-}'+<0pt,-6pt>[d] 
           \ar@<-4pt>@{-}'+<0pt,-6pt>[d]^{#1}
          }

\newcommand{\did}       
         {
          \ar@{=}[d]
         }

\newcommand{\pmr}[2]
{
\xymatrix@C=5ex@R=2.4ex
         {
          {} \ar@<1.6ex>[r]^{#1} 
	     \ar@<-1.1ex>[r]^{#2} & {}
         }
}

\newcommand{\pml}[2]
{
\xymatrix@C=5ex@R=2.4ex
         {
            {} 
          & {} \ar@<1.0ex>[l]_{#1} 
	       \ar@<-1.7ex>[l]_{#2}
         }
}

\newcommand{\cellr}[3]
{
\xymatrix@C=7ex@R=2.4ex
         {
          {} \ar@<1.6ex>[r]^{#1} 
          \ar@{}@<-1.3ex>[r]^{\!\! #2 \, \!\Downarrow}
                                         \ar@<-1.1ex>[r]_{#3} & {}
         }
}

\newcommand{\celll}[3]
{
\xymatrix@C=7ex@R=2.4ex
         {
            {} 
          & {} \ar@<1.0ex>[l]^{#1} 
          \ar@{}@<-1.7ex>[l]^{\!\! #2 \, \!\Downarrow}
	                                 \ar@<-1.7ex>[l]_{#3}
         }
}

\newcommand{\cc}{\mathcal}

\theoremstyle{plain}
		\newtheorem{theorem}{Teo}[section]
		\newtheorem{lemma}[theorem]{Lema}
\theoremstyle{definition}
		\newtheorem{remark}[theorem]{Obs}
	
\numberwithin{equation}{section}	
		
\begin{document}

\newcommand{\colim}{\mathop{\hbox{colím}}}

\begin{center}

\textbf{On Tannaka duality}

\vspace{3ex}

Tesis de Licenciatura\footnote{Departamento de Matematica, F.C.E. y
  N., University of Buenos Aires.} - December 23, 2009

\vspace{3ex}

\textit{Martín Szyld}

\vspace{7ex}

\textbf{Abstract}
\end{center}

\footnotesize
The purpose of this work is twofold: to expose the existing
similarities 
between the generalizations of the Tannaka and Galois theories,  and
on the other hand, to develop in detail our own treatment of part of
the 
content of Joyal and Street \cite{JS} paper, generalizing from vector
spaces to an abstract tensor category. We also develop in detail the
proof of the Tannaka equivalence of categories in the case of vector
spaces.  

Saavedra Rivano \cite{Saavedra}, Deligne and Milne \cite{DM}
generalize classical Tannaka theory to the context of $K$-linear
tensor (or monoidal)  categories.  They  obtain a 
lifting-equivalence into a category of ``group representations'' for a finite-dimensional vector space valued monoidal functor. 
This lifting theorem is similar to the one of Grothendieck Galois theory \cite{SGA} for a finite sets valued functor. On the other hand, Joyal and Street \cite{JS} work on the algebraic side of the 
duality between algebra and geometry, and also obtain a lifting-equivalence, but now to the category of finite-dimensional comodules over a Hopf algebra.

In this work we follow the ideas of  Joyal and Street and prove their lifting-equivalence, and then we show how it corresponds, by dualizing, to the ones of Saavedra Rivano \cite{Saavedra} and Deligne and Milne \cite{DM}.

Text is in spanish, but a brief english introduction is provided.
\normalsize

\vspace{7ex}

\begin{center}
\textbf{Abbreviated Introduction}
\end{center}

\vspace{1ex}

Consider the following Galois context, as developed in \cite {SGA}:

$$\xymatrix { & c\Ens_{<\infty}^{\pi} \ar[d]^{U} \\
				\Cat \ar[ur]^{\tilde{F}} \ar[r]^{F} & \Ens_{<\infty}. }$$

Starting with a category $\Cat$ and a functor $F: \Cat \rightarrow
\Ens_{<\infty}$ into the category of finite sets, Grothendieck constructs a lifting $\tilde{F}: \Cat \rightarrow c\Ens_{<\infty}^{\pi}$, into the category of continuous finite actions of the pro-finite group $\pi = Aut(F)$. He then proves a theorem where he gives conditions for $\Cat$ and $F$ under which $\tilde{F}$ is an equivalence of categories.

\vspace{1ex}

We now turn to the development of modern Tannaka theory. Broadly
speaking, starting with a category $\Cat$ and a functor $F: \Cat
\rightarrow Vec_K^{<\infty}$ into the category of finite dimensional
vector spaces, we construct a lifting and obtain conditions under
which this lifting is an equivalence of categories. This lifting can
be into the category of representations of an (algebraic group) (as in \cite{Saavedra}
and \cite{DM}), or into the category of comodules of a Hopf algebra (as in \cite{JS}). The theorems that specify this conditions are \cite{Saavedra}, p.198, 3.2.2.3, \cite{DM}, p.130, Theorem 2.11 and \cite{JS}, p. 456, Theorem 3.

We consider a  Tannaka context which consists of a pair $(\cc{V}_0,
\cc{V})$, where $\cc{V}$ is a monoidal closed category, and  $\cc{V}_0
\subset \cc{V}$ is a small category of objects with duals, and a
functor $\cc{C} \mr{F} \cc{V}_0$.
$$\xymatrix { & Comod_{<\infty}(End^\lor(F)) \ar[d]^{U} \\
				\Cat \ar[ur]^{\tilde{F}} \ar[r]^{F} &
                                \cc{V}_0. }$$

Starting with $\Cat$ and $F$, a coalgebra  $End^\lor(F)$ (which
classifies endomorphisms of the functor F) and a lifting $\tilde{F}$
are constructed, and conditions are given under which $End^\lor(F)$ is
a Hopf algebra in $\cc{V}$. Then, in the particular case where $\cc{V} = Vec_K$
and $\cc{V}_0 = Vec_K^{<\infty}$,  if $\cc{C}$ is abelian, and $F$ is
exact, $\tilde{F}$ an equivalence of categories. Throughout this
thesis we develop all this, and show explicitely in the last section
how the approach in \cite{Saavedra}
and \cite{DM} is equivalent to the one in \cite{JS}.

\vspace{1ex}

The content of this thesis can be split into two: sections \ref{sec:CatTens}-\ref{sec:Ends} and sections \ref{sec:NatPredual}-\ref{sec:OtroLadoDualidad}. The first four sections are mainly introductory, since we develop there 
the concepts that we will use in the other four. 

\begin{itemize}
	\item In section \ref{sec:CatTens} we define the basic concepts of tensor categories. It is within them that modern Tannaka thoery is developed, so this definitions are fundamental for our work. We also develop here the \textit{``elevators calculus''} invented by E.Dubuc, which is the tool that will allow us to do calculations in these categories. We prove a simple coherence theorem for symmetric tensor categories, which is a particular case of the one of \cite{JS1992} for braided categories, that justifies the use of this \textit{calculus}.
	\item In section \ref{sec:Dualidad} we define the important concept of dual pairing. In \cite{JS} and in \cite{DM} the authors use two (a priori) different definitions of duality, so we expose both and establish relations between them. We prove that they are equivalent, but in the process we also obtain a formulation of the concept of rigid category (of \cite{DM}) equivalent to the usual one but with fewer axioms and valid also in the non-symmetric case. 
	\item In section \ref{sec:DefAlg} we define the different algebraic structures that we'll consider, in an abstract tensor category $\Vat$. We will work with functors $\Cat \rightarrow \Vat$, where $\Vat$ plays the role of $Vec$, so as usual, definitions and proofs are given in terms of arrows of $\Vat$, and not elements.
	\item Finally, in section \ref{sec:Ends} we introduce the ends and coends that, even though they are not used explicitely in \cite{JS}, give a simple and convenient way of defining the object $Nat^\lor(F,G)$ and working with it.
\end{itemize}

In the last four sections is where we actually develop Tannaka theory in an abstract setting, and compare it with Galois theory as we explained before. The following observations should also be made concerning this development:

\begin{itemize}
	\item In section \ref{sec:NatPredual} we give Joyal's definition of $Nat^\lor(F,G)$, though properly generalized to the abstract case, and without assuming that $\Vat$ is symmetric. We obtain the coalgebra structure of $End^\lor(F) := Nat^\lor(F,F)$ as a consequence of the $\Vat$-cocategory structure for the functors $\Cat \rightarrow \Vat$.
	\item In section \ref{sec:EndPredual} we focus on the coalgebra $End^\lor(F)$ and show under which conditions it is a bialgebra and a Hopf algebra. We also show how the lifting $\Cat \stackrel{\tilde{F}}{\rightarrow} Comod_0(End^\lor(F))$ into the category of comodules is constructed from $\Cat \stackrel{F}{\rightarrow} \Vat$.
	\item In section \ref{sec:Vec} we show that in the $\Vat = Vec_K$ case, under certain additional hypothesis this lifting is an equivalence of categories.
	\item To compare this Tannaka theory with Galois theory, in section \ref{sec:OtroLadoDualidad} we show in detail how the Hopf algebra structure of $End^\lor(F)$ corresponds to a group object structure in the category dual of the category of algebras, and that the representations also correspond. This way, we express the Tannaka equivalence theorem in a geometric Galois lenguage.
\end{itemize}

We state explicitally the following theorem (\ref{teofinal}, the last result of the thesis) that shows this phenomenon.

\vspace{2ex}

\noindent \textbf{Theorem.} Let $\Cat$ be a tensor category with a right duality, $\Vat$ a symmetric cocomplete tensor category with internal Hom, $\Vat_0$ a full subcategory of $\Vat$ in which every object has a dual object, and $F: \Cat \rightarrow \Vat_0$ a tensor functor. Then we have the lifting
	
$$\xymatrix {  & Rep_0(G) \ar[d]^{U} \\
							{\Cat} \ar[ur]^{\tilde{F}} \ar[r]^{F} & \Vat_0, }$$
							

	where $G:=Spec(End^\lor(F))$ plays the role of the group $Aut(F)$ of Galois theory.
	
	In addition, in the case $\Vat = Vec_K$, if $\Cat$ is abelian and $F$ is exact and faithful, then $\tilde{F}$ is an equivalence of categories.

\pagebreak

\small
\tableofcontents
\normalsize

\pagebreak

\pagestyle{plain}

\section{Introducción}

El objetivo principal de esta tesis es exponer las relaciones existentes entre las generalizaciones de las teorías de Tannaka y de Galois, con el objetivo a futuro de ensayar una teoría que generalice ambas.

Para ello expondremos, en esta introducción, en primer lugar los enunciados principales de la teoría de Galois de Grothendieck como está desa- \linebreak rrollada en \cite{SGA} y \cite{Dubuc}, luego la teoría clásica de Tannaka como está explicada en \cite{JS} y, por último, la teoría moderna de Tannaka según sus dos vertientes principales, la de \cite{DM} y la de \cite{JS}.

Como mostraremos a continuación, la relación entre la teoría moderna de Tannaka en su versión \cite{JS} y la teoría de Galois de Grothendieck no es tan directa como la de la versión \cite{DM}. En el cuerpo principal de esta tesis, trabajaremos siguiendo las definiciones de \cite{JS}, pero al final de la misma mostraremos cómo nuestros resultados se pueden traducir también al lenguaje de \cite{DM}, más cercano al de la teoría de Galois de Grothendieck. Esto está explicitado con mayor detalle en la sección \ref{sub:contenido}.

Por dudas sobre la notación al leer esta introducción, referirse primero a la sección \ref{notac}.

\subsection{La generalización de Grothendieck de la teoría de Galois} \label{sec:introGal}

En el SGA I \cite{SGA}, Grothendieck desarrolla una teoría que tiene como caso particular la correspondencia de Galois entre subextensiones de una extensión galoisiana $L/K$ y subgrupos del grupo de Galois $G(L/K)$. Más precisamente, a partir de una categoría $\Cat$ y de un funtor $F: \Cat \rightarrow \Ens_{<\infty}$ Grothendieck realiza la contrucción de un funtor \linebreak $\tilde{F}: \Cat \rightarrow c\Ens_{<\infty}^{\pi}$, de $\Cat$ en la categoría de acciones continuas finitas del grupo pro-finito $\pi = Aut(F)$ al que se le da una topología conveniente, y demuestra un teorema en el que da condiciones que deben cumplir $\Cat$ y $F$ para que $\tilde{F}$ sea una equivalencia de categorías. Notemos que una tal acción es un morfismo $\pi \rightarrow Aut(S)$, tomando en $Aut(S)$ la topología discreta.

Grothendieck demuestra el teorema pasando al límite ciertas igualdades que obtiene para algunos objetos de $\Cat$ a los que denomina galoisianos, y que bajo las hipótesis del teorema son cofinales en el diagrama de $F$. Podemos pensar como ejemplo inspirador que $\Cat$ es la categoría (dual) de todas las subextensiones galoisianas finitas de una clausura algebraica $\bar{K}/K$. Luego, como se ve al estudiar teoría de Galois de cuerpos, tenemos correspondencias de Galois para las subextensiones galoisianas $L/K$, y estas subextensiones son cofinales en el diagrama de todas las subextensiones de la clausura algebraica. Luego tenemos dos tipos de extensiones distintas, para las galoisianas podemos demostrar directamente una correspondencia entre todas sus subextensiones y los subgrupos de su grupo de Galois, mientras que para demostrar todas estas correspondencias a la vez tenemos que pasarlas al límite. 

En \cite{Dubuc}, Dubuc y S. de la Vega ponen de manifiesto esta diferencia estableciendo los axiomas correspondientes al caso galoisiano (al que llaman \textit{representable connected}), y obteniendo a partir de este, mediante un paso al límite, otro caso (al que llaman \textit{connected}) y el caso de Grothendieck. 


Podemos pensar gráficamente el teorema de \cite {SGA}, mirando el siguiente diagrama.
\begin{equation} \label{diagramaGalois}
\xymatrix { & c\Ens_{<\infty}^{\pi} \ar[d]^{U} \\
				\Cat \ar[ur]^{\tilde{F}} \ar[r]^{F} & \Ens_{<\infty} }
\end{equation}				

En este sentido, podemos decir que $\tilde{F}$ es un "`levantamiento"' de $F$. Luego podemos pensar en dos enunciados distintos, a los que más adelante les buscaremos equivalentes formales en la teoría de Tannaka.
\begin{enumerate}
	\item A partir de un grupo profinito $\pi$ (por definición esto es un límite cofiltrante de grupos finitos discretos con la topología límite) y del funtor de olvido $c\Ens_{<\infty}^{\pi} \stackrel{U}{\rightarrow} \Ens_{<\infty}$ se obtiene un morfismo $\pi \rightarrow Aut(U)$ que resulta un isomorfismo 
	\item A partir de una categoría $\Cat$ y un funtor $F: \Cat \rightarrow \Ens_{<\infty}$, podemos dar condiciones que deben cumplir $\Cat$ y $F$ para que el levantado $\tilde{F}$ sea una equivalencia de categorías.
\end{enumerate}

\subsection{La generalización de Joyal-Street de la teoría de Tannaka} \label{sub:introJoyal}

La teoría de la dualidad de Tannaka comienza en \cite{Tannaka}, donde se recupera un grupo topológico compacto a partir de sus representaciones. La descripción de la misma que daremos es la de \cite{JS}, cap.1, pues tiene un enfoque más cercano a lo que desarrollaremos en este trabajo. Lo primero que haremos para ello es recordar algunas nociones básicas de representaciones de grupos.

\subsubsection*{Representaciones de Grupos}


Sean $G$ un grupo, y $\A$ una categoría. Por definición, una representación de $G$ en $\A$ es elegir un objeto $A$ de $\A$, y elegir para cada elemento de $G$ un endomorfismo de $A$ de forma tal que al producto de dos elementos de $\G$ le corresponda la composición de los endomorfismos correspondientes. Observemos que para cualquier elemento $g \in G$, $F(g)$ debe ser un automorfismo de $A$, pues su inversa es $F(g^{-1})$. Luego, otra forma de expresar esto es decir que una representación de $G$ en un objeto $A$ (de una categoría prefijada $\A$) es un morfismo de grupos $F:G \longrightarrow Aut(A)$.

\begin{remark} \label{represcomofuntor}
	Si se interpreta al grupo $G$ como una categoría $\G$ (con un solo objeto y una flecha por cada elemento de $G$, donde la composición viene dada por el producto), entonces una representación de $G$ en $\A$ es un funtor $\G \rightarrow \A$. Los funtores covariantes corresponden a representaciones a izquierda, y los contravariantes a representaciones a derecha.
\end{remark}

Si enriquecemos ahora a $G$ y a $Aut(A)$ con una estructura de espacio topológico, podemos pedir que el morfismo $F$ sea continuo. Estas se llamarán representaciones continuas. Si $A$ es un espacio vectorial $V$ sobre $\C$ de dimensión $n$, tomando una base de $V$, notamos que $Aut(V)$ recibe una topología canónica como subespacio de $\C^{n \times n}$. En esta introducción, cuando hablemos de representaciones a secas nos estaremos refiriendo a las representaciones continuas de grupos topológicos en los $\C$-espacios vecto- \linebreak riales $V$ de dimensión finita.

Dada una representación $F$ de $G$ en $V$, podemos armar otras representaciones. Consideramos en primer lugar el espacio topológico $\overline{V}$ cuyos elementos están en biyección con los de $V$, pero donde la multiplicación está conjugada en el siguiente sentido: si notamos con $\overline{v}$ a los elementos de $\overline{V}$ para cada $v \in V$, entonces $\bar{} : V \longrightarrow \overline{V}$, la función que manda $v$ a $\overline{v}$, resulta un isomorfismo antilineal. O sea, $\overline{\alpha v + \beta w} = \bar{\alpha}\bar{v} + \bar{\beta}\bar{w} \  \forall v,w \in V$, $\alpha,\beta \in \C$. Luego podemos considerar la representación $\overline{F}$ de $G$ en $\overline{V}$ dada por $\overline{F}(g)(\bar{v}) = \overline{F(g)(v)}$. $\overline{F}(g)$ es un endomorfismo de $\overline{V}$ pues, en general, si $T$ es una transformación lineal de $V$ a $W$ y definimos $\overline{T}: \overline{V} \longrightarrow \overline{W} $ según $\overline{T}(\bar{v}) = \overline{T(v)}$, entonces $\overline{T}$ es una transformación lineal de $\overline{V}$ a $\overline{W}$:
\begin{eqnarray} \label{eq:Tbarra}
\overline{T}(\alpha \bar{x} + \beta \bar{y}) & = & \overline{T(\bar{\alpha} x + \bar{\beta} y)} \nonumber \\
& = & \overline{\bar{\alpha} T(x) + \bar{\beta} T(y)} \nonumber \\
& = & \alpha \overline{T(x)} + \beta \overline{T(y)} \nonumber \\
& = & \alpha \overline{T}(\bar{x}) + \beta \overline{T}(\bar{y}) 
\end{eqnarray}  
Resta verificar que $\overline{F}(gg') = \overline{F}(g) \circ \overline{F}(g')$. En efecto: 
\begin{eqnarray*} 
\overline{F}(gg')(\bar{v}) & = & \overline{F(gg')(v)} \\
& = & \overline{F(g)(F(g')(v))} \\
& = & \overline{F}(g) (\overline{F(g')(v)}) \\
& = & \overline{F}(g) (\overline{F}(g')(\bar{v})) \\
& = & (\overline{F}(g) \circ \overline{F}(g'))(\bar{v}) \hbox{ si } \bar{v} \in \overline{V}
\end{eqnarray*}
En segundo lugar, miramos $V^\lor$, el espacio dual de $V$ (usualmente notado como $V^*$), y definimos la representación $F^\lor : G \longrightarrow End(V^\lor)$ según $F^\lor (g)(\varphi)(v) = \varphi (F(g^{-1})(v)) \ \forall \varphi \in V^\lor, \ v \in V$. En otras palabras, \linebreak $F^\lor (g) = F(g^{-1})^\lor$, donde si $T: V \longrightarrow W$ es una transformación lineal entre espacios vectoriales, $T^\lor : W^\lor \longrightarrow V^\lor$ es la transformación adjunta definida por $T^\lor(\varphi)(v) = \varphi (T(v)) \ \forall \varphi \in W^\lor, \ v \in V$. Es conocido el hecho que $T^\lor$ es una tranformación lineal, resta verificar que $F^\lor(gg') = F^\lor(g) \circ F^\lor(g')$:
\begin{center}
$F^\lor(gg')(\varphi)(v) = \varphi (F((gg')^{-1})(v)) = \varphi (F(g'^{-1}g^{-1})(v)) = \varphi (F(g'^{-1})(F(g^{-1})(v))) = F^\lor(g')(\varphi) (F(g^{-1})(v)) = (F^\lor(g) (F^\lor(g')(\varphi)))(v) = (F^\lor(g) \circ F^\lor(g'))(\varphi)(v)$ si $\varphi \in V^\lor$, $v \in V$.
\end{center}

Por otro lado, dadas dos representaciones $F$ y $H$ de $G$ en $V$ y $W$ \linebreak respectivamente, podemos armar su producto tensorial $F \otimes H$ de $G$ en $V \otimes W$ definido según $F \otimes G (g)(v \otimes w) = F(g)(v) \otimes G(g)(w)$. Tenemos además un neutro para este producto tensorial que es la representación trivial unidimensional $I: G \longrightarrow End(\C)$ que a cada $g \in G$ le asigna la identidad.

\vspace{2ex}

Dados un grupo $G$, dos objetos $A$ y $B$ de una categoría $\A$ y dos representaciones $F:G \rightarrow Aut(A)$ y $H:G \rightarrow Aut(B)$, un morfismo de representaciones $\theta$ entre $F$ y $H$ es, por definición, una flecha $\theta : A \rightarrow B$ en la categoría $\A$ de forma tal que $\forall g \in G$, el cuadrado
\begin{equation} \label{morfrep}
 \xymatrix{ A \ar[r]^{F(g)} \ar[d]_\theta & A \ar[d]^\theta \\
					 B \ar[r]^{H(g)} & B}
\end{equation}					 
conmuta.

\begin{remark}
Interpretando las representaciones como funtores como en la observación \ref{represcomofuntor}, un morfismo de representaciones equivale a una transformación natural $\theta: F \Rightarrow H$. 
\end{remark}

De esta forma, el conjunto de las representaciones de $G$ en $\A$ resulta una categoría. A la categoría de representaciones de $G$ (en los $\C$-e.v. de dimensión finita) la notamos $Rep(G,\C)$.
 
\begin{remark}
Observemos que, por definición, una representación de $G$ en \linebreak $S \in \Ens$ no es otra cosa que una acción de $G$ en $S$. También observemos que un morfismo entre dos de estas representaciones equivale a un morfismo de acciones.
\end{remark}
 
Un morfismo de representaciones $\theta$ entre $F$ y $G$ en $Rep(G,\C)$ viene dado entonces por una transformación lineal $\theta : V \rightarrow W$ que induce, como vimos en \eqref{eq:Tbarra}, una transformación lineal $\bar{\theta}: \overline{F} \rightarrow \overline{G}$, que resulta también un morfismo entre las representaciones conjugadas: debemos verificar que 
$$
 \xymatrix{ {\overline{V}} \ar[r]^{\overline{F}(g)} \ar[d]_{\overline{\theta}} & {\overline{V}} \ar[d]^{\overline{\theta}} \\ 
					 {\overline{W}} \ar[r]_{\overline{H}(g)} & {\overline{W}} }
$$
conmuta, o sea que $\forall \bar{v} \in \overline{V}$, $\bar{\theta}(\overline{F}(g)(\bar{v})) = \overline{H}(g)(\bar{\theta}(\bar{v}))$, i.e.
\begin{center}
$\bar{\theta}(\overline{F(g)(v)}) = \overline{H}(g)(\overline{{\theta}(v)})$, i.e.

$\overline{\theta(F(g)(v))} = \overline{H(g)(\theta(v))}$. 
\end{center}
Pero esta ecuación es la conjugada de la que obtuvimos en \eqref{morfrep}.

Se deja al lector realizar un trabajo similar con las representaciones duales para obtener $\theta^\lor$.

\subsubsection*{Teoría Clásica de Tannaka}

Ahora sí estamos en condiciones de describir la teoría clásica de la dua-\linebreak lidad de Tannaka como está expuesta en \cite{JS}.

Sea $G$ un grupo topológico, y sea $Rep(G,\C)$ la categoría de las representaciones de $G$ en los $\C$-e.v. de dimensión finita, donde las flechas son los morfismos de representaciones. Tenemos un funtor de olvido \linebreak $U: Rep(G,\C) \longrightarrow Vec^{<\infty}_{\C}$, que manda una representación $F:G\longrightarrow V$ a $V$ y un morfismo $\theta : V \longrightarrow W$ a $\theta$. Nos interesa entender el conjunto $End(U)$.

Una transformación natural $u:U\Rightarrow U$ es, por definición, una familia de transformaciones lineales $u_{F,V} \in End(V)$, una para cada representación $F:G\longrightarrow V$, de forma tal que para cada $\theta : V \longrightarrow W$ morfismo entre las representaciones $F$ y $H$, el diagrama
\begin{equation} \label{utn}
\xymatrix{ V \ar[r]^\theta \ar[d]_{u_{F,V}} & W \ar[d]^{u_{G,W}} \\
					 V \ar[r]^\theta & W}
\end{equation}
conmuta.

A $End(U)$ le podemos dar la siguiente estructura:
\begin{itemize}
	\item Para cada $(F,V) \in Rep(G,\C)$, tenemos una función \linebreak $End(U) \rightarrow End(V) \subset \C^{n \times n}$ que a cada $u$ le asigna $u_{F,V}$. Podemos darle a $End(U)$ la topología inicial respecto de estas funciones.
	\item $End(U)$ tiene una estructura de $\C$-álgebra con la estructura de $\C$-e.v. punto a punto y con la composición $(u \circ u')_{F,V} = u_{F,V} \circ u'_{F,V}$.
	\item También tenemos una conjugación $\bar{} : End(U) \longrightarrow End(U)$, definiendo $\bar{u}_{F,V}(v) = \overline{u_{\overline{F},\overline{V}}(\bar{v})}$. Se verifica que $\bar{u} \in End(U)$ pues eso significa (observando el diagrama \eqref{utn}) que $\theta(\bar{u}_{F,V}(v)) = \bar{u}_{G,W}(\theta(v))$, y esto equivale por definición de $\bar{u}$ a $\theta(\overline{u_{\bar{F},\bar{V}}(\bar{v})}) = \overline{u_{\bar{G},\bar{W}}(\overline{\theta(v)})}$, lo cual se obtiene del diagrama \eqref{utn} aplicado a $\overline{F}$, $\overline{G}$ y $\bar{\theta}$. Diremos que \linebreak $u \in End(U)$ es autoconjugada si $\bar{u} = u$.
\end{itemize}

Como vimos anteriormente, dadas dos representaciones $F$ y $H$ de $G$ en $V$ y $W$ respectivamente, podemos armar su producto tensorial $F \otimes H$ que es una representación de $G$ en $V \otimes W$. También definimos la representación $I$, que es el neutro para este producto tensorial. En la sección \ref{sec:CatTens} veremos esto con mayor generalidad, pero esta es una estructura de categoría tensorial de $Rep(G,\C)$ que convierte a $U$ en un funtor entre categorías tensoriales. Más adelante definiremos en general qué significa que $u \in End(U)$ preserve este producto, pero en este caso se traduce a que $u_{F \otimes H, V \otimes W} = u_{F,V} \otimes u_{H,W}$, y que $u_{I,\C} = id_{\C}$.

Observemos ahora que cada $g \in G$ nos da un $u^g \in End(U)$ según \linebreak $u^{g}_{F,V} = F(g)$. Así obtenemos una función $u^{(-)} : G \longrightarrow End(U)$ Para poder recuperar $G$ a partir de $End(U)$, un camino posible es caracterizar las $u \in End(U)$ tales que $u = u^g$ para algún $g \in G$, o sea la imagen de esta función. Este camino es el de la dualidad de Tannaka.

La caracterización es la siguiente: sea $\T(G)$ el conjunto de las $u$ de $End(U)$ que preservan el producto tensorial y son autoconjugadas. Entonces:
\begin{itemize}
	\item Pidiendo solamente que $G$ sea un monoide, $\T(G)$ es un monoide con la composición y $u^{(-)} : G \longrightarrow \T(G)$ es un morfismo de monoides.
	\item $\T(G)$ es un grupo topológico compacto si $G$ lo es.
	\item $u^{(-)} : G \longrightarrow \T(G)$ es, en este caso, un isomorfismo de grupos topológicos.
\end{itemize}

Notemos la similitud de estas propiedades con el primer enunciado de la teoría de Galois de Grothendieck como fue expuesta en la sección \ref{sec:introGal}: en ambos casos a partir de un grupo y del funtor de olvido de la categoría ya sea de representaciones o de acciones, se trata de recuperar al grupo, definiendo un morfismo del grupo al grupo de automorfismos del funtor de olvido y a continuacion demostrando que es un isomorfismo.

Las tres propiedades anteriores están probadas en el capítulo 1 de \cite{JS} y no las demostraremos en esta tesis pues lo que nos interesa son los resultados categóricos que motivan.

\subsubsection*{Teoría Moderna de Tannaka}

Las mayores referencias para este tema son, en orden cronológico, \cite{Saavedra}, \cite{DM} y \cite{JS}. En esta tesis, el trabajo que seguiremos más de cerca es el de Joyal-Street. En los tres trabajos, ya no se trata solamente de recuperar un objeto a partir de sus representaciones. A partir de una categoría $\Cat$ y un funtor $F$ a los espacios vectoriales de dimensión finita, se buscan propiedades para $\Cat$ y $F$ que permitan recuperar a la categoría como las representaciones de un objeto (en \cite{Saavedra} y \cite{DM}) o como los comódulos sobre un objeto (en \cite{JS}). Este tipo de teorema se puede encontrar en \cite{Saavedra}, p.198, 3.2.2.3, \cite{DM}, p.130, Theorem 2.11 y \cite{JS}, p. 456, Theorem 3 y corresponde al item 2' de la enumeración del final de esta sección.

En \cite{Saavedra} y \cite{DM}, los autores utilizan propiedades y terminología (y también probablemente el modo de pensar) de la geometría algebraica y la cohomología no-abeliana. Esto tiene sin duda una ventaja, que es que se obtiene un teorema mucho más similar al de \cite{SGA} que en \cite{JS}, por ejemplo el teorema de \cite{DM} recién mencionado dice que bajo ciertas condiciones para una categoría $\Cat$ y un funtor $w: \Cat \rightarrow Vec_K^{<\infty}$, tenemos un esquema grupo afín $G$ que re- \linebreak presenta al funtor $Aut^{\otimes}(w)$ y una equivalencia de categorías $\Cat \rightarrow Rep_K(G)$. Luego este teorema se puede interpretar de forma directa como un análogo vectorial al de la teoría de Galois de Grothendieck. Observemos que no es una generalización del mismo.

Sin embargo, todo el bagaje teórico que utilizan estos dos trabajos tiene también una desventaja. En primer lugar, lo cual no es en nuestra opinión algo menor, los hacen de difícil acceso a los matemáticos y estudiantes no adentrados en estos temas (entre ellos el autor de esta tesis). Pero además, y este sí es un motivo más objetivo para haber elegido seguir a \cite{JS}, no se puede observar de forma tan clara cuáles son las condiciones para la categoría y el funtor realmente necesarias para el teorema, y cuáles sólo para poder encuadrarlo en este marco teórico. Ante un objetivo a futuro de desarrollar una teoría más general, tener bien en claro cuáles son las hipótesis que no se pueden omitir es algo fundamental.

Una forma sintética de diferenciar el artículo de \cite{JS} de los otros dos, es decir que Joyal trabaja "`del lado algebraico de la dualidad"' (veremos esto con todo detalle en la sección \ref{sec:OtroLadoDualidad}). Dicho con un poco más de precisión, en lugar de trabajar con un esquema grupo afín $G$, Joyal construye a partir de un funtor $F: \Cat \rightarrow Vec_K^{<\infty}$ un espacio vectorial $End^\lor(F)=Nat^\lor(F,F)$ que cumple que $Hom(End^\lor(F),K)$ coincide con el espacio vectorial $End(F)$ de los endomorfismos de $F$ (escribimos $Hom(End^\lor(F),K)$ y no $(End^\lor(F))^\lor$ pues no será esta una dualidad en el sentido fuerte que veremos más adelante). Luego Joyal muestra la distinta estructura que $End^\lor(F)$ va obteniendo, a medida que se aumentan las hipótesis sobre $\Cat$ y $F$: 
\begin{itemize}
	\item Sin ninguna hipótesis adicional $End^\lor(F)$ resulta una $K$-coálgebra.
	\item Si $\Cat$ y $F$ son tensoriales, $End^\lor(F)$ resulta una $K$-biálgebra.
	\item Si todo objeto de $\Cat$ tiene un dual a izquierda, $End^\lor(F)$ resulta una $K$-álgebra de Hopf.
\end{itemize}

Los $K$-e.v. $FC$ obtienen una estructura de $End^\lor(F)$-comódulos, y se construye así un levantamiento de $F$ según el diagrama

\begin{equation} \label{diagramaJoyal}
\xymatrix { & Comod_{<\infty}(End^\lor(F)) \ar[d]^{U} \\
				\Cat \ar[ur]^{\tilde{F}} \ar[r]^{F} & Vec_K^{<\infty} }
\end{equation}	

Surge entonces la pregunta de qué condiciones sobre $\Cat$ y $F$ nos permiten probar que $\tilde{F}$ es una equivalencia de categorías, es decir qué propiedades caracterizan a las categorías de comódulos sobre una coálgebra.

Podemos encontrar, en el desarrollo de Joyal, dos enunciados correspondientes a los que ya vimos para la teoría de Galois de Grothendieck, de los cuales el primero también corresponde al desarrollo que expusimos de la teoría clásica de Tannaka. Los notaremos con 1' y 2' para marcar esta analogía.

\begin{enumerate}
	\item[1'.] Dada una $K$-coálgebra $E$, y el funtor de olvido \linebreak $U: Comod_{<\infty}E \rightarrow Vec_K^{<\infty}$, se tiene un morfismo $\tilde{\alpha}: End^\lor(U) \rightarrow E$ que resulta un isomorfismo de $K$-coálgebras. (\cite{JS}, 6, p.453, proposition 5)
	\item[2'.] Si $\Cat$ es abeliana y enriquecida en $Vec_K$, y $F$ es aditivo, exacto y fiel, entonces en el diagrama \eqref{diagramaJoyal} $\tilde{F}$ es una equivalencia de categorías (\cite{JS}, 7, p. 456, Theorem 3)
\end{enumerate}

\subsection{El contenido de esta tesis} \label{sub:contenido}

	
	El contenido de esta tesis puede separarse en dos: las secciones \ref{sec:CatTens}, \ref{sec:Dualidad}, \ref{sec:DefAlg} y \ref{sec:Ends}; y las secciones \ref{sec:NatPredual}, \ref{sec:EndPredual}, \ref{sec:Vec} y \ref{sec:OtroLadoDualidad}. Puede pensarse que las primeras cuatro secciones son principalmente introductorias, pues en ellas desarrollamos más que nada los conceptos que vamos a utilizar en las siguientes cuatro. Sin embargo, el desarrollo que hacemos en ellas de los diferentes conceptos lo realizamos apuntando al trabajo que haremos después. Más precisamente:
	
\begin{itemize}
	\item En la sección \ref{sec:CatTens} definimos los conceptos básicos de las categorías tensoriales. Dentro de estas categorías se encuadra la teoría moderna de Tannaka, por lo que estas son fundamentales para este trabajo. En esta sección desarrollamos también el "`cálculo de ascensores"' ideado por E. Dubuc, que es la herramienta que nos permitirá realizar cálculos dentro de estas categorías. La validez del uso que de él hacemos está justificada por el teorema de coherencia para categorías tenso- \linebreak riales simétricas, que probamos también en esta sección.
	\item En la sección \ref{sec:Dualidad} definimos qué quiere decir en una categoría tensorial que un objeto tenga un dual. En \cite{JS} y en \cite{DM} se utilizan definiciones de dualidad a priori diferentes, por lo que exponemos ambas y mostramos qué relaciones podemos establecer entre ellas. Concluimos, al probar la equivalencia de ambas definiciones, que algunas hipótesis de la definición de \cite{DM} se pueden omitir. Además diferenciamos el caso en el que la categoría tensorial es simétrica de aquel en que no.
	\item En la sección \ref{sec:DefAlg} damos las definiciones de las distintas estructuras algebraicas que vamos a considerar, sobre una categoría tensorial. En lugar de trabajar siempre con funtores llegando a $Vec_K^{<\infty}$, en ciertas partes de esta tesis trabajamos de forma más general con funtores llegando a una subcategoría $\Vat_0$ de una categoría tensorial $\Vat$, en la cual los objetos de $\Vat_0$ tengan un dual en el sentido de la sección \ref{sec:Dualidad}. Por lo tanto, es necesario para poder realizar esto dar las deficiones como en esta sección. 
	\item Finalmente, en la sección \ref{sec:Ends} introducimos a los ends y coends que, si bien no son usados en \cite{JS}, son una forma cómoda y compacta de introducir a $Nat^\lor(F,G)$ y de trabajar con este objeto, y los utilizaremos intensivamente en las secciones siguientes para definirlo y describir su estructura. Ademas mostramos algnas propiedades generales de estos objetos y del producto tensorial entre funtores que también nos servirán luego.
\end{itemize}
	
En las últimas cuatro secciones es que desarrollamos realmente la teoría moderna de Tannaka de forma similar a como fue expuesta en \ref{sub:introJoyal}. Rea- \linebreak lizamos sin embargo algunas modificaciones:

	
	Como ya mencionamos, en algunas partes de esta tesis trabajamos reemplazando la categoría $Vec_K$ por una categoría tensorial $\Vat$ cualquiera. Más precisamente, en las secciones \ref{sec:NatPredual}, \ref{sec:EndPredual} y \ref{sec:OtroLadoDualidad} hacemos esto, mientras que en la sección \ref{sec:Vec} (donde están las demostraciones de los teoremas que se hacen "`con elementos"') nos vemos obligados (por ahora) a trabajar en $Vec_K$.
	 Una de las ventajas de trabajar en una categoría $\Vat$ abstracta es que se pone de ma- nifiesto qué condiciones de la categoría son necesarias para cada propiedad que uno quiera obtener. En particular en nuestro caso, definiendo $Nat^\lor$ de una forma ligeramente diferente a la de Joyal pero que coincide con la de él en el caso simétrico (con duales a derecha en lugar de a izquierda, y relajando la hipótesis de que ambos funtores deban caer en $\Vat_0$), obtenemos que $Nat^\lor$ da una estructura de cocategoría enriquecida en $\Vat$, y por lo tanto $End^\lor(F)$ resulta una coálgebra en $\Vat$, aún en el caso en que $\Vat$ no sea simétrica. La simetría de $\Vat$ sí resulta fundamental en nuestro caso (entre otras condiciones) para que $End^\lor(F)$ tenga una estructura de biálgebra y de álgebra de Hopf.
	
	Otra ventaja de trabajar con $\Vat$ es que nos obliga a trabajar con flechas y no con elementos. Para ayudarnos a realizar los cálculos de composiciones utilizamos el sistema gráfico de ascensores desarrollado en la sección \ref{sec:CatTens}.
	
	
	
	En la sección \ref{sec:NatPredual} damos la definición de $Nat^\lor(F,G)$ de Joyal, aunque ge- \linebreak neralizada adecuadamente al caso en el que la categoría $\Vat$ no es simétrica, y encuadramos a la coevaluación, la cocomposición y la counidad en una estructura que llamamos de "`cocategoría enriquecida"' dual a la de categoría enriquecida. De esta forma obtenemos la estructura de coálgebra de $End^\lor(F)$ de forma dual a como $[C,C]$ (u $Hom(C,C)$) es un monoide con la composición en el caso usual. 
	
	Luego, en la sección \ref{sec:EndPredual} nos concentramos en este objeto y obtenemos para él todo lo siguiente:
\begin{itemize}
	\item A partir de una coálgebra $C$ y el funtor de olvido \linebreak $U: Comod_0C \rightarrow \Vat_0$, obtenemos un morfismo de coálgebras \linebreak $\tilde{\rho}: End^\lor (U) \rightarrow C$. En la sección \ref{sec:Vec} mostraremos que en el caso $\Vat_0 = Vec_K^{<\infty}$ este morfismo resulta biyectivo, es decir probaremos 1' como está expresado en la sección \ref{sub:introJoyal}.
	\item A partir de un funtor $F: \Cat \rightarrow \Vat_0$, construimos un levantamiento de este a la categoría de $End^\lor(F)$-comódulos. En la sección \ref{sec:Vec} mostraremos que en el caso $\Vat_0 = Vec_K^{<\infty}$ bajo ciertas hipótesis adicionales este levantamiento resultará una equivalencia de categorías (es decir probaremos 2').
	\item Obtenemos para $End^\lor(F)$ las mismas propiedades de biálgebra y álgebra de Hopf que Joyal pero ahora sobre $\Vat$, con alguna pequeña diferencia (que desaparece en el caso simétrico) debida a que adoptamos para nuestra definición duales a derecha. 
\end{itemize}

	 
	 Como ya mencionamos antes, y se observa comparando los diagramas \eqref{diagramaJoyal} y \eqref{diagramaGalois}, la similitud de 2' con 2 no es tan directa como en \cite{DM}. Para obtenerla, en la sección \ref{sec:OtroLadoDualidad} mostramos cómo la estructura de álgebra de Hopf de $End^\lor(F)$ se corresponde con ser un objeto grupo en la categoría dual de las álgebras, definimos las representaciones de este objeto grupo en un objeto $V$ de $\Vat$ y mostramos que estas se corresponden con una estructura de $End^\lor(F)$-comódulo para $V$. De esta forma 2' nos quedará expresado en una forma análoga a la de \cite{Saavedra}, \cite{DM} y \cite{SGA}.

\subsection{Notaciones y otras aclaraciones} \label{notac}

El lector de este trabajo debería manejar los conceptos de categoría, funtor, transformación natural, límite, colímite, colímite filtrante, Yoneda y adjunción. Todos ellos están desarrollados en \cite{ML}. Se intentó adoptar para este trabajo un estilo explicativo, tratando de dejar sólo detalles a cargo del lector.

Notaremos con $\Ens$ a la categoría de los conjuntos, con $\Gr$ a la categoría de los grupos, con $Vec_K$ a la categoría de los $K$-espacios vectoriales y con $Vec_K^{<\infty}$ a la subcategoría de los $K$-e.v. de dimensión finita. Si $\Cat$ y $\Dat$ son dos categorías, notaremos con $\Cat^\Dat$ a la categoría de funtores de $\Dat$ a $\Cat$, donde las flechas son las transformaciones naturales.

Si $\Cat$ es una categoría, y $C$, $C'$ son objetos de $\Cat$, notaremos con $[C,C']$ al conjunto de flechas entre $C$ y $C'$, para diferenciarlo de $Hom(C,C')$ que tendrá otro sentido cuando $\Cat$ tenga hom internos. Al conjunto de tranformaciones naturales entre dos funtores $F$ y $G$ lo notaremos $nat[F,G]$. A las identidades de un objeto, a menudo (especialmente dentro de diagramas) las notaremos como al objeto, por ejemplo a $id_C$ la notaremos $C \stackrel{C}{\rightarrow} C$. Cuando tengamos un isomorfismo (o una biyección), varias veces notaremos con el mismo nombre al isomorfismo y a su inversa. También representaremos una biyección como una raya horizontal que separa a un elemento genérico de cada uno de los conjuntos que están en biyección.

\pagebreak

\section {Categorías Tensoriales} \label{sec:CatTens}

\subsection{Definiciones y propiedades básicas}

La idea es extender los conceptos de producto tensorial y de dualidad de la categoría $Vec_K$ a otras categorías. 

Una primera definición, informal, es que una categoría $\Cat$ es una categoría tensorial si tiene una operación binaria funtorial $\otimes$ entre sus objetos y sus flechas, que es asociativa y tiene neutro "`salvo isomorfismos naturales"'. Además, uno puede pedirle a una categoría tensorial que la operación sea conmutativa, también "`salvo un isomorfismo natural"'. Estas categorías se llaman simétricas. Para evitar repeticiones, introduciremos entre corchetes esa noción, de forma tal que los enunciados que incluyan una parte entre corchetes puedan leerse con o sin ellos dando lugar a dos enunciados dife- \linebreak rentes. También cabe observar que toda categoría tensorial simétrica es una categoría tensorial, por lo tanto los enunciados y definiciones para categorías tensoriales también son válidos para categorías tensoriales simétricas.

Podemos realizar entonces una primera

\begin {\de} 
	Sea $\Cat$ es una categoría. Consideremos $F: \Cat \times \Cat \rightarrow \Cat \times \Cat$ el funtor $(X,Y) \rightarrow (Y,X)$. Una 6[+1]-upla $(\Cat,\otimes,\phi[,\psi],I,l,r)$ es una categoría tensorial [simétrica] si 
\begin{enumerate}
	\item $\otimes: \Cat \times \Cat \rightarrow \Cat$ es un funtor (Notamos $(X,Y) \stackrel{\otimes}{\mapsto} X \otimes Y$ y $(f,g) \stackrel{\otimes}{\mapsto} f \otimes g$)
  \item $\phi : \otimes \circ (id_\Cat \times \otimes) \stackrel{\cong}{\Rightarrow} \otimes \circ (\otimes \times id_\Cat)$ es un isomorfismo natural \linebreak ($\phi_{X,Y,Z} : X \otimes (Y \otimes Z) \stackrel{\cong}{\rightarrow} (X \otimes Y) \otimes Z$ es la asociatividad de $\otimes$ salvo isomorfismo)
  \item \ [ $\psi : \otimes \stackrel{\cong}{\Rightarrow} \otimes \circ F$ es un isomorfismo natural y su inversa coincide con sí mismo ($\psi_{X,Y} : X \otimes Y \stackrel{\cong}{\rightarrow} Y \otimes X$ es la conmutatividad de $\otimes$ salvo isomorfismo) ]
  \item 
  Los funtores $I \otimes (-), (-) \otimes I : \Cat \rightarrow \Cat$ son equivalencias de categorías con inversa $id_\Cat$ y transformaciones naturales $l: I \otimes (-) \Rightarrow id_\Cat$ y $r: (-) \otimes I \rightarrow id_\Cat$ respectivamente ($l_X: I \otimes X \stackrel{\cong}{\rightarrow} X$ y $r_X: X \otimes I \stackrel{\cong}{\rightarrow} X$ muestran que $I$ es un neutro para $\otimes$ salvo isomorfismo).
  \item los diagramas triángulo, pentágono [y hexágono] que mostramos a continuación conmutan.
\end{enumerate}
\end {\de}


Combinando los isomorfismos $\phi$[, $\psi$], $l$ y $r$ podemos armar los siguientes diagramas triángulo, pentágono [y hexágono] que pedimos por definición que conmuten (omitimos en general los subíndices de estas transformaciones naturales):


\xymatrix { X \otimes (I \otimes Y) \ar[dr]_{X \otimes l} \ar[rr]^{\phi} & & (X \otimes I) \otimes Y \ar[dl]^{r \otimes Y} \\
							& X \otimes Y }

\xymatrix {  & (X \otimes Y) \otimes (Z \otimes T) \ar[dr]^{\phi} & \\
						X \otimes (Y \otimes (Z \otimes T)) \ar[ur]^{\phi} \ar[d]_{X \otimes \phi} & & ((X \otimes Y) \otimes Z) \otimes T \\
						X \otimes ((Y \otimes Z) \otimes T) \ar[rr]^{\phi} & & (X \otimes (Y \otimes Z)) \otimes T \ar[u]_{\phi \otimes T} }

\xymatrix { & (X \otimes Y) \otimes Z \ar[r]^{\psi} & Z \otimes (X \otimes Y) \ar[dr]^{\phi} & \\
						X \otimes (Y \otimes Z) \ar[ur]^{\phi} \ar[dr]_{X \otimes \psi} & & & (Z \otimes X) \otimes Y \\
						& X \otimes (Z \otimes Y)  \ar[r]^{\phi} & (X \otimes Z) \otimes Y \ar[ur]_{\psi \otimes Y} }

Abusaremos la notación de la siguiente manera: cuando digamos que una categoría $\Cat$ es tensorial [simétrica], y no explicitemos nada más, supondremos dadas $\otimes$, $\phi$[, $\psi$] e $I$ y usaremos siempre esas letras para ellos, aun en categorías distintas (si la categoría es $\Cat'$ quizás pongamos $\phi'$[, $\psi'$] e $I'$). También omitiremos en general los subíndices de $\phi$ y $\psi$.

Se deduce de los axiomas anteriores que $l_I = r_I : I \otimes I \rightarrow I$ (ver \cite{JS1992}, Proposition 1.1, Page 3). También se tiene la siguiente

\begin{\prop} \label{unsoloneutro}
	Los neutros son todos isomorfos
\end{\prop}

\begin{proof}
	Considerar \xymatrix@1 {I & I \otimes I' \ar[l]_>>>>>>{r'_I} \ar[r]^>>>>{l_{I'}} & I'}
\end{proof}

%

Veamos con más detalle en qué consiste la naturalidad de $\psi$, puesto que luego la utilizaremos. $\psi$ es una transformación natural entre los funtores $\otimes$ y $\otimes \circ F$, eso significa que para flechas $f: C \rightarrow D$ y $f': C' \rightarrow D'$, el cuadrado
\begin{equation} \label{psinat}
\xymatrix { C \otimes C' \ar[r]^{f \otimes f'} \ar[d]_{\psi} & D \otimes D' \ar[d]^{\psi} \\
						C' \otimes C \ar[r]^{f' \otimes f} & D' \otimes D }
\end{equation}
conmuta.

\subsection{Coherencia para categorías tensoriales}

A partir de que los diagramas de la definición (sin el hexágono) conmuten se obtiene que "`todos"' los diagramas de esta forma lo hacen, es decir que se pueden realizar los productos tensoriales en cualquier orden (respecto de la asociatividad, o sea sin permutar) y agregando o quitando al neutro $I$ dando lugar a objetos isomorfos con un único isomorfismo canónico (ver \cite{ML}, VII, 2, pp.161-166). Este hecho se conoce como las propiedades de coherencia de Mac Lane, y lo usaremos en este trabajo realizando los siguientes abusos:
\begin{itemize}
	\item Por la coherencia de la asociatividad, escribiremos todos los productos tensoriales sin paréntesis.
	\item Omitiremos casi siempre escribir al neutro para el producto tensorial, y a las flechas $l$ y $r$.
\end{itemize}

Otra forma de expresar este hecho es definiendo como en \cite{DM}, 1, pp.106-108, para cualquier conjunto finito $A$, un funtor que extiende $\otimes$ a \linebreak $\displaystyle \otimes_{a \in A}: C^A \rightarrow C$ satisfaciendo ciertas propiedades generales de asociatividad, que nos permite realizar los productos tensoriales en algún orden canónico, por ejemplo de izquierda a derecha.

\subsubsection*{Cálculo de Ascensores (concebido por E. Dubuc)}

La coherencia de Mac Lane nos permitirá escribir las composiciones de flechas de forma vertical, con un sistema gráfico que nos facilitará mucho los cálculos. Los diagramas expresan distintas composiciones de productos tensoriales de flechas, donde se omiten los símbolos $\otimes$ y se reemplazan las identidades por dobles líneas. Las composiciones deben leerse de arriba hacia abajo. La funtorialidad del producto tensorial nos permitirá subir o bajar flechas que no tengan "`obstáculos"' en el camino (de allí el nombre de "`ascensores"'), por ejemplo dadas flechas $f:C \rightarrow D$, $f':C' \rightarrow D'$ tenemos las igualdades

\begin{equation}  \label{ascensor}
\xymatrix@C=-0.2pc @R=0.5pc{C {\ar@<4pt>@{-}'+<0pt,-6pt>[dd] \ar@<-4pt>@{-}'+<0pt,-6pt>[dd]^{f}}  & & C' \ar@2{-}[dd] 	& & & 			& & 		& & &  				C \ar@2{-}[dd]  & & C' {\ar@<4pt>@{-}'+<0pt,-6pt>[dd] \ar@<-4pt>@{-}'+<0pt,-6pt>[dd]^{f'}} \\
									 &&&&&	C {\ar@<4pt>@{-}'+<0pt,-6pt>[dd] \ar@<-4pt>@{-}'+<0pt,-6pt>[dd]^{f}} & & C' {\ar@<4pt>@{-}'+<0pt,-6pt>[dd] \ar@<-4pt>@{-}'+<0pt,-6pt>[dd]^{f'}} \\
									 D \ar@2{-}[dd] & & C' {\ar@<4pt>@{-}'+<0pt,-6pt>[dd] \ar@<-4pt>@{-}'+<0pt,-6pt>[dd]^{f'}}  & \textcolor{white}{XX} = \textcolor{white}{XX} & & 			& & 							& \textcolor{white}{XX} = \textcolor{white}{XX} & & 	C {\ar@<4pt>@{-}'+<0pt,-6pt>[dd] \ar@<-4pt>@{-}'+<0pt,-6pt>[dd]^{f}} & & D' \ar@2{-}[dd] \\
									 &&&&& D & & D' \\
									 D & & D'							& & & 										& &									& & & 											D & & D' }
\end{equation}
		
Otro ejemplo es el cuadrado \eqref{psinat} que con estos diagramas nos queda

\begin{equation} \label{swap}
 \xymatrix@C=-0.2pc{ C {\ar@<4pt>@{-}'+<0pt,-6pt>[d] \ar@<-4pt>@{-}'+<0pt,-6pt>[d]^{f}} & & C' {\ar@<4pt>@{-}'+<0pt,-6pt>[d] \ar@<-4pt>@{-}'+<0pt,-6pt>[d]^{f'}}		& & & & & & 				C \ar@2{-}[rrd]|{\psi} & & C' \ar@2{-}[lld]|{\psi} \\
									 D \ar@2{-}[rrd]|{\psi} & & D' \ar@2{-}[lld]|{\psi}	& & \textcolor{white}{XX} = \textcolor{white}{XX} & & & & 		C' {\ar@<4pt>@{-}'+<0pt,-6pt>[d]        
                  \ar@<-4pt>@{-}'+<0pt,-6pt>[d]^{f'}} & & C {\ar@<4pt>@{-}'+<0pt,-6pt>[d]        
                  \ar@<-4pt>@{-}'+<0pt,-6pt>[d]^{f}} \\
									 D & & D'															& & & & & & 									D' & & D }
\end{equation}
									 
Es decir, la naturalidad de $\psi$ nos permite subir o bajar flechas que vengan antes o después que $\psi$. También tenemos, como $\psi$ es su propia inversa, la igualdad

	\begin{equation} \label{psipsi}
 \xymatrix@C=-0.2pc @R=0.5pc {& & 		& & & & & & C \ar@2{-}[rrdd]|{\psi} & & D \ar@2{-}[lldd]|{\psi}  \\
 									 C \ar@2{-}[dd] & & D \ar@2{-}[dd] \\
									  & & 		& & & \textcolor{white}{XX} = \textcolor{white}{XX} & & & D \ar@2{-}[rrdd]|{\psi} & & C \ar@2{-}[lldd]|{\psi}	\\ 
									 C & & D \\
									 	& & 		& & & & & & C & & D}
\end{equation}	

Una vez que usamos la coherencia de Mac Lane, notamos que en las categorías tensoriales simétricas cualquier flecha que involucre a $\phi$, $l$, $r$ y $\psi$ se puede representar solamente escribiendo las $\psi$, por ejemplo el diagrama hexágono nos queda la igualdad

\begin{equation} \label{hex}
\xymatrix@C=-0.3pc@R=0.5pc { X \ar@2{-}[dd] & & Y \ar@2{-}[rrdd]|{\psi} & & Z \ar@2{-}[lldd]|{\psi} & & & & & &  \\	
								 & & & & & & & & & & X \ar@2{-}[rrdd] & & Y \ar@2{-}[rrdd] & & Z \ar@2{-}[lllldd]|{\psi} \\
								 X \ar@2{-}[rrdd]|{\psi} & & Z \ar@2{-}[lldd]|{\psi} & & Y \ar@2{-}[dd]   & & \textcolor{white}{XX} = \textcolor{white}{XX} & & & & \\							 
								 & & & & & & & & & & Z & & X & & Y \\
								 Z & & X & & Y & & & & & & & }
\end{equation}
								 
	que podemos pensar que (usándola repetidas veces) nos da la forma de describir cualquier $\psi$ que involucre a más de dos objetos, en términos de $\psi$ que involucran sólo dos. Notemos también que, como $\psi$ es su propia inversa, y las inversas de las dos flechas descriptas en el diagrama de arriba también deben ser las mismas, tenemos también la igualdad
	
\begin{equation} \label{hex2}
\xymatrix@C=-0.3pc @R=0.5pc{ X \ar@2{-}[rrdd]|{\psi} & & Y \ar@2{-}[lldd]|{\psi} & & Z \ar@2{-}[dd] & & & & & &  \\	
								 & & & & & & & & & & X \ar@2{-}[rrrrdd]|{\psi} & & Y \ar@2{-}[lldd] & & Z \ar@2{-}[lldd] \\
								 Y \ar@2{-}[dd] & & X \ar@2{-}[rrdd]|{\psi} & & Z \ar@2{-}[lldd]|{\psi}   & & \textcolor{white}{XX} = \textcolor{white}{XX} & & & & \\							 
								 & & & & & & & & & & Y & & Z & & X \\
								 Y & & Z & & X & & & & & & & }
\end{equation}	
	
		Recordemos que el grupo simétrico de permutaciones $S_n$ es el grupo libre generado por los ciclos $s_i=(i,i+1)$ para $i=1,...,n-1$ sujetos a las condiciones:
\begin{itemize}
	\item $s_i^2 = 1$, lo cual tenemos para las $\psi$ por \eqref{psipsi} (relajando esta hipótesis se tienen las "`braided tensor categories"' por ejemplo en \cite{JS1992}).
	\item $s_i$ conmuta con $s_j$ si $|i-j|>1$, lo cual tenemos para las $\psi$ por \eqref{ascensor} (si $|i-j|>1$ las $\psi$ no se cruzan al subir y bajar).
	\item $(s_i s_{i+1})^3 = 1$, lo cual tenemos para las $\psi$ por el siguiente argumento:
\end{itemize}

\xymatrix@C=-0.3pc@R=0.5pc { 	
X \ar@2{-}[rrdd]|{\psi} & & Y \ar@2{-}[lldd]|{\psi} & & Z \ar@2{-}[dd] & & & 						X \ar@2{-}[rrdd]|{\psi} & & Y \ar@2{-}[lldd]|{\psi} & & Z \ar@2{-}[dd] & & & X \ar@2{-}[rrdd]|{\psi} & & Y \ar@2{-}[lldd]|{\psi} & & Z \ar@2{-}[dd] 					\\
\\
Y \ar@2{-}[dd] & & X \ar@2{-}[rrdd]|{\psi} & & Z \ar@2{-}[lldd]|{\psi} & & & 						Y \ar@2{-}[rrdd] & & X \ar@2{-}[rrdd] & & Z \ar@2{-}[lllldd]|{\psi}		& & & Y \ar@2{-}[rrdd] & & X \ar@2{-}[rrdd] & & Z \ar@2{-}[lllldd]|{\psi}						\\
& & & & & & & & & & & & & & & & & & & & & X \ar@2{-}[dd] & & Y \ar@2{-}[dd] & & Z \ar@2{-}[dd] \\
Y \ar@2{-}[rrdd]|{\psi} & & Z \ar@2{-}[lldd]|{\psi} & & X \ar@2{-}[dd] & \textcolor{white}{XX} = \textcolor{white}{XX} & & 						Z \ar@2{-}[dd] & & Y \ar@2{-}[rrdd]|{\psi} & & X \ar@2{-}[lldd]|{\psi}	& \textcolor{white}{XX} = \textcolor{white}{XX}  & & 	Z \ar@2{-}[rrrrdd]|{\psi} & & Y \ar@2{-}[lldd] & & X \ar@2{-}[lldd]		& \textcolor{white}{XX} = \textcolor{white}{XX}  & & 				\\
& & & & & & & & & & & & & & & & & & & & & X & & Y & & Z \\
Z \ar@2{-}[dd] & & Y \ar@2{-}[rrdd]|{\psi} & & X \ar@2{-}[lldd]|{\psi} & & & Z \ar@2{-}[rrrrdd]|{\psi} & & X \ar@2{-}[lldd] & & Y \ar@2{-}[lldd]		& & &  Y \ar@2{-}[rrdd]|{\psi} & & X \ar@2{-}[lldd]|{\psi} & & Z \ar@2{-}[dd]						\\
\\
Z \ar@2{-}[rrdd]|{\psi} & & X \ar@2{-}[lldd]|{\psi} & & Y \ar@2{-}[dd] & & & 						X & & Y & & Z 		& & & 				X & & Y & & Z 		 						\\
\\
X \ar@2{-}[dd] & & Z \ar@2{-}[rrdd]|{\psi} & & Y \ar@2{-}[lldd]|{\psi} & & & 								& & & 						\\
\\
X & & Y & & Z 																											& & & 								& & & 						}

\hspace{1ex}

La primer igualdad es por los diagramas \eqref{hex} y \eqref{hex2}, la segunda por la naturalidad de $\psi$ y la tercera pues $\psi$ es su propia inversa.

Luego, cualquier igualdad que sea válida en $S_n$ nos dará una igualdad en cualquier categoría tensorial simétrica al traducir las permutaciones a las $\psi$. En otras palabras, y esto será lo que utilizaremos en este trabajo, distintas combinaciones de $\psi$ serán iguales si lo son en $S_n$, es decir si realizan la misma permutación de los objetos. Gracias al sistema gráfico que utilizaremos para representarlas, será sencillo ver si dos combinaciones de $\psi$ son iguales siguiendo a cada uno de los objetos involucrados de arriba hacia abajo y viendo si en ambos casos terminan en la misma posición. Esta propiedad es un caso particular de la que se tiene en \cite{JS1992}, Corollary 2.6, allí se generaliza este hecho a "`braided tensor categories"', donde el grupo de trenzas ocupa el lugar que para las categorías tensoriales simétricas ocupa el grupo simétrico. La enunciamos a continuación con total precisión:

\begin{\prop}[Coherencia para categorías tensoriales simétricas] \label{cohpsi}

	En cualquier categoría tensorial simétrica, a cada flecha que sea una composición de flechas "`básicas"', donde una flecha es básica si está armada a partir de distintas $\psi$ e identidades con el funtor $\otimes$, se le puede asociar una permutación de $S_n$, donde $n$ es la cantidad de objetos que están siendo multiplicados en el dominio (y el codominio) de la flecha. Esta asignación se realiza mandando a cada flecha "`básica"' al elemento de $S_n$ que realiza las permutaciones que realizan las $\psi$ en los lugares que las realizan, y a la composición de flechas al producto en $S_n$, que corresponde a la composición de las biyecciones.
	
	Entonces, dos de estas flechas son iguales si y solo si tienen asociada la misma permutación de $S_n$.
\end{\prop}

\subsection {Funtores y transformaciones naturales tensoriales}

Definiremos qué significa para un funtor respetar la estructura de categoría tensorial [simétrica]. Informalmente hablando, un funtor será tensorial si respeta el producto tensorial, la asociatividad, [la conmutatividad] y el neutro (salvo isomorfismos). 

\begin{\de} 
	Sean $\Cat$ y $\Cat'$ dos categorías tensoriales [simétricas]. Una 3-upla $(F,f,s)$ es un funtor tensorial si 
	
\begin{enumerate}
	\item $F$ es un funtor $F: \Cat \rightarrow \Cat'$ 
	
	\item $FI$ es un neutro para el producto tensorial de $\Cat'$ y $f: I' \stackrel{\cong}{\leftrightarrow} FI$ es un isomorfismo.
	
	\item $s: \otimes \circ (F \times F) \stackrel{\cong}{\Rightarrow} F \circ \otimes$ es un isomorfismo natural. ($s_{X,Y} : F(X) \otimes F(Y) \stackrel{\cong}{\rightarrow} F(X \otimes Y)$ dice que F respeta el producto tensorial salvo isomorfismo)
	
	\item los siguientes 3[+1] diagramas conmutan.
	
\end{enumerate}
	
\end{\de}


De forma similar a la sección anterior, surgen distintas formas de ver con estos isomorfismos ($s$ y $f$) que distintas configuraciones de $F$ y $\otimes$ son isomorfas. Tenemos los siguientes diagramas que pedimos por definición que conmuten (omitiremos en general los subíndices de s):

1. \xymatrix { X \ar[r]^{F} \ar[d]_{(r_X)^{-1}} & FX  \ar[r]^{(-) \otimes I'} & FX \otimes I' \ar[d]^{FX \otimes f} \\
						X \otimes I \ar[r]^{F} & F(X \otimes I) \ar[r]^{s} & FX \otimes FI }
												
2. El diagrama análogo multiplicando por $I$ e $I'$ a izquierda.

3. \xymatrix { FX \otimes (FY \otimes FZ) \ar[r]^{FX \otimes s} \ar[d]_{\phi} & FX \otimes F (Y \otimes Z) \ar[r]^{s} & F (X \otimes (Y \otimes Z)) \ar[d]^{F(\phi)} \\
						(FX \otimes FY) \otimes FZ \ar[r]^{s \otimes FZ} & F(X \otimes Y) \otimes FZ \ar[r]^{s} & F ((X \otimes Y) \otimes Z) }

4. \xymatrix { FX \otimes FY \ar[r]^{s} \ar[d]^{\psi} & F (X \otimes Y) \ar[d]_{F(\psi)} \\
						FY \otimes FX \ar[r]^{s} & F(Y \otimes X) }




Abusaremos la notación cuando digamos que un funtor es tensorial sin explicitar los isomorfismos, y usando siempre las letras $s$ y $f$ para estos isomorfismos. Definiremos ahora qué significa para una transformación natural respetar la estructura tensorial de los funtores. 

\begin{\de}
	Sean $\Cat$ y $\Cat'$ dos categorías tensoriales [simétricas], y $F,G: \Cat \rightarrow \Cat'$ dos funtores tensoriales. Una transformación natural $\theta: F \Rightarrow G$ es tensorial si los siguientes diagramas conmutan:
	
1. \xymatrix@1 {F(X) \otimes F(Y) \ar[r]^{s} \ar[d]_{\theta_X \otimes \theta_Y} & F(X\otimes Y) \ar[d]^{\theta_{X \otimes Y}} \\
					 G(X) \otimes G(Y) \ar[r]^{s} & G(X\otimes Y)  	}						 
 \ \ \ \ \  2. \xymatrix@1 { I' \ar[r]^{f} \ar[dr]_{f} & FI \ar[d]^{\theta_U} \\
						& GI } 
\end{\de}

\begin{\de}
	Un funtor $F: \Cat \rightarrow \Cat'$ tensorial es una equivalencia de categorías tensoriales si es una equivalencia de categorías entre las categorías $\Cat$ y $\Cat'$.
\end{\de}

La definición anterior tiene su motivación en \cite{Saavedra} 1, 4.4, p.70, donde se prueba la siguiente

\begin{\prop}
	Si un funtor $F: \Cat \rightarrow \Cat'$ tensorial es una equivalencia de categorías entre las categorías $\Cat$ y $\Cat'$, entonces existe un funtor tensorial \linebreak $G: \Cat' \rightarrow \Cat$ e isomorfismos naturales tensoriales $\alpha: id_\Cat \Rightarrow GF$ y \linebreak $\beta: FG \Rightarrow id_{\Cat'}$.
\end{\prop}

También se tiene la siguiente propiedad, que no utilizaremos en este trabajo, pero que es importante pues nos dice que bajo ciertas condiciones el conjunto $End(F)$ de los endomorfismos tensoriales de un funtor tensorial es igual al grupo de sus automorfismos. Su demostración se encuentra en \cite{DM}, 1, Proposition 1.13, p.117.

\begin{\prop}
Si dos categorías tensoriales $\Cat$ y $\Cat'$ son rígidas (ver sección \ref{sub:homint}), entonces cualquier transformación natural tensorial entre dos funtores tensoriales $F,G: \Cat \rightarrow \Cat'$ es un isomorfismo tensorial.
\end{\prop}

\subsection{Categorías tensoriales estrictas} \label{sub:estrictas}

\begin{\de}
	Una categoría tensorial $(\Cat,\otimes,\phi,I,l,r)$ es una categoría tensorial estricta si $\phi,l$ y $r$ son identidades, es decir si $X \otimes (Y \otimes Z) = (X \otimes Y) \otimes Z$ y $X \otimes I = X = I \otimes X$.
\end{\de}

Toda categoría tensorial $\Cat$ es (tensorialmente) equivalente a una ca- \linebreak tegoría tensorial estricta. En efecto, definimos $st(\Cat)$ la categoría tensorial estricta que tiene como objetos a las "`palabras"' cuyas letras son los objetos de $\Cat$ (incluso la palabra vacía, que es $I$) y como producto tensorial a la concatenación. Consideremos la asignación realización $r: st(\Cat) \rightarrow \Cat$, definida en los objetos de $st(\Cat)$ como $r(C_1...C_n) = C_1 \otimes ... \otimes C_n$, donde se realizan los productos tensoriales de izquierda a derecha (es indiferente en qué orden se realizan pero es necesario prefijar uno). Luego podemos definir las flechas en $st(\Cat)$ entre dos palabras $P$ y $P'$ como las flechas entre $r(P)$ y $r(P')$, de forma tal que $r$ resulta trivialmente un funtor tensorial, y podemos definir el funtor tensorial inclusión $i: \Cat \rightarrow st(\Cat)$, $i(C)=C$. La composición $ri$ es igual a $id_\Cat$, mientras que $ir(C_1...C_n) = C_1 \otimes ... \otimes C_n$ visto como palabra de una sola letra. $\theta_{C_1...C_n}: C_1...C_n \rightarrow C_1 \otimes ... \otimes C_n$, la flecha correspondiente a la identidad (pues $r(C_1...C_n) = r(C_1 \otimes ... \otimes C_n)$), es un isomorfismo natural $\theta: id_{st(\Cat)} \Rightarrow ir$.

\pagebreak
\section {Dualidad} \label{sec:Dualidad}

\subsection {Objeto dual} \label{sub:objdual}

	Recordemos que, en $Vec_K^{<\infty}$, para cada objeto $V$ tenemos un objeto $V^{*}$, su dual, que es el espacio vectorial $Hom(V,K)$, y para cualquier base $\{v_1,...,v_n\}$ de $V$ tenemos su base dual, que es una base $\{v_1^*,...,v_n^*\}$ de $V^*$ tal que $v_i^* (v_j) = \delta_{ij}$. Es un ejercicio de álgebra lineal que, para cualquier $v \in V$ y para cualquier $\phi \in V^*$, valen las igualdades 
	\begin{equation} \label{dual1}
		v = v_1^*(v) \cdot v_1 + ... + v_n^*(v) \cdot v_n
	\end {equation}
	\begin{equation} \label{dual2}
		\phi = \phi(v_1) \cdot v_1^* + ... + \phi(v_n) \cdot v_n^* 
	\end {equation}

Notando con $\varepsilon : V^* \otimes V \rightarrow K$ a la evaluación $\varepsilon (\phi,v) = \phi(v)$ y con $\eta : K \rightarrow V \otimes V^*$ a la función lineal definida según $\eta(1) = v_1 \otimes v_1^* + ... + v_n \otimes v_n^*$ para alguna base ${v_1,...,v_n}$ de $V$, obtenemos que las igualdades \eqref{dual1} y \eqref{dual2} se pueden expresar mediante los diagramas

\xymatrix { (V\otimes V^*) \otimes V \ar[r]^{\cong} & V \otimes (V^* \otimes V) \ar[d]^{V \otimes \varepsilon} & V^* \otimes (V \otimes V^*) \ar[r]^{\cong} & (V^*  \otimes V) \otimes V^* \ar[d]^{\varepsilon \otimes V^*} \\
						K \otimes V \ar[u]^{\eta \otimes V} & V \otimes K \ar[d]^{\cong} & V^* \otimes K \ar[u]^{V^* \otimes \eta} & K \otimes V^* \ar[d]^{\cong} \\
						V \ar[u]^{\cong} \ar[r]^{id} & V & V^* \ar[u]^{\cong} \ar[r]^{id} & V^* } 

que, omitiendo los isomorfismos, nos dan las igualdades triangulares

	1. \xymatrix { & V \otimes V^* \otimes V \ar[dr]^{V \otimes \varepsilon} \\
						 V \ar[ur]^{\eta \otimes V} \ar[rr]^{id} & & V }
	\ \ \ \ 2. \xymatrix { & V^* \otimes V \otimes V^* \ar[dr]^{\varepsilon \otimes V^*} \\
						 V^* \ar[ur]^{V^* \otimes \eta} \ar[rr]^{id} & & V^* }
\\

Por último notemos que tenemos un isomorfismo entre $Hom(V,W)$ y $W \otimes V^*$ que, si $B=\{v_1,...,v_n\}$ es una base de $V$ y $B'=\{w_1,...,w_n\}$ es una base de $W$, está definido en bases por
\begin{equation} \label{isovec}
Hom(V,W) \stackrel{\cong}{\rightarrow} W \otimes V^*
\end{equation}
$${v_i}^*(-) \cdot w_j \mapsto w_j \otimes {v_i}^*$$

Cada morfismo $f$ de $Hom(V,W)$ tiene asociada una matriz $T=M_{B,B'}(f)$ que actúa multiplicando a izquierda $T \cdot (v)_B = (f(v))_{B'}$. Como \linebreak $f=\sum T_{ji} \cdot {v_i}^* \cdot w_j$, le corresponde $\sum T_{ji} \cdot w_j \otimes {v_i}^*$. Notemos que esta escritura se corresponde con la forma que tenemos de multiplicar la matriz $T$ a derecha y a izquierda respectivamente por las coordenadas en forma vertical de un vector de $V$ y las coordenadas en forma horizontal de un vector de $W^*$. Observemos finalmente que, vía el isomorfismo $V \otimes V^* \cong Hom(V,V)$, $\eta(1)$ corresponde a la $id_V$.

Todo lo anterior motiva las siguientes definiciones:

\begin{\de} \label{dual}
	Sean $\Cat$ una categoría tensorial, y $X$ e $Y$ objetos de $\Cat$. Diremos que $Y$ es dual a izquierda (también llamado adjunto en \cite{JS}) de $X$ y que $X$ es dual a derecha de $Y$, y notaremos $Y \dashv X$, si existen flechas $Y \otimes X \stackrel{\varepsilon}{\rightarrow} I$ y \linebreak $I \stackrel{\eta}{\rightarrow} X \otimes Y$ tales que los siguientes triángulos (donde se omiten 3 isomorfismos de la estructura de categoría tensorial en cada uno) conmutan.
	
	1. \xymatrix { & X \otimes Y \otimes X \ar[dr]^{X \otimes \varepsilon} \\
						 X \ar[ur]^{\eta \otimes X} \ar[rr]^{id} & & X }
	\ \ \ \ 2. \xymatrix { & Y \otimes X \otimes Y \ar[dr]^{\varepsilon \otimes Y} \\
						 Y \ar[ur]^{Y \otimes \eta} \ar[rr]^{id} & & Y }
	
	Si todo objeto de $\Cat$ tiene un dual a izquierda (derecha), diremos que $\Cat$ tiene una dualidad a izquierda (derecha).
\end {\de}

Gráficamente, las ecuaciones triangulares nos dicen

\begin{equation} \label{ectriang}
\xymatrix@C=-0.3pc @R=0.5pc{ 			& \ar@{-}[ldd] \ar@{-}[rdd] \ar@{}[dd]|{\eta} & & & X \ar@2{-}[dd] 					& & &    						& & & & & &  Y \ar@2{-}[dd] & & & \ar@{-}[ldd] \ar@{-}[rdd] \ar@{}[dd]|{\eta} & 						& & &  \\
											& & & & & & & X \ar@2{-}[dd] & & & & & & & & & & & & &  Y \ar@2{-}[dd] \\
											X \ar@2{-}[dd] & & Y \ar@{-}[rdd] & \ar@{}[dd]|{\varepsilon} & X \ar@{-}[ldd] & \textcolor{white}{XX} = \textcolor{white}{XX} & &  & & & \textcolor{white}{XXX} \hbox{y} \textcolor{white}{XXX} & & & Y \ar@{-}[ddr] & \ar@{}[dd]|{\varepsilon} & X \ar@{-}[ldd] & & Y \ar@2{-}[dd] & \textcolor{white}{XX} = \textcolor{white}{XX} & &  \\
											& & & & & & & X & & & & & & & & & & & & &  Y \\
											X & & & & 																															& & &    							& & & & & & 		& & & & Y	}
\end{equation}

Abusaremos la notación llamando siempre $\eta$ y $\varepsilon$ a estas flechas, aunque correspondan a dualidades de objetos diferentes, pues su dominio y codominio las distinguirán. También usaremos, como es habitual, estas mismas letras para la unidad y counidad de la mayoría de las adjunciones, pero esto tampoco debería llevar a confusiones pues estas son transformaciones naturales. Las flechas $\eta_X$, o $\varepsilon_X$, entonces, corresponderán siempre a estas t.n. y no a las flechas de la dualidad.

Para toda propiedad que probemos para los duales a izquierda tendremos la análoga para duales a derecha. No explicitaremos esto en todas las propiedades, pero es necesario tenerlo presente pues usaremos más adelante también duales a derecha. 

Veremos a continuación la unicidad del dual a izquierda de $X$ salvo por un isomorfismo canónico. Luego no habrá problema en notar con $X^\lor$ al dual a izquierda de $X$, y con $X^\wedge$ a su dual a derecha. 

\begin{\prop} \label{unicodual}
	Sean $\Cat$ una categoría tensorial, y $X$, $Y$ e $Y'$ objetos de $\Cat$ tales que $Y$ e $Y'$ son duales a izquierda de $X$, con flechas $\varepsilon$, $\eta$, $\varepsilon'$ y $\eta'$. Entonces \linebreak $Y' \stackrel{Y' \otimes \eta}{\longrightarrow} Y' \otimes X \otimes Y \stackrel{\varepsilon' \otimes Y}{\longrightarrow} Y$ es un isomorfismo con inversa \linebreak $Y \stackrel{Y \otimes \eta'}{\longrightarrow} Y \otimes X \otimes Y' \stackrel{\varepsilon \otimes Y'}{\longrightarrow} Y'$.
\end{\prop}

\begin{proof}
Tenemos que verificar que $$Y' \stackrel{Y'\otimes\eta}{\longrightarrow} Y' \otimes X \otimes Y \stackrel{\varepsilon'\otimes Y}{\longrightarrow} Y \stackrel{Y \otimes \eta'}{\longrightarrow} Y \otimes X \otimes Y' \stackrel{\varepsilon \otimes Y'}{\longrightarrow} Y'$$ es la identidad de $Y'$ (la otra composición es análoga). La forma más clara de ver esto es traducir esta composición de flechas al cálculo de ascensores

\ \ \ \ \xymatrix@C=-0.3pc @R=0.5pc {  Y' \ar@2{-}[dd] & & & \ar@{-}[ldd] \ar@{-}[rdd] \ar@{}[dd]|{\eta} & & & & & 						   & & Y' \ar@2{-}[dd] & & & & & & & \ar@{-}[ldd] \ar@{-}[rdd] \ar@{}[dd]|{\eta'} & 																		 & & & & & &   & &  \\
											\\
					 					 Y' \ar@{-}[ddr] & \ar@{}[dd]|{\varepsilon'} & X \ar@{-}[ldd] & & Y \ar@2{-}[dd] & & & &   &  & Y' \ar@2{-}[dd] & & & \ar@{-}[ldd] \ar@{-}[rdd] \ar@{}[dd]|{\eta} & & & X \ar@2{-}[dd] & & Y' \ar@2{-}[dd]  &  & Y' \ar@2{-}[dd] & & & \ar@{-}[ldd] \ar@{-}[rdd] \ar@{}[dd]|{\eta'} &   & &  \\
					 					 & & &  & & & & & 						   & & & & & & & & & & 			 & & & & & &  & & Y' \ar@2{-}[dd] \\
					 					 & & & & Y \ar@2{-}[dd] & & & \ar@{-}[ldd] \ar@{-}[rdd] \ar@{}[dd]|{\eta'} & & \textcolor{white}{X} = \textcolor{white}{X} & Y' \ar@2{-}[dd] & & X \ar@2{-}[dd] & & Y \ar@{-}[rdd] & \ar@{}[dd]|{\varepsilon} & X \ar@{-}[ldd] & & Y' \ar@2{-}[dd]   & \textcolor{white}{X} = \textcolor{white}{X} & Y' \ar@{-}[ddr] & \ar@{}[dd]|{\varepsilon'} & X \ar@{-}[ldd] & & Y' \ar@2{-}[dd] & \textcolor{white}{X} = \textcolor{white}{X} & \\
					 					 & & &  & & & & & 						   & & & & & & & & & & 			 & & & & & &  & & Y' \\
					 					 & & & & Y \ar@{-}[rdd] & \ar@{}[dd]|{\varepsilon} & X \ar@{-}[ldd] & & Y' \ar@2{-}[dd]    & & Y' \ar@{-}[ddr] & \ar@{}[dd]|{\varepsilon'} & X \ar@{-}[ldd] & & & & & & Y' \ar@2{-}[dd] & & & & & & Y' \\						 
					 					 \\
					 					 & & & & & & & & Y'																																		 & & & & & & & & & & Y' }		
					 					 
La primer igualdad se justifica por el diagrama \eqref{ascensor}, que nos permite subir el $\eta'$ y bajar el $\varepsilon'$ de manera tal de juntar a $\eta$ y $\varepsilon$ en el centro del diagrama. Luego, por las igualdades triangulares expuestas en forma de diagrama en \eqref{ectriang}, se justifican la segunda y tercer igualdad.
\end{proof}

Podemos realizar entonces la siguiente definición de notación.

\begin{\de}
	Sea $\Cat$ una categoría tensorial, y $X$ un objeto de $\Cat$. Si existen, notamos con $X^\lor$ al dual a izquierda de $X$ y con $X^\wedge$ al dual a derecha de $X$.
\end{\de}

La siguiente propiedad es corolario inmediato de la anterior y del hecho de que por definición $X$ es dual a derecha de $X^\lor$.

\begin{\prop} \label{dualdosveces}
  Sean $\Cat$ una categoría tensorial, y $X$ un objeto de $\Cat$ tal que existe su dual $X^\lor$. Entonces $X = {X^\lor}^\wedge$. Análogamente $X = {X^\wedge}^\lor$ si existe $X^\wedge$.
\end{\prop}

\begin{\prop} \label{dualXporY}
	Sean $\Cat$ una categoría tensorial, y $X$ e $Y$ objetos de $\Cat$ tales que existen sus duales $X^\lor $ e $Y^\lor $. Entonces $(X \otimes Y)^\lor  = Y^\lor  \otimes X^\lor $
\end{\prop}

\begin{proof}
	Veremos que $Y^\lor \otimes X^\lor$ es un dual a izquierda de $X \otimes Y$:
	
	Tenemos $Y^\lor \otimes X^\lor \otimes X \otimes Y \stackrel{Y^\lor \otimes \varepsilon \otimes Y}{\longrightarrow} Y^\lor \otimes Y \stackrel{\varepsilon}{\rightarrow} I$ y \linebreak $I \stackrel{\eta}{\rightarrow} X \otimes X^\lor \stackrel{X \otimes \eta \otimes X^\lor}{\longrightarrow} X \otimes Y \otimes Y^\lor \otimes Y$. Dejamos por verificar las propiedades triangulares.
\end{proof}

Notemos la similitud que existe entre la definición de objeto dual y la de funtores adjuntos. La siguiente proposición relaciona ambos conceptos.

\begin{\prop} \label{CadjC*}
	Sean $\Cat$ una categoría tensorial, y $C$ un objeto de $\Cat$ tal que existe su dual a izquierda $C^\lor $. Entonces se tiene la adjunción \xymatrix@1 { {\Cat} \ar@/^/[r]^{(-)\otimes C}_{\bot} &  {\Cat} \ar@/^/[l]^{(-)\otimes C^\lor } }
\end{\prop}

\begin{proof}
	A partir de las flechas $I \stackrel{\eta}{\rightarrow} C \otimes C^\lor$ y $C^\lor  \otimes C \stackrel{\varepsilon}{\rightarrow} I$ construimos la unidad $\eta: id_{\Cat} \Rightarrow (-) \otimes C \otimes C^\lor $ y la counidad $\varepsilon: (-) \otimes C^\lor  \otimes C \Rightarrow id_{\Cat}$ de la adjunción de la siguiente manera: 
	
	$\eta_X : X \stackrel{X \otimes \eta}{\longrightarrow} X \otimes C \otimes C^\lor $, \ \ \ \ $\varepsilon_X : X \otimes C^\lor  \otimes C \stackrel{X \otimes \varepsilon}{\longrightarrow} X$
	
	Veremos que una de las igualdades triangulares de la dualidad implican una de las de la adjunción, y la otra será análoga. 
	
	En efecto, debemos verificar que $$X \otimes C \stackrel{X \otimes \eta \otimes C}{\longrightarrow} X \otimes C \otimes C^\lor  \otimes C \stackrel{X \otimes C \otimes \varepsilon}{\longrightarrow} X \otimes C$$ es la identidad de $X \otimes C$, pero esto es así pues es la primer igualdad triangular de la dualidad $C^\lor \dashv C$ multiplicada a izquierda por $X$.
\end{proof}

\begin{remark}
	Es pertinente observar para evitar confusiones que si bien por notación tenemos $C^\lor \dashv C \dashv C^\wedge$, el dual a izquierda da el adjunto a derecha como está enunciado en la propiedad anterior, $(-) \otimes C \dashv (-) \otimes C^\lor$. Análogamente se tiene $C \otimes (-)  \dashv C^\wedge \otimes (-)$. 
\end{remark}

Duales a derecha y a izquierda en general no coinciden, pero sí lo hacen en las categorías tensoriales simétricas. Para ver esto probaremos primero el siguiente lema, que también utilizaremos más adelante.

\begin{lemma} \label{etapsi}
	Sean $\Cat$ una categoría tensorial simétrica, y $C$ un objeto que tiene un dual a izquierda $C^\lor$. Entonces para todo objeto $Y$ el siguiente cuadrado conmuta
	
\ \ \ \ \ \ \ \ \ \ \ \ \ \ \ \ \ \ \ \ \ \ \ \	\xymatrix{ Y \ar[r]^{Y \otimes \eta} \ar[d]_{\eta \otimes Y} & Y \otimes C \otimes C^\lor \ar[d]^{\psi \otimes C^\lor} \\
						 C \otimes C^\lor \otimes Y \ar[r]^{C \otimes \psi} & C \otimes Y \otimes C^\lor }

Con demostración análoga conmuta el cuadrado

\ \ \ \ \ \ \ \ \ \ \ \ \ \ \ \ \ \ \ \ \ \ \ \	\xymatrix{ C^\lor \otimes Y \otimes C \ar[r]^{C^\lor \otimes \psi} \ar[d]_{\psi \otimes C} & C^\lor \otimes C \otimes Y \ar[d]^{\varepsilon \otimes Y} \\
						 Y \otimes C^\lor \otimes C \ar[r]^{Y \otimes \varepsilon} & Y  }

En forma gráfica:

\ \ \ \ \ \ \ \ \ \ \ \ \ \ \ \ \ \ \ \ \ \ \xymatrix@C=-0.3pc {Y \ar@2{-}[d] & & & \ar@{-}[ld] \ar@{-}[rd] \ar@{}[d]|{\eta} &  					 & & & 				& \ar@{-}[ld] \ar@{-}[rd] \ar@{}[d]|{\eta} & & & Y \ar@2{-}[d] \\
									Y \ar@2{-}[rrd]|{\psi} & & C \ar@2{-}[lld]|{\psi} & & C^\lor \ar@2{-}[d]    & \textcolor{white}{XX} = \textcolor{white}{XX} & & 				C \ar@2{-}[d] & & C^\lor \ar@2{-}[rrd]|{\psi} & & Y \ar@2{-}[lld]|{\psi}\\
									C & & Y & & C^\lor 																											 & & &	C & & Y & & C^\lor }

y uno análogo para $\varepsilon$:

\ \ \ \ \ \ \ \ \ \ \ \ \ \ \ \ \ \ \ \ \ \ \xymatrix@C=-0.3pc {C^\lor \ar@2{-}[rrd]|{\psi} & & Y \ar@2{-}[lld]|{\psi} & & C \ar@2{-}[d]    & & & 				C^\lor \ar@2{-}[d] & & Y \ar@2{-}[rrd]|{\psi} & & C \ar@2{-}[lld]|{\psi} \\
Y \ar@2{-}[d] & & C^\lor \ar@{-}[rd] & \ar@{}[d]|{\varepsilon} & C \ar@{-}[ld]  					 & \textcolor{white}{XX} = \textcolor{white}{XX} & & 				C^\lor \ar@{-}[rd] & \ar@{}[d]|{\varepsilon} & C \ar@{-}[ld] & & Y \ar@2{-}[d] \\
Y & & & &  																											 & & & & & & & Y }
\end{lemma}

\begin{proof}

	La naturalidad de $\psi$ aplicada a las flechas $f=id_Y$, $f'=\eta$ (ver \ref{psinat}), junto con el hecho de que $\psi_{Y,I} = id_Y$, nos dice que

\ \ \ \ \ \ \ \ \ \ \ \ \ \ \ \ \ \ \ \ \ \ \xymatrix@C=-0.3pc { Y \ar@2{-}[d] & & & \ar@{-}[ld] \ar@{-}[rd] \ar@{}[d]|{\eta} &  			& & & 							& & & & Y \ar@2{-}[d]			\\
										Y \ar@2{-}[rrrrd]|{\psi} & & C \ar@2{-}[lld] & & C^\lor \ar@2{-}[lld] & \textcolor{white}{XX} = \textcolor{white}{XX} & & 	& \ar@{-}[ld] \ar@{-}[rd] \ar@{}[d]|{\eta} & & & Y \ar@2{-}[d]			\\
										C & & C^\lor & & Y 																										& & &	 							C & & C^\lor & & Y			}

	Componiendo con el isomorfismo $id_C \otimes \psi$ obtenemos

\ \ \ \ \ \ \ \ \ \ \ \ \ \ \ \ \ \ \ \ \ \ \xymatrix@C=-0.3pc { Y \ar@2{-}[d] & & & \ar@{-}[ld] \ar@{-}[rd] \ar@{}[d]|{\eta} &  			& & & 							& & & & Y \ar@2{-}[d]			\\
										Y \ar@2{-}[rrrrd]|{\psi} & & C \ar@2{-}[lld] & & C^\lor \ar@2{-}[lld] & \textcolor{white}{XX} = \textcolor{white}{XX} & & 	& \ar@{-}[ld] \ar@{-}[rd] \ar@{}[d]|{\eta} & & & Y \ar@2{-}[d]			\\
										C \ar@2{-}[d] & & C^\lor \ar@2{-}[rrd]|{\psi} & & Y \ar@2{-}[lld]|{\psi}					& & &	 		C \ar@2{-}[d] & & C^\lor \ar@2{-}[rrd]|{\psi} & & Y	\ar@2{-}[lld]|{\psi} 		\\
										C & & Y & & C^\lor									& & &				C & & Y & & C^\lor
										}

	que, por la coherencia para las $\psi$ como está enunciada en la propiedad \ref{cohpsi}, resulta en la igualdad deseada.
\end{proof}

\begin{\prop} \label{DualEnSimetrica}
	Sean $\Cat$ una categoría tensorial simétrica, y $C$ un objeto que tiene un dual a izquierda $C^\lor$. Entonces $C^\lor$ es también dual a derecha de $C$, con $\eta = \psi \circ \eta$ y $\varepsilon = \varepsilon \circ \psi$. En particular, por la propiedad \ref{dualdosveces}, ${C^\lor}^\lor = C$.
\end{\prop}

\begin{proof}

	Verifiquemos una de las igualdades triangulares. Debemos ver que el primero de los diagramas siguientes se colapsa a la identidad de $C^\lor$.
	
\ \ \ \ \ \ \ \ \ \ \ \ \ \ \ \ \ \ \ \ \ \ \xymatrix@C=-0.3pc{ & \ar@{-}[ld] \ar@{-}[rd] \ar@{}[d]|{\eta} & & & C^\lor \ar@2{-}[d] & & &   & \ar@{-}[ld] \ar@{-}[rd] \ar@{}[d]|{\eta} & & & C^\lor \ar@2{-}[d]\\
								 C \ar@2{-}[rrd]|{\psi} & & C^\lor \ar@2{-}[lld]|{\psi} & & C^\lor \ar@2{-}[d] & & & C \ar@2{-}[d] & & C^\lor \ar@2{-}[rrd]|{\psi} & & C^\lor \ar@2{-}[lld]|{\psi} \\	
								 C^\lor \ar@2{-}[d] & & C \ar@2{-}[rrd]|{\psi} & & C^\lor \ar@2{-}[lld]|{\psi} & \textcolor{white}{XX} = \textcolor{white}{XX} & &  C \ar@2{-}[rrrrd]|{\psi} & & C^\lor \ar@2{-}[lld] & & C^\lor \ar@2{-}[lld] \\							 
								 C^\lor \ar@2{-}[d] & & C^\lor \ar@{-}[rd] & \ar@{}[d]|{\varepsilon} & C \ar@{-}[ld] & & &  C^\lor \ar@2{-}[rrd]|{\psi} & & C^\lor \ar@2{-}[lld]|{\psi} & & C \ar@2{-}[d] \\							 
								 C^\lor & & & & & & & C^\lor \ar@2{-}[d] & & C^\lor \ar@{-}[rd] & \ar@{}[d]|{\varepsilon} & C \ar@{-}[ld] \\
								 & & & & & & & 		C^\lor & & & & & & & &  } 

	 La igualdad la obtenemos gracias a la coherencia para las $\psi$ expresada en la propiedad \ref{cohpsi}. Ahora aplicando el lema \ref{etapsi} nos queda el diagrama
	
\ \ \ \ \ \ \ \ \ \xymatrix@C=-0.3pc @R=0.5pc{ C^\lor \ar@2{-}[dd] & & & \ar@{-}[ldd] \ar@{-}[rdd] \ar@{}[dd]|{\eta} & 		& & & 		&&&&    & & & 		 \\
									\\
								 C^\lor \ar@2{-}[rrdd]|{\psi} & & C \ar@2{-}[lldd]|{\psi} & & C^\lor \ar@2{-}[dd] &  & &  &&&&	&  & &  \\	
									& & & &   & & &   C^\lor \ar@2{-}[dd] & & & \ar@{-}[ldd] \ar@{-}[rdd] \ar@{}[dd]|{\eta} & \\								 
								 C \ar@2{-}[rrrrdd]|{\psi} & & C^\lor \ar@2{-}[lldd] & & C^\lor \ar@2{-}[lldd] & & & &&&& & & & C^\lor \ar@2{-}[dd] \\							 
									& & & &   & \textcolor{white}{XX} = \textcolor{white}{XX} & & C^\lor \ar@{-}[rdd] & \ar@{}[dd]|{\varepsilon} & C \ar@{-}[ldd] & \ \ & C^\lor \ar@2{-}[dd] & \textcolor{white}{XX} = \textcolor{white}{XX} & &  \\								 
								 C^\lor \ar@2{-}[dd] & & C^\lor \ar@2{-}[rrdd]|{\psi} & & C \ar@2{-}[lldd]|{\psi} & & & & & & & & & & C^\lor \\							 
									& & & &   & & & & & & & C^\lor \\								
								 C^\lor \ar@{-}[rdd] & \ar@{}[dd]|{\varepsilon} & C \ar@{-}[ldd] & & C^\lor \ar@2{-}[dd] & & & & & & & \\							 
									\\								 
								 & & & & C^\lor  }
								 
	que como se ve arriba se colapsa a la identidad de $C^\lor$. La primer igualdad es válida nuevamente por las propiedades de coherencia de las $\psi$, y la segunda es la segunda igualdad triangular de la dualidad $C^\lor \dashv C$.
\end{proof}

\subsection{Dualidad global} \label{subsec:DualidadGlobal}

Veremos ahora que, si todo objeto de $\Cat$ tiene un dual a izquierda, esta asignación resulta funtorial. Recordemos primero, a modo de motivación, la construcción en $Vec_K$ de la transformación lineal adjunta.

Dada $T: V \rightarrow W$ una t. lineal, definimos $T^* : W^* \rightarrow V^*$ según $T^*(\phi)(v) = \phi(T(v))$. Sea $\{v_1,...,v_n\}$ una base de $V$, y $\{v_1^*,...,v_n^*\}$ su base dual. Cualquier $v \in V$ se escribe $v = v_1^*(v) \cdot v_1 + ... + v_n^*(v) \cdot v_n$, luego $\phi(T(V)) = \phi(T(v_1)) \cdot v_1^*(v) +...+ \phi(T(v_n)) \cdot v_n^*(v)$. Es decir, 
$$T^*(\phi) = \phi(T(v_1)) \cdot v_1^* +...+ \phi(T(v_n)) \cdot v_n^*$$
 Esto muestra que podemos armar $T^*$ como la composición de las siguientes flechas:
\[	 W^* \stackrel{W^* \otimes \eta}{\longrightarrow} W^* \otimes V \otimes V^* \stackrel{W^* \otimes T \otimes V^*}{\longrightarrow} W^* \otimes W \otimes V^* \stackrel{\varepsilon \otimes V^*}{\longrightarrow} V^* \]
\[	\phi \mapsto \phi \otimes (v_1 \otimes v_1^* +...+ v_n \otimes v_n^*) \mapsto \phi \otimes (T(v_1) \otimes v_1^* +...+ T(v_n) \otimes v_n^*) \mapsto \]
\[ \ \ \ \ \ \ \ \ \ \ \ \ \ \ \ \ \ \ \ \ \ \ \ \ \ \ \ \ \ \ \ \ \ \ \ \ \ \ \ \ \ \ \ \ \ \ \ \ \mapsto \phi(T(v_1)) \cdot v_1^* +...+ \phi(T(v_n)) \cdot v_n^* \]

\begin{\prop} 
Sea $\Cat$ una categoría tensorial. Si todo objeto $X$ de $\Cat$ tiene un dual a izquierda $X^\lor $, se tiene un funtor $F: \Cat \rightarrow \Cat$ contravariante tal que $F(X)=X^\lor$ (y análogamente si todo objeto tiene un dual a derecha tenemos otro funtor contravariante $G(X)=X^\wedge$)
\end{\prop}

\begin{proof}
Mostramos la funtorialidad de $F$ (la de $G$ es análoga). Dada $f : Y \rightarrow X$ en $\Cat$, debo definir $F (f) : X^\lor  \rightarrow Y^\lor $.
	
	Como si estuviéramos en $Vec_K$, definimos $F(f)$ como la composición
	
	\label{adjuntaabstracta}
	\[ 	X^\lor  \stackrel{X^\lor  \otimes \eta}{\longrightarrow} X^\lor  \otimes Y \otimes Y^\lor  \stackrel{X^\lor  \otimes f \otimes Y^\lor }{\longrightarrow} X^\lor  \otimes X \otimes Y^\lor  \stackrel{\varepsilon \otimes Y^\lor }{\longrightarrow} Y^\lor   \]
	
	Debemos verificar ahora que, si tenemos también $g : Z \rightarrow Y$ en $\Cat$, entonces $F(f \circ g) = F(g) \circ F(f)$. $F(f \circ g)$ es la composición

		 \label{adjuntaabstracta2}
		\[ 	X^\lor  \stackrel{X^\lor  \otimes \eta}{\longrightarrow} X^\lor  \otimes Z \otimes Z^\lor  \stackrel{X^\lor  \otimes (f \circ g) \otimes Z^\lor }{\longrightarrow} X^\lor  \otimes X \otimes Z^\lor  \stackrel{\varepsilon \otimes Z^\lor }{\longrightarrow} Z^\lor  \]
	
	Mientras que $F(g) \circ F(f)$ es la composición
	
	 \label{adjuntaabstracta3}
	\[	X^\lor  \stackrel{X^\lor  \otimes \eta}{\longrightarrow} X^\lor  \otimes Y \otimes Y^\lor  \stackrel{X^\lor  \otimes f \otimes Y^\lor }{\longrightarrow} X^\lor  \otimes X \otimes Y^\lor  \stackrel{\varepsilon \otimes Y^\lor }{\longrightarrow} Y^\lor  \stackrel{Y^\lor  \otimes \eta}{\longrightarrow} Y^\lor  \otimes Z \otimes Z^\lor  ... \]
		
  \[	\ \ \ \ \ \ \ \ \ \ \ \ \ \ \ \ \ \ \ \ \ \ \ \ \ \ \ \ \ \ \ \ ... Y^\lor  \otimes Z \otimes Z^\lor  \stackrel{Y^\lor  \otimes g \otimes Z^\lor }{\longrightarrow} Y^\lor  \otimes Y \otimes Z^\lor  \stackrel{\varepsilon \otimes Z^\lor }{\longrightarrow} Z^\lor  	\]

Estas dos composiciones son iguales por la primera igualdad triangular de la dualidad de $Y$. Nuevamente utilizaremos el cálculo de ascensores, como en la propiedad \ref{unicodual}. La segunda composición se traduce en

\xymatrix@C=-0.3pc{ X^\lor \ar@2{-}[d] & & & \ar@{-}[rd] \ar@{}[d]|{\eta} \ar@{-}[ld] & & & & & 	& & & 		X^\lor \ar@2{-}[d] & & & & & & & \ar@{-}[ld] \ar@{-}[rd] \ar@{}[d]|{\eta} & 		& & & 		X^\lor  \ar@2{-}[d] & & & \ar@{-}[ld] \ar@{-}[rd] \ar@{}[d]|{\eta} & \\
								 X^\lor \ar@2{-}[d] & & Y {\ar@<4pt>@{-}'+<0pt,-6pt>[d] \ar@<-4pt>@{-}'+<0pt,-6pt>[d]^{f}} & & Y^\lor  \ar@2{-}[d] & & & & 			& & & 			X^\lor  \ar@2{-}[d] & & & & & & Z {\ar@<4pt>@{-}'+<0pt,-6pt>[d] \ar@<-4pt>@{-}'+<0pt,-6pt>[d]^{g}} & & Z^\lor \ar@2{-}[d] 		& & & 		X^\lor  \ar@2{-}[d] & & Z {\ar@<4pt>@{-}'+<0pt,-6pt>[d] \ar@<-4pt>@{-}'+<0pt,-6pt>[d]^{g}} & &  Z^\lor  \ar@2{-}[d] \\	
								 X^\lor \ar@{-}[rd] & \ar@{}[d]|{\varepsilon} & X \ar@{-}[ld] & & Y^\lor  \ar@2{-}[d] & & & & 		& \textcolor{white}{XX} = \textcolor{white}{XX} & & 			X^\lor  \ar@2{-}[d] & & & \ar@{-}[ld] \ar@{}[d]|{\eta} \ar@{-}[rd] & & & Y \ar@2{-}[d] & & Z^\lor  \ar@2{-}[d]		& \textcolor{white}{XX} = \textcolor{white}{XX} & & 		X^\lor  \ar@2{-}[d] & & Y {\ar@<4pt>@{-}'+<0pt,-6pt>[d] \ar@<-4pt>@{-}'+<0pt,-6pt>[d]^{f}} & &  Z^\lor  \ar@2{-}[d] 	\\							 
								 & & & & Y^\lor  \ar@2{-}[d] & & & \ar@{-}[rd] \ar@{-}[ld] \ar@{}[d]|{\eta} & 				& & & 			X^\lor  \ar@2{-}[d] & & Y \ar@2{-}[d] & & Y^\lor  \ar@{-}[rd] & \ar@{}[d]|{\varepsilon} & Y \ar@{-}[ld] & &  Z^\lor  \ar@2{-}[d]		& & & 		X^\lor  \ar@{-}[rd] & \ar@{}[d]|{\varepsilon} & X \ar@{-}[ld] & &  Z^\lor  \ar@2{-}[d]	\\							 
								 & & & & Y^\lor  \ar@2{-}[d] & & Z {\ar@<4pt>@{-}'+<0pt,-6pt>[d] \ar@<-4pt>@{-}'+<0pt,-6pt>[d]^{g}} & & Z^\lor  \ar@2{-}[d] 			& & & 		X^\lor  \ar@2{-}[d]  & & Y {\ar@<4pt>@{-}'+<0pt,-6pt>[d] \ar@<-4pt>@{-}'+<0pt,-6pt>[d]^{f}} & & & & & & Z^\lor  \ar@2{-}[d]			& & & 		  		& & & & Z^\lor	\\							 
								 & & & & Y^\lor  \ar@{-}[rd] & \ar@{}[d]|{\varepsilon} & Y \ar@{-}[ld] & & Z^\lor  \ar@2{-}[d] 				& & & 	X^\lor  \ar@{-}[rd] & \ar@{}[d]|{\varepsilon} & X \ar@{-}[ld] & & & &  &  & Z^\lor \ar@2{-}[d]			& & & 		\\	
								 & & & & & & & & Z^\lor 		& & & 					& & & & & & & & Z^\lor }

	Para la primera igualdad, usando \eqref{ascensor}, juntamos la $\eta$ y la $\varepsilon$ correspondientes a la dualidad $Y^\lor \dashv Y$ en el medio del diagrama corriendo hacia arriba o hacia abajo el resto de las flechas. Luego la segunda igualdad se justifica por una igualdad triangular de esta dualidad. El diagrama de la derecha resulta exactamente el que corresponde a la primera composición de flechas.
\end{proof}

\begin{\de}
	Dada una flecha $f: X \rightarrow Y$ en $\Cat$, notaremos con $f^\lor: Y^\lor \rightarrow X^\lor$ a $F(f)$ y con $f^\wedge: Y^\wedge \rightarrow X^\wedge$ a $G(f)$
\end{\de}

La siguiente propiedad la usaremos más adelante. 

\begin{\prop} \label{wedgeswitch}
Sea $\Cat$ una categoría tensorial con una dualidad a derecha. Para toda flecha $f: X \rightarrow Y$ el cuadrado 

\ \ \ \ \ \ \ \ \ \ \ \ \xymatrix { X \otimes Y^\wedge \ar[r]^{X \otimes f^\wedge} \ar[d]_{f \otimes Y^\wedge} & X \otimes X^\wedge \ar[d]^{\varepsilon} \\
							Y \otimes Y^\wedge \ar[r]^{\varepsilon} & I }

	conmuta. Análogamente conmuta el cuadrado
	
\ \ \ \ \ \ \ \ \ \ \ \ \xymatrix { I \ar[d]_{\eta}	\ar[r]^{\eta} & X^\wedge \otimes X \ar[d]^{X^\wedge \otimes f} \\
																		Y^\wedge \otimes Y \ar[r]^{f^\wedge \otimes Y} & X^\wedge \otimes Y}
																		
En forma gráfica:

\ \ \ \ \ \ \ \ \ \ \ \ \ \ \ \ \ \ \ \ \ \ \xymatrix@C=-0.3pc {  X {\ar@<4pt>@{-}'+<0pt,-6pt>[d]        
                  \ar@<-4pt>@{-}'+<0pt,-6pt>[d]^{f}} & & Y^\wedge \ar@2{-}[d] 											& & & 			X \ar@2{-}[d] & & Y^\wedge {\ar@<4pt>@{-}'+<0pt,-6pt>[d]    
                  \ar@<-4pt>@{-}'+<0pt,-6pt>[d]^{f^\wedge}}\\
								 Y \ar@{-}[dr] & \ar@{}[d]|{\varepsilon} & Y^\wedge \ar@{-}[ld] & \textcolor{white}{XX} = \textcolor{white}{XX} & & X \ar@{-}[dr] & \ar@{}[d]|{\varepsilon} & X^\wedge \ar@{-}[ld]\\
								 & & 																														& & & 											& & }

y uno análogo para $\eta$:

\ \ \ \ \ \ \ \ \ \ \ \ \ \ \ \ \ \ \ \ \ \ \xymatrix@C=-0.3pc {& \ar@{-}[dr] \ar@{}[d]|{\eta} \ar@{-}[ld] & & & & & \ar@{-}[dr] \ar@{}[d]|{\eta} \ar@{-}[ld] \\
Y^\wedge {\ar@<4pt>@{-}'+<0pt,-6pt>[d] \ar@<-4pt>@{-}'+<0pt,-6pt>[d]^{f^\wedge}} & & Y \ar@2{-}[d] & \textcolor{white}{XX} = \textcolor{white}{XX} & & X^\wedge \ar@2{-}[d] & & X {\ar@<4pt>@{-}'+<0pt,-6pt>[d] \ar@<-4pt>@{-}'+<0pt,-6pt>[d]^{f}} \\
X^\wedge & & Y & & & X^\wedge & & Y }	\end{\prop}

\begin{remark}
	Con las definiciones \ref{def:wedge} y \ref{def:cowedge} de más adelante, esta propiedad equivale a decir que $\eta$ y $\varepsilon$ son di-conos.
\end{remark}

\begin{proof}	
	
	 Por la definición de $f^\wedge$, tenemos que la composición \linebreak $\varepsilon \circ (X \otimes f^\wedge)$ se corresponde con el diagrama
	
 \ \ \ \ \ \ \xymatrix@C=-0.3pc{ X \ar@2{-}[d] & & & \ar@{-}[ld] \ar@{-}[rd] \ar@{}[d]|{\eta} & & & Y^\wedge \ar@2{-}[d] & & & 		X \ar@2{-}[d] & & & \ar@{-}[ld] \ar@{-}[rd] \ar@{}[d]|{\eta} & & & Y^\wedge \ar@2{-}[d] &	& & 		X {\ar@<4pt>@{-}'+<0pt,-6pt>[d] \ar@<-4pt>@{-}'+<0pt,-6pt>[d]^{f}} & & Y^\wedge \ar@2{-}[d]	\\
								 X \ar@2{-}[d] & & X^\wedge \ar@2{-}[d] & & X {\ar@<4pt>@{-}'+<0pt,-6pt>[d] \ar@<-4pt>@{-}'+<0pt,-6pt>[d]^{f}} & & Y^\wedge \ar@2{-}[d] & \textcolor{white}{XX} = \textcolor{white}{XX} & & 		X \ar@{-}[dr] & \ar@{}[d]|{\varepsilon} & X^\wedge \ar@{-}[ld] & & X \ar@2{-}[d] & & Y^\wedge \ar@2{-}[d] & \textcolor{white}{XX} = \textcolor{white}{XX} & & 		Y \ar@{-}[dr] & \ar@{}[d]|{\varepsilon} & Y^\wedge \ar@{-}[ld]	\\	
								 X \ar@2{-}[d] & & X^\wedge \ar@2{-}[d] & & Y \ar@{-}[rd] & \ar@{}[d]|{\varepsilon} & Y^\wedge \ar@{-}[ld] & & & 		 & & & & X \ar@<4pt>@{-}'+<0pt,-6pt>[d] \ar@<-4pt>@{-}'+<0pt,-6pt>[d]^{f} & & Y^\wedge \ar@2{-}[d] & & & 		& & & & & & 	\\							 
								 X \ar@{-}[dr] & \ar@{}[d]|{\varepsilon} & X^\wedge \ar@{-}[ld] & & & &	& & & 	 & & & & Y \ar@{-}[rd] & \ar@{}[d]|{\varepsilon} & Y^\wedge \ar@{-}[ld] & & & \\
								 & & & & & & & & & 			& & & & & & & & &}
		
	La primer igualdad se obtiene subiendo y bajando flechas, y la segunda por la primer igualdad triangular de la dualidad $X^\lor \dashv X$. El diagrama de la derecha es la otra composición del cuadrado.
\end{proof}

Los isomorfismos de la propiedad \ref{dualdosveces} nos dan unidades y counidades para las adjunciones $F \dashv G$ y $G \dashv F$. Luego tenemos la siguiente propiedad.


\begin{\prop} \label{FadjF}
	Sea $\Cat$ una categoría tensorial. Si todo objeto $X$ de $\Cat$ tiene un dual a izquierda y un dual a derecha, $F$ y $G$ constituyen una equivalencia de categorías entre $\Cat$ y $\Cat^{op}$.
\end {\prop}

En el caso en que $\Cat$ es simétrica, esta adjunción es un caso particular del hecho de que el funtor $Hom(-,C)$ es autoadjunto para todo objeto $C$ de $\Cat$ (ver la siguiente sección).

\subsection{Hom internos} \label{sub:homint}

\begin{\de}
	Sea $\Cat$ una categoría tensorial. Si para algún objeto $C$ de $\Cat$ el funtor $\Cat \stackrel{(-) \otimes C}{\longrightarrow} \Cat$ tiene un adjunto a derecha, llamamos $Hom(C,-)$ al adjunto. Si esto ocurre para todos los objetos de $\Cat$, decimos que la categoría $\Cat$ tiene hom internos.
\end{\de}

Expresando esta adjunción como biyección natural entre flechas obte- \linebreak nemos la ley exponencial:
\begin{center}
\underline { $A \otimes C \rightarrow B$ } \\
$A \rightarrow Hom(C,B)$
\end{center}

que tiene como caso particular a la biyección $$[I, Hom(C,B)] = [C,B]$$

Por otro lado, la counidad de la adjunción $Hom(C,B) \otimes B \stackrel{\varepsilon_B}{\rightarrow} B$ es una flecha que podemos identificar con la evaluación.

A partir de la evaluación y de la ley exponencial podemos construir la flecha que podemos identificar con la "`composición"', de la siguiente forma: buscamos una flecha $Hom(Y,Z) \otimes Hom (X,Y) \stackrel{\circ}{\rightarrow} Hom (X,Z)$, y la conseguimos por la ley exponencial a partir de la flecha $$Hom(Y,Z) \otimes Hom (X,Y) \otimes X \stackrel{Hom(Y,Z) \otimes \varepsilon_Y}{\longrightarrow} Hom(Y,Z) \otimes Y \stackrel{\varepsilon_Z}{\rightarrow} Z$$

 Esta flecha se corresponde con la definición usual de la composición con elementos dada por $(f \circ g)(x)=f(g(x))$ pues $f(g(x)) = ev(f,ev(g,x))$.

Que una categoría tensorial $\Cat$ con objeto neutro $I$ tenga hom internos nos da en virtud de la propiedad \ref{CadjC*} un candidato único para ser el dual a izquierda de $C$, el objeto $Hom(C,I)$. Sin embargo una categoría puede tener hom internos pero no una dualidad, por ejemplo la categoría $Vec_K$ tiene hom internos y es simétrica, luego si tuviera una dualidad se debería verificar que $Hom(Hom(V,K),K) = V$, pero este no es el caso si $V$ no tiene dimensión finita.

Por otro lado, por unicidad del funtor adjunto, la proposición \ref{CadjC*} también nos dice lo siguiente

\begin{\prop}
Sea $\Cat$ una categoría tensorial. Si un objeto $C$ tiene un dual a izquierda $C^\lor $, entonces el funtor $Hom(C,-)$ existe y su valor es \linebreak $Hom(C,X) = X \otimes C^\lor$. En particular, $Hom(C,I)=C^\lor $. Globalmente, se obtiene que si $\Cat$ tiene una dualidad a izquierda entonces tiene hom internos.
\end{\prop}

\begin{remark}
La igualdad $Hom(C,X) = X \otimes C^\lor$ es idéntica a lo que vimos en particular para $\Cat = Vec_K^{<\infty}$ en \eqref{isovec}. 
\end{remark}

 En \cite{JS} se usa la definición de objeto dual, mientras que en \cite{DM} se trabaja con una dualidad en categorías con hom internos donde se piden ciertas condiciones extras (ver \cite{DM}, p.112). Tiene sentido entonces preguntarse qué relaciones hay entre ambas definiciones, para lo cual describiremos más explicitamente la definición de Deligne-Milne.
 
 En una categoría tensorial simétrica con hom internos $\Cat$, podemos tomar para cada objeto $C$ su potencial dual $C^*=Hom(C,I)$. Nos preguntamos entonces qué condiciones tiene que satisfacer $\Cat$ para que $C^*$ sea un objeto dual de $C$ en el sentido de la definición \ref{dual}. Deligne-Milne realizan las siguientes definiciones.
 
 Para cada objeto $C$ tenemos una inclusión de $C$ en ${C^*}^*$ que corresponde a $C \otimes C^* \stackrel{\psi}{\rightarrow} C^* \otimes C \stackrel{\varepsilon_I}{\rightarrow} I$ vía la ley exponencial. Si esta inclusión resulta un isomorfismo, decimos que $C$ es reflexivo.
 
 También tenemos morfismos $$Hom(C_1,D_1) \otimes Hom(C_2,D_2) \stackrel{\varphi_{C_1D_1C_2D_2}}{\longrightarrow} Hom (C_1 \otimes C_2, D_1 \otimes D_2)$$ que corresponden a 
 \begin{equation} \label{fCDCD}
   f_{C_1D_1C_2D_2}: Hom(C_1,D_1) \otimes Hom(C_2,D_2) \otimes C_1 \otimes C_2 \stackrel{Hom(C_1,D_1) \otimes \psi \otimes C_2}{\longrightarrow} ... 
 \end{equation}	
 $$... Hom(C_1,D_1) \otimes C_1 \otimes Hom(C_2,D_2) \otimes C_2 \stackrel{\varepsilon_{D_1} \otimes \varepsilon_{D_2}}{\longrightarrow} D_1 \otimes D_2$$
 (notar que cada $\varepsilon$ corresponde a una adjunción de Hom internos distinta) vía la ley exponencial (que también corresponde a una adjunción distinta, a la de $(-) \otimes C_1 \otimes C_2$). 
 
\begin{\de}[Categoría rígida según Deligne-Milne]
 Una categoría tensorial simétrica con hom internos se dice rígida si todos los morfismos $\varphi_{C_1D_1C_2D_2}$ son isomorfismos, y además todos sus objetos son reflexivos.
\end{\de}
 
Lo primero que haremos es ver que si una categoría $\Cat$ simétrica tiene una dualidad (a izquierda), entonces $\Cat$ es rígida.
 
 Ya hemos visto que si una categoría tiene una dualidad a izquierda, entonces la categoría tiene hom internos, y que el valor de este funtor es $Hom(B,C) = C \otimes B^\lor$. Notemos además que en este caso la ley exponencial queda 
 
\begin{equation} \label{leyexpcondual}
\xymatrix { {}  \ar@<-4ex>@{-}[rrr] & A \otimes B \ar[r]^{f} & C & {} \ar@/^2pc/[d]^{\theta}\\
						{} \ar@/^2pc/[u]^{\theta} & A \ar[r]^>>>>{g} & Hom(B,C)=C\otimes B^* & {} }
\end{equation}
$\theta(f)= A \stackrel{A \otimes \eta}{\rightarrow} A \otimes B \otimes B^* \stackrel{f \otimes B^*}{\rightarrow} C \otimes B^*$, $\theta(g)= A \otimes B \stackrel{g \otimes B}{\rightarrow} C \otimes B^* \otimes B \stackrel{C \otimes \varepsilon}{\rightarrow} C$

Notemos entonces que las flechas que queremos ver que son isomorfismos vienen dadas por el diagrama
 
\ \ \ \ \ \ \ \ \ \ \ \ \ \ \ \ \ \ \ \ \ \ \ \xymatrix@C=-0.3pc{ 	 D_1 \ar@2{-}[d] & & C_1^\lor \ar@2{-}[d] & & D_2 \ar@2{-}[d] & & C_2^\lor \ar@2{-}[d] & & & & & \ar@{-}[llld] \ar@{}[d]|{\eta} \ar@{-}[rrrd] \\
 										 D_1 \ar@2{-}[d] & & C_1^\lor \ar@2{-}[d] & & D_2 \ar@2{-}[d] & & C_2^\lor \ar@2{-}[d] & & C_1 \ar@2{-}[d] & & & \ar@{-}[ld] \ar@{}[d]|{\eta} \ar@{-}[rd] & & & C_1^\lor \ar@2{-}[d] \\
 										 D_1 \ar@2{-}[d] & & C_1^\lor \ar@2{-}[d] & & D_2 \ar@2{-}[drr] & & C_2^\lor \ar@2{-}[drr] & & C_1 \ar@2{-}[lllld]|{\psi} & & C_2 \ar@2{-}[d] & & C_2^\lor \ar@2{-}[d] & & C_1^\lor \ar@2{-}[d] \\
 										 D_1 \ar@2{-}[d] & & C_1^\lor \ar@{-}[dr] & \ar@{}[d]|{\varepsilon} & C_1 \ar@{-}[dl] & & D_2 \ar@2{-}[d] & & C_2^\lor \ar@{-}[dr] & \ar@{}[d]|{\varepsilon} & C_2 \ar@{-}[dl] & & C_2^\lor \ar@2{-}[d] & & C_1^\lor \ar@2{-}[d] \\
 										 D_1 & & & & & & D_2 & & & & & & C_2^\lor & & C_1^\lor }

Primero juntamos a $\eta$ y $\varepsilon$ de la dualidad $C_2^\lor \dashv C_2$, y usamos la segunda igualdad triangular de esa dualidad para obtener el diagrama de la izquierda:

 \ \ \ \ \ \ \ \ \ \ \ \ \ \ \ \xymatrix@C=-0.3pc{ D_1 \ar@2{-}[d] & & C_1^\lor \ar@2{-}[d] & & D_2 \ar@2{-}[d] & & C_2^\lor \ar@2{-}[d] & & & \ar@{-}[ld] \ar@{}[d]|{\eta} \ar@{-}[rd] & & & & & D_1 \ar@2{-}[d] & & C_1^\lor \ar@2{-}[d] & & D_2 \ar@2{-}[d] & & C_2^\lor \ar@2{-}[d] & & & \ar@{-}[ld] \ar@{}[d]|{\eta} \ar@{-}[rd] \\
										D_1 \ar@2{-}[d] & & C_1^\lor \ar@2{-}[d] & & D_2 \ar@2{-}[drr] & & C_2^\lor \ar@2{-}[drr] & & C_1 \ar@2{-}[lllld]|{\psi} & & C_1^\lor \ar@2{-}[d] & & & & D_1 \ar@2{-}[d] & & C_1^\lor \ar@2{-}[d] & & D_2 \ar@2{-}[d] & & C_2^\lor \ar@2{-}[drr]|{\psi} & & C_1 \ar@2{-}[lld]|{\psi} & & C_1^\lor \ar@2{-}[d] \\
										D_1 \ar@2{-}[d] & & C_1^\lor \ar@{-}[dr] & \ar@{}[d]|{\varepsilon} & C_1 \ar@{-}[dl] & & D_2 \ar@2{-}[d] & & C_2^\lor \ar@2{-}[d] & & C_1^\lor \ar@2{-}[d] & & \ \ \ \ \hbox{=} \ \ \ \ & & D_1 \ar@2{-}[d] & & C_1^\lor \ar@2{-}[d] & & D_2 \ar@2{-}[drr]|{\psi} & & C_1 \ar@2{-}[dll]|{\psi} & & C_2^\lor \ar@2{-}[d] & & C_1^\lor \ar@2{-}[d] \\
										D_1 & & & & & & D_2 & & C_2^\lor & & C_1^\lor & & & & D_1 \ar@2{-}[d] & & C_1^\lor \ar@{-}[dr] & \ar@{}[d]|{\varepsilon} & C_1 \ar@{-}[dl] & & D_2 \ar@2{-}[d] & & C_2^\lor \ar@2{-}[d] & & C_1^\lor \ar@2{-}[d] \\
										& & & & & & & & & & & & & & D_1 & & & & & & D_2 & & C_2^\lor & & C_1^\lor }

La igualdad es válida por \eqref{hex}. Ahora aplicamos el lema \ref{etapsi} dos veces (a $\eta$, pero podríamos aplicarlo a $\varepsilon$) para obtener

\ \ \ \ \ \ \ \ \ \ \ \ \ \ \ \ \ \ \ \ \ \ 
\xymatrix@C=-0.3pc{ 
							D_1 \ar@2{-}[d] & & C_1^\lor \ar@2{-}[d] & & & \ar@{-}[ld] \ar@{}[d]|{\eta} \ar@{-}[rd] & & & D_2 \ar@2{-}[d] & & C_2^\lor \ar@2{-}[d] & & & 			D_1 \ar@2{-}[d] & & C_1^\lor \ar@2{-}[drr]|{\psi} & & D_2 \ar@2{-}[dll]|{\psi} & & C_2^\lor \ar@2{-}[d]		\\
							D_1 \ar@2{-}[d] & & C_1^\lor \ar@2{-}[d] & & C_1 \ar@2{-}[d] & & C_1^\lor \ar@2{-}[drr]|{\psi} & & D_2 \ar@2{-}[dll]|{\psi} & & C_2^\lor \ar@2{-}[d] &  \ \ \ \ \hbox{=} \ \ \ \ & & 			D_1 \ar@2{-}[d] & & D_2 \ar@2{-}[d] & & C_1^\lor \ar@2{-}[drr]|{\psi} & & C_2^\lor \ar@2{-}[dll]|{\psi} \\
							D_1 \ar@2{-}[d] & & C_1^\lor \ar@2{-}[d] & & C_1 \ar@2{-}[d] & & D_2 \ar@2{-}[d] & & C_1^\lor \ar@2{-}[drr]|{\psi} & & C_2^\lor \ar@2{-}[dll]|{\psi} & & & 			D_1 & & D_2 & & C_2^\lor & & C_1^\lor		\\
							D_1 \ar@2{-}[d] & & C_1^\lor \ar@{-}[dr] & \ar@{}[d]|{\varepsilon} & C_1 \ar@{-}[dl] & & D_2 \ar@2{-}[d] & & C_2^\lor \ar@2{-}[d] & & C_1^\lor \ar@2{-}[d] & & & \\
							D_1 & & & & & & D_2 & & C_2^\lor & & C_1^\lor & & & }

La igualdad se obtiene por la segunda igualdad triangular de la dualidad $C_1^\lor \dashv C_1$, juntando antes a $\eta$ y $\varepsilon$. El diagrama de la derecha es un isomorfismo pues las $\psi$ lo son. Veamos ahora que cualquier objeto $C$ es reflexivo. Ya hemos visto en la propiedad \ref{DualEnSimetrica} que como $\Cat$ es simétrica, ${C^\lor}^\lor = C$ y que las $\eta$ y $\varepsilon$ de esta dualidad son las mismas de la dualidad de $C$ pero compuestas con $\psi$. Luego ver que $C$ es reflexivo es ver que el diagrama

\ \ \ \ \ \ \ \ \ \ \ \ \ \ \ \ \ \ \ \ \ \ \xymatrix@C=-0.3pc{ C \ar@2{-}[d] & & & \ar@{-}[ld] \ar@{-}[rd] \ar@{}[d]|{\eta} & \\
								 C \ar@2{-}[d] & & C \ar@2{-}[rrd]|{\psi} & & C^\lor \ar@2{-}[lld]|{\psi} \\	
								 C \ar@2{-}[rrd]|{\psi} & & C^\lor \ar@2{-}[lld]|{\psi} & & C \ar@2{-}[d] \\							 
								 C^\lor \ar@{-}[rd] & \ar@{}[d]|{\varepsilon} & C \ar@{-}[ld] & & C \ar@2{-}[d] \\							 
								 & & & & C }

es un isomorfismo, pero este diagrama corresponde a la otra igualdad triangular de la dualidad $C \dashv C^\lor$, según fue vista en la propiedad \ref{DualEnSimetrica}, y por lo tanto se colapsa a la identidad de $C$.
 
 Para ver que en una categoría rígida los objetos $C^* = Hom(C,I)$ resultan un dual de $C$ en el sentido de la definición \ref{dual}, planteemos primero la siguiente pregunta: si dado un objeto $C$ tenemos un objeto $D$ que hace que valga la adjunción $(-) \otimes C \dashv (-) \otimes D$, ¿qué condición podemos pedirle a esta adjunción para que $D=C^\lor$?
 
 Notemos que las igualdades triangulares en este caso nos dan
 
 $$ C \stackrel{\eta_I \otimes C}{\longrightarrow} C \otimes D \otimes C \stackrel{\varepsilon_C}{\longrightarrow} C = id_C$$
 y
 $$ D \stackrel{\eta_D}{\longrightarrow} D \otimes C \otimes D \stackrel{\varepsilon_I \otimes Y}{\longrightarrow} D = id_D$$
 
 Luego, para que $\eta = \eta_I$ y $\varepsilon = \varepsilon_I$ nos den una dualidad $D \dashv C$, bastará con pedir que $\varepsilon_C = C \otimes \varepsilon$ y que $\eta_D = D \otimes \eta$.
 
 En el caso de una categoría $\Cat$ rígida, notemos primero que $Hom(I,C)$ es un objeto isomorfo a $C$ para todo $C$, pues $\Cat \stackrel{(-) \otimes I}{\longrightarrow} \Cat$ es una equivalencia de categorías con pseudoinversa la identidad y los adjuntos a derecha son únicos salvo isomorfismo. Luego podemos asumir para simplificar la escritura $Hom(I,C) = C$, y $Hom(I,C) \otimes I \stackrel{\varepsilon_C}{\rightarrow} C$ el isomorfismo de la estructura de categoría tensorial $r_C: C \otimes I \rightarrow C$.
 
 Luego, consideramos el diagrama conmutativo
 
\begin{equation} \label{psiisom}
 \xymatrix@C=-2pc { & Hom(I,D) \otimes I \otimes Hom(C,I) \otimes C \ar[rd]^{\varepsilon_D \otimes \varepsilon_I} \\
 						Hom(I,D) \otimes Hom(C,I) \otimes I \otimes C \ar[ru]^{Hom(I,D) \otimes \psi \otimes C} & & D \otimes I \\
 							D \otimes Hom(C,I) \otimes C \ar@{-}[u]_{\cong} \ar[rr]^{D \otimes \varepsilon_I} & & D \ar@{-}[u]_{\cong} }
\end{equation} 
 
 del que observamos que la composición superior coincide con la definición de $f_{IDCI}$. Notemos además que la $\psi$ que aparece, como está haciendo conmutar a una $I$, puede ser reemplazada por la composición de $r$ y $l^{-1}$, los isomorfismos de la estructura de categoría tensorial que hacen aparecer y desaparecer a la $I$. Por lo tanto no necesitamos para esto que la categoría sea simétrica. Luego, si la categoría es rígida, obtenemos que si $\varphi$ se obtiene de la ley exponencial de la siguiente forma
 
 \begin{center}
\underline { $ D \otimes C^* \otimes C \stackrel{D \otimes \varepsilon_I}{\rightarrow} D$ } \\
$D \otimes C^* \stackrel{\varphi}{\rightarrow} Hom(C,D)$
\end{center}

 entonces es un isomorfismo. Luego obtenemos que \linebreak $Hom(C,-) \cong (-) \otimes C^*$, entonces estamos en la situación recién planteada (en la cual la adjunta de $(-) \otimes C$ viene dada por $(-) \otimes D$, con $D = C^*$) y para ver que $C^*$ es un dual de $C$ bastará con ver que si notamos $\varepsilon = \varepsilon_I$ y $\eta = \eta_I$, entonces para esta adjunción $\varepsilon_C = C \otimes \varepsilon$ y $\eta_{C^*} = C^* \otimes \eta$. Para eso usaremos la naturalidad de la biyección entre las flechas que da la adjunción $(-) \otimes C \dashv Hom(C,-)$:
 
  \ \ \ \ \xymatrix { {}  \ar@<-4ex>@{-}[rrrr] & Hom(C,C) \otimes C \ar[r]^<<<<<<{\varphi^{-1} \otimes C} \ar@/^2pc/[rr]^{\varepsilon_C} & C \otimes C^* \otimes C \ar[r]^{C \otimes \varepsilon} & C & {} \\
						{}  & Hom(C,C) \ar[r]^{\varphi^{-1}} \ar@/_2pc/[rr]_{id} & C \otimes C^* \ar[r]^{\varphi} & Hom(C,C) & {} }

 El diagrama anterior debe leerse de la siguiente forma: $C \otimes \varepsilon$ nos da el isomorfismo $\varphi$ vía la ley exponencial. Además, si componemos abajo con la inversa de $\varphi$, se corresponderá con componer arriba con $\varphi^{-1} \otimes C$. Pero entonces abajo queda la identidad y por lo tanto arriba la counidad $\varepsilon_C$. La conmutatividad del triángulo de arriba es una de las igualdades deseadas.
 
 Análogamente se obtiene la otra igualdad considerando el diagrama 
 
  \xymatrix@C=3.5pc { {\! \! \! \! \! \! \! \! \! \!}  \ar@<-4ex>@{-}[rrr] & C^* \otimes C \ar[r]^>>>>>{C^* \otimes \eta \otimes C} \ar@/^2pc/[rr]^{id} & C^* \otimes Hom(C,C) \otimes C \ar[r]^<<<<<<<<<<{C^* \otimes \varepsilon_C} & C^* \otimes C & {\! \! \! \! \! \! \! \! \! \!} \\
						{}  & C^* \ar[r]^<<<<<<<<<<<<<<{C^* \otimes \eta} \ar@/_2pc/[rr]_{\eta_{C^*}} & C^* \otimes Hom(C,C) \ar[r]^{\varphi} & Hom(C,C^* \otimes C) & {} }
 
 siempre que $\varphi$ resulte un isomorfismo, lo cual se tiene para una categoría rígida considerando el siguiente diagrama similar a \eqref{psiisom}:
 
 \xymatrix@C=-2pc { & Hom(I,C^*) \otimes I \otimes Hom(C,C) \otimes C \ar[rd]^{\varepsilon_C^* \otimes \varepsilon_C} \\
 						Hom(I,C^*) \otimes Hom(C,C) \otimes I \otimes C \ar[ru]^{Hom(I,C^*) \otimes \psi \otimes C} & & C^* \otimes C \\
 							C^* \otimes Hom(C,C) \otimes C \ar@{-}[u]_{\cong} \ar[rr]^{C^* \otimes \varepsilon_C} & & C^* \ar@{-}[u]_{\cong} } 

	 del que observamos que la composición superior coincide con la definición de $f_{IC^*CC}$. La misma observación sobre la innecesariedad de la $\psi$ es válida aquí.

 Observando qué propiedades de la rigidez de una categoría hemos utilizado y cuáles no en esta demostración, notamos que hemos demostrado la siguiente propiedad:
 
 \begin{\prop}
 		Sea $\Cat$ una categoría tensorial (no necesariamente simétrica). Son equivalentes:
\begin{enumerate}
	\item $\Cat$ tiene una dualidad a izquierda.
	\item $\Cat$ tiene hom internos y los morfismos $D \otimes C^* \stackrel{\varphi}{\rightarrow} Hom(C,D)$ y \linebreak $C^* \otimes Hom(C,C) \stackrel{\varphi}{\rightarrow} Hom(C,C^* \otimes C)$, obtenidos respectivamente a partir de $D \otimes C^* \otimes C \stackrel{D \otimes \varepsilon_I}{\longrightarrow} D$ y de $C^* \otimes Hom(C,C) \otimes C \stackrel{C \otimes \varepsilon_C}{\longrightarrow} C^* \otimes C$ por la ley exponencial, son isomorfismos.
\end{enumerate}
	Además, si $\Cat$ es simétrica, se agrega a estas equivalencias
\begin{enumerate}	
	\item[3.] $\Cat$ es rígida.
\end{enumerate}
 \end{\prop} 
 
 De esta propiedad puede observarse que nos dice por un lado que la generalización a objetos duales de Joyal de la noción de rigidez de Deligne-Milne es la correcta para categorías tensoriales no simétricas, y también que el enunciado del item 2 es otra definición posible de rigidez, más general pues no requiere la simetría y tiene hipótesis más débiles, que es equivalente a la definición de Deligne-Milne en el caso simétrico. Las condiciones del item 2 pueden tomarse como una definición de rigidez a izquierda.
 
 También puede estudiarse en categorías tensoriales no simétricas el adjunto a derecha del "`otro"' funtor dado por el producto tensorial: $\Cat \stackrel{C \otimes (-)}{\longrightarrow} \Cat$, y relacionarlo con los duales a derecha. Para estas cuestiones el lector interesado puede consultar \cite{sch}, 2.1, p.329.
 
\begin{remark}
	En general, en una categoría tensorial que no sea simétrica, no tiene por qué ser cierto que $Hom(X^\lor,Y) = Y \otimes X$, pues no es cierto que ${X^\lor}^\lor = X$. Sin embargo, el dual a derecha sí tiene esta propiedad, pues como si $X^\wedge$ existe vale que ${X^\wedge}^\lor = X$, se deduce que $Hom(X^\wedge,Y) = Y \otimes X$. Este es el motivo por el cual usaremos el dual a derecha para definir $Nat^\lor(F,G)$ en la sección $\ref{sec:NatPredual}$, pues si lo definiéramos reemplazando el dual a derecha por dual a izquierda no tendríamos la estructura de coálgebra cuando los funtores son sobre categorías tensoriales no simétricas.
\end{remark}

\pagebreak
\section{Definiciones Algebraicas} \label{sec:DefAlg}

\subsection{Álgebras y coálgebras}

A partir de esta sección, $\Vat$ será una categoría tensorial con neutro $I$. En toda esta sección puede reemplazarse $\Vat$ por $R$-Mod e $I$ por un anillo conmutativo $R$, o bien por $Vec_K$ y $K$ respectivamente para obtener las definiciones más usuales.

\begin{\de}
	Un objeto $A$ de $\Vat$ es un álgebra si se tienen dos flechas de $\Vat$, una multiplicación $m: A \otimes A \rightarrow A$ asociativa y una unidad $u: I \rightarrow A$.
	
	Asoc:
	\xymatrix{ A \otimes A \otimes A \ar[r]^{m \otimes id} \ar[d]_{id \otimes m} & A \otimes A \ar[d]^{m} \\
							A \otimes A \ar[r]^{m} & A}
	\ \ Unid:	
	\xymatrix {& A \otimes A \ar[dd]^{m} & \\
						 I \otimes A \ar[ur]^{u \otimes id} \ar[dr]_{\cong} & & A \otimes I \ar[ul]_{id \otimes u} \ar[dl]^{\cong} \\
						 & A }
\end{\de}

\begin{remark}
	Las flechas $m$ y $u$ son parte de la estructura de álgebra de $A$, pero se omiten para aliviar la escritura. Siempre usaremos esas letras para referirnos a estas flechas. Observaciones análogas deben considerarse hechas para todas las definiciones de esta sección.
\end{remark}

Recordemos que, en una categoría tensorial $\Vat$, tener un monoide $M$ respecto de $\otimes$ es tener un objeto $M$ de $\Vat$ con una flecha multiplicación $\mu: M \otimes M \rightarrow M$ asociativa ($M \otimes M \otimes M \stackrel{\mu \circ (\mu \otimes M) = \mu \circ (M \otimes \mu)}{\longrightarrow} M$) y una flecha $e : I \rightarrow M$ unidad ($M \stackrel{\mu \circ (M \otimes e) = id_M = \mu \circ (e \otimes M)}{\longrightarrow} M$).

Un álgebra no es otra cosa que un monoide respecto de $\otimes$.

\begin{\de}
	Una flecha en $\Vat$ entre dos álgebras $f: A \rightarrow A'$ es un morfismo de álgebras si respeta la multiplicación y la unidad, es decir si los siguientes diagramas conmutan:
	
Resp. mult.
	\xymatrix{ A \otimes A \ar[r]^{m} \ar[d]_{f \otimes f} & A \ar[d]^{f} \\
						 A' \otimes A' \ar[r]^{m} & A'} 
\ \ \ \ Resp. unid.						 
	\xymatrix { & A \ar[dd]^{f} \\
					I \ar[ur]^{u} \ar[dr]_{u} \\
							& A' }
\end{\de}

Las álgebras con sus morfismos forman una categoría que notaremos $Alg_\Vat$. Se tiene un funtor de olvido $U: Alg_\Vat \rightarrow \Vat$.

Podemos dualizar estas definiciones para definir la estructura de coálgebra y sus morfismos.

\begin{\de}
	Un objeto $C$ de $\Vat$ es una coálgebra si se tienen dos flechas de $\Vat$, una comultiplicación $\triangle : C \rightarrow C \otimes C$ coasociativa y una counidad $\varepsilon: C \rightarrow I$.
	Coasoc:	
	\xymatrix{ C \otimes C \otimes C & C \otimes C \ar[l]_>>>>>>{\triangle  \otimes id} \\
							C \otimes C \ar[u]^{id \otimes \triangle } & C \ar[l]_{\triangle } \ar[u]_{\triangle } }
	\ \ Counid:
	
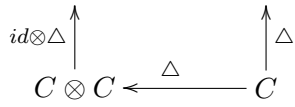
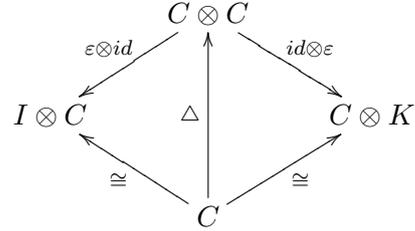
\xymatrix {& C \otimes C \ar[dl]_{\varepsilon \otimes id} \ar[dr]^{id \otimes \varepsilon} & \\
						 I \otimes C & & C \otimes K \\
						 & C \ar[uu]^{\triangle } \ar[ul]^{\cong} \ar[ur]_{\cong} }
	
\end{\de}

Una coálgebra es un comonoide respecto de $\otimes$. 

\begin{\de}
	Una flecha en $\Vat$ entre dos coálgebras $f: C \rightarrow C'$ es un morfismo de coálgebras si respeta la comultiplicación y la counidad, es decir si los siguientes diagramas conmutan:
	
	Resp. comult.
		
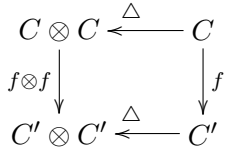
\xymatrix{ C \otimes C \ar[d]_{f \otimes f} & C \ar[l]_{\triangle} \ar[d]^{f} \\
						 C' \otimes C' & C' \ar[l]_{\triangle} } 
	\ \ \ \ Resp. counid.						 
		
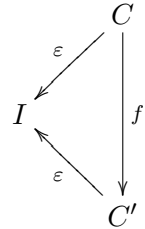
\xymatrix { & C \ar[dd]^{f} \ar[dl]_{\varepsilon} \\
							I \\
								& C' \ar[ul]^{\varepsilon} } 
		
\end{\de}

Las coálgebras con sus morfismos forman categoría que notaremos \linebreak $Coalg_\Vat$. Se tiene también un funtor de olvido $U: Coalg_\Vat \rightarrow \Vat$.

Si $\Vat$ es simétrica, su producto tensorial se restringe correctamente a sus álgebras y sus coálgebras, y a los morfismos. Más explícitamente, si $A$ y $B$ son álgebras en $\Vat$ con multiplicaciones $m$ y $m'$, tenemos la multiplicación 
$$ A \otimes B \otimes A \otimes B \stackrel{A \otimes \psi \otimes B}{\longrightarrow} A \otimes A \otimes B \otimes B \stackrel{m \otimes m'}{\longrightarrow} A \otimes B$$
que le da a $A \otimes B$ una estructura de álgebra. Se dejan por verificar las unidades y el hecho de que el producto tensorial de morfismos de álgebras resulta un morfismo de álgebras. De forma análoga se define el producto tensorial de coálgebras y de sus morfismos. Se tiene entonces la siguiente propiedad

\begin{\prop} \label{AlgCoalgTens}
Si $\Vat$ es una categoría tensorial simétrica con producto tensorial $\otimes$, entonces $Alg_\Vat$ y $Coalg_\Vat$ obtienen con $\otimes$ estructuras de categoría tensorial (simétrica). Además, el álgebra $I$ es el objeto inicial de $Alg_\Vat$.
\end{\prop}

\begin{\ej}
	Los anillos (no necesariamente conmutativos) son las álgebras en la categoría de grupos abelianos. Allí se tiene $\Z[X] \otimes \Z[Y] = \Z[X,Y]$ (anillos de polinomios), pero el coproducto $\Z[X] \coprod{} \Z[Y]$ es el anillo de polinomios en dos variables que no conmutan entre sí. Luego, en general, en $Alg_\Vat$ producto tensorial y coproducto (aunque existan) no coinciden.
\end{\ej}

Dentro de la categoría $Alg_\Vat$ tenemos la subcategoría plena $Alg_\Vat^{Sim}$ de las álgebras conmutativas, es decir aquellas para las que vale \linebreak $m = m \circ \psi$. Allí también se restringe el producto tensorial de $\Vat$, pero en este caso sí resulta un coproducto:

\begin{\prop} \label{AlgSimTens}
En la categoría tensorial $Alg_\Vat^{Sim}$ el producto tensorial es un coproducto con inclusiones $$A \stackrel{A \otimes u}{\longrightarrow} A \otimes B \stackrel{u \otimes B}{\longleftarrow} B$$
\end{\prop}

\begin{proof}
	Debemos mostrar que dados morfismos de álgebras $A \stackrel{f}{\rightarrow} C$ y $B \stackrel{g}{\rightarrow} C$, existe un único morfismo $A \otimes B \stackrel{\varphi}{\rightarrow} C$ que hace conmutar al diagrama
	
\begin{equation} \label{varphicumple}
	\xymatrix{ B \ar[rd]^{g} \ar[d]_{u \otimes B} \\
						 A \otimes B \ar[r]^{\varphi} & C \\
						 A \ar[u]^{a \otimes u} \ar[ru]_{f} }
\end{equation}

Consideramos el diagrama

$$ \xymatrix@R=0.5pc { & A \otimes B \ar@/^8pc/[dddd]_{f \otimes g} \ar[dd]^{A \otimes u \otimes u \otimes B} \\ \\
					 				 & A \otimes B \otimes A \otimes B \ar[dd]^{\varphi \otimes \varphi} \ar[dl]^{A \otimes \psi \otimes B} \\
					 				 A \otimes A \otimes B \otimes B \ar[dd]_{m \otimes m} \\
					 				 & C \otimes C \ar[dd]^{m} \\
					 				 A \otimes B \ar[dr]^{\varphi} \\
					 				 & C } $$

del que la parte derecha conmuta por \eqref{varphicumple} y la parte izquierda pues $\varphi$ es morfismo de álgebras. La composición de la izquierda se corresponde, en el cálculo de ascensores, a

$$ \xymatrix@C=0pc{ A \ar@2{-}[d] & & \ar@{-}[dl] \ar@{}[d]|{u} \ar@{-}[dr] & & \ar@{-}[dl] \ar@{}[d]|{u} \ar@{-}[dr] & & B \ar@2{-}[d] \\
								 A \ar@2{-}[d] & & B \ar@2{-}[drr]|{\psi} & & A \ar@2{-}[dll]|{\psi} & & B \ar@2{-}[d] 				& &     A \ar@2{-}[d] & & \ar@{-}[dl] \ar@{}[d]|{u} \ar@{-}[dr] & & \ar@{-}[dl] \ar@{}[d]|{u} \ar@{-}[dr] & & B \ar@2{-}[d] \\
								 A \ar@{-}[dr] & \ar@{}[d]|{m} & A \ar@{-}[dl] & & B \ar@{-}[dr] & \ar@{}[d]|{m} & B \ar@{-}[dl] 				& \ \ \ \ \ar@{}[d]|{=} \ \ \ \ & 					A \ar@{-}[dr] & \ar@{}[d]|{m} & A \ar@{-}[dl] & & B \ar@{-}[dr] & \ar@{}[d]|{m} & B \ar@{-}[dl]   & \ \ \ \  \ar@{}[d]|{=} \ \ \ \ &			A \ar@{-}[dr] & \ar@{}[d]|{\varphi} & B \ar@{-}[dl] \\
								 & A \ar@{-}[drr] & & \ar@{}[d]|{\varphi} & & B \ar@{-}[dll] & 			& &       & A \ar@{-}[drr] & & \ar@{}[d]|{\varphi} & & B \ar@{-}[dll] &			& & 	& C \\
								 & & & C & & & 				& & 			& & & C } $$

donde la primer igualdad es válida por \eqref{swap} y la segunda pues $u$ es unidad para $m$. Luego la única definición posible es $\varphi = m \circ (f \otimes g)$, y la igualdad

$$ \xymatrix@C=0pc{ A \ar@2{-}[d] & & \ar@{-}[dl] \ar@{}[d]|{u} \ar@{-}[dr] \\
								 A {\ar@<4pt>@{-}'+<0pt,-6pt>[d] \ar@<-4pt>@{-}'+<0pt,-6pt>[d]^{f}} & & B {\ar@<4pt>@{-}'+<0pt,-6pt>[d] \ar@<-4pt>@{-}'+<0pt,-6pt>[d]^{g}} & & \ \ \ \ \ar@{}[d]|{=} \ \ \ \ & & A {\ar@<4pt>@{-}'+<0pt,-6pt>[d] \ar@<-4pt>@{-}'+<0pt,-6pt>[d]^{f}} & & \ar@{-}[dl] \ar@{}[d]|{u} \ar@{-}[dr] & & \ \ \ \ \ar@{}[d]|{=} \ \ \ \ & &	A {\ar@<4pt>@{-}'+<0pt,-6pt>[d] \ar@<-4pt>@{-}'+<0pt,-6pt>[d]^{f}} \\
								 C \ar@{-}[dr] & \ar@{}[d]|{m} & C \ar@{-}[dl] & & & & C \ar@{-}[dr] & \ar@{}[d]|{m} & C \ar@{-}[dl] & & & & C \\
								 & C & & & & & & C } $$

donde la primer igualdad vale pues $g$ respeta la unidad y la segunda pues $u$ es unidad para $m$, nos da la conmutatividad del triángulo inferior en \eqref{varphicumple}. La conmutatividad del otro triángulo es análoga. Resta verificar que $m \circ (f \otimes g)$ es efectivamente un morfismo de álgebras, este hecho es el que es falso para álgebras no conmutativas. Ver que $m \circ (f \otimes g)$ respeta la multiplicación es ver la igualdad

$$ 
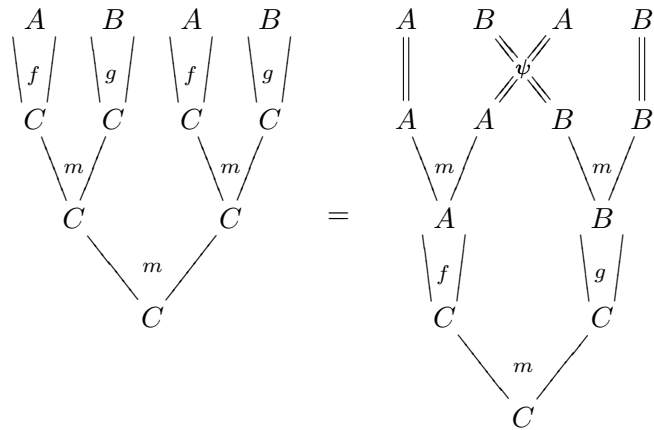
\xymatrix@C=0pc{ A {\ar@<4pt>@{-}'+<0pt,-6pt>[d] \ar@<-4pt>@{-}'+<0pt,-6pt>[d]^{f}} & & B {\ar@<4pt>@{-}'+<0pt,-6pt>[d] \ar@<-4pt>@{-}'+<0pt,-6pt>[d]^{g}} & & A {\ar@<4pt>@{-}'+<0pt,-6pt>[d] \ar@<-4pt>@{-}'+<0pt,-6pt>[d]^{f}} & & B {\ar@<4pt>@{-}'+<0pt,-6pt>[d] \ar@<-4pt>@{-}'+<0pt,-6pt>[d]^{g}} 			& & 			A \ar@2{-}[d] & & B \ar@2{-}[drr]|{\psi} & & A \ar@2{-}[dll]|{\psi} & & B \ar@2{-}[d] \\
								C \ar@{-}[dr] & \ar@{}[d]|{m} & C \ar@{-}[dl] & & C \ar@{-}[dr] & \ar@{}[d]|{m} & C \ar@{-}[dl] 			& & 			A \ar@{-}[dr] & \ar@{}[d]|{m} & A \ar@{-}[dl] & & B \ar@{-}[dr] & \ar@{}[d]|{m} & B \ar@{-}[dl] \\
								& C \ar@{-}[drr] & & \ar@{}[d]|{m} & & C \ar@{-}[dll] &		& \ \ \ = \ \ \ & 				& A {\ar@<4pt>@{-}'+<0pt,-6pt>[d] \ar@<-4pt>@{-}'+<0pt,-6pt>[d]^{f}} & & & & B {\ar@<4pt>@{-}'+<0pt,-6pt>[d] \ar@<-4pt>@{-}'+<0pt,-6pt>[d]^{g}}  \\
								& & &  C & & & 				& & 			& C \ar@{-}[drr] & & \ar@{}[d]|{m} & & C \ar@{-}[dll] \\
								& & & & & & 					& & 			& & & C} $$
								
Trabajando con la composición de la derecha, usando primero que $f$ y $g$ respetan la multiplicación y luego la igualdad \eqref{psinat} el diagrama de la derecha nos queda

$$ 
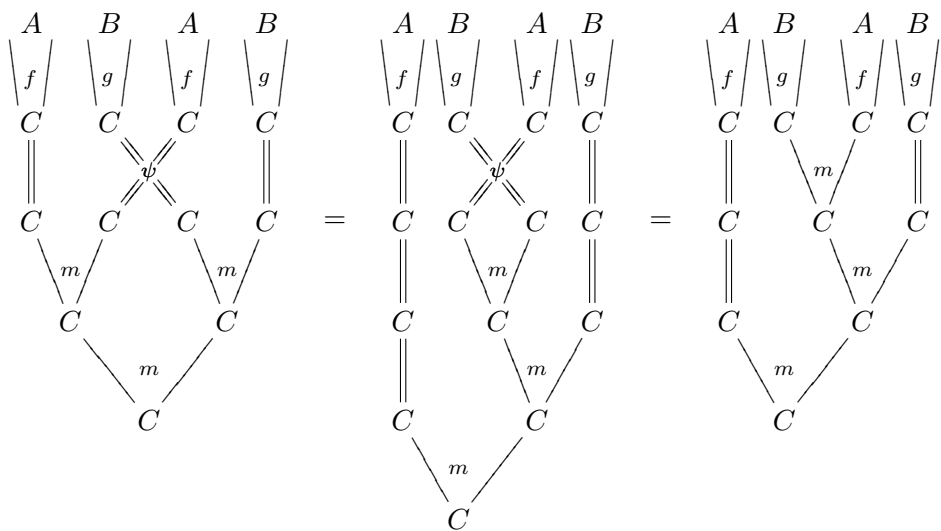
\xymatrix@C=0pc{ A {\ar@<4pt>@{-}'+<0pt,-6pt>[d] \ar@<-4pt>@{-}'+<0pt,-6pt>[d]^{f}} & & B {\ar@<4pt>@{-}'+<0pt,-6pt>[d] \ar@<-4pt>@{-}'+<0pt,-6pt>[d]^{g}} & & A {\ar@<4pt>@{-}'+<0pt,-6pt>[d] \ar@<-4pt>@{-}'+<0pt,-6pt>[d]^{f}} & & B {\ar@<4pt>@{-}'+<0pt,-6pt>[d] \ar@<-4pt>@{-}'+<0pt,-6pt>[d]^{g}} 		& & 		A {\ar@<4pt>@{-}'+<0pt,-6pt>[d] \ar@<-4pt>@{-}'+<0pt,-6pt>[d]^{f}} & & B {\ar@<4pt>@{-}'+<0pt,-6pt>[d] \ar@<-4pt>@{-}'+<0pt,-6pt>[d]^{g}} & & A {\ar@<4pt>@{-}'+<0pt,-6pt>[d] \ar@<-4pt>@{-}'+<0pt,-6pt>[d]^{f}} & & B {\ar@<4pt>@{-}'+<0pt,-6pt>[d] \ar@<-4pt>@{-}'+<0pt,-6pt>[d]^{g}}			& & 		A {\ar@<4pt>@{-}'+<0pt,-6pt>[d] \ar@<-4pt>@{-}'+<0pt,-6pt>[d]^{f}} & & B {\ar@<4pt>@{-}'+<0pt,-6pt>[d] \ar@<-4pt>@{-}'+<0pt,-6pt>[d]^{g}} & & A {\ar@<4pt>@{-}'+<0pt,-6pt>[d] \ar@<-4pt>@{-}'+<0pt,-6pt>[d]^{f}} & & B {\ar@<4pt>@{-}'+<0pt,-6pt>[d] \ar@<-4pt>@{-}'+<0pt,-6pt>[d]^{g}}	\\
										C \ar@2{-}[d] & & C \ar@2{-}[drr]|{\psi} & & C \ar@2{-}[dll]|{\psi} & & C \ar@2{-}[d] & & C \ar@2{-}[d] & & C \ar@2{-}[drr]|{\psi} & & C \ar@2{-}[dll]|{\psi} & & C \ar@2{-}[d]  			& & 			C \ar@2{-}[d] & & C \ar@{-}[dr] & \ar@{}[d]|{m} & C \ar@{-}[dl] & & C \ar@2{-}[d] \\
										C \ar@{-}[dr] & \ar@{}[d]|{m} & C \ar@{-}[dl] & & C \ar@{-}[dr] & \ar@{}[d]|{m} & C \ar@{-}[dl] 		& \ \ \ = \ \ \ & 		C \ar@2{-}[d] & & C \ar@{-}[dr] & \ar@{}[d]|{m} & C \ar@{-}[dl] & & C \ar@2{-}[d] 		& \ \ \ = \ \ \ & 		C \ar@2{-}[d] & & & C \ar@{-}[dr] & \ar@{}[d]|{m} & & C \ar@{-}[dll] \\
										& C \ar@{-}[drr] & & \ar@{}[d]|{m} & & C \ar@{-}[dll] &		& & 		C \ar@2{-}[d] & & & C \ar@{-}[dr] & \ar@{}[d]|{m} & & C \ar@{-}[dll] 	& & 		C \ar@{-}[drr] & & \ar@{}[d]|{m} & & C \ar@{-}[dll] \\
										& & &  C & & & 				& & 		C \ar@{-}[drr] & & \ar@{}[d]|{m} & & C \ar@{-}[dll] 	& & 	& & & & C & & 	\\
										& & & & & & 					& & 		& & C & & } $$

La primer igualdad es válida por la asociatividad de $m$ (usada dos veces) y la segunda por la conmutatividad de $C$. Usando nuevamente la asociatividad de $m$ dos veces se obtiene el diagrama deseado.
\end{proof}

\subsection{Biálgebras y álgebras de Hopf} \label{BialgyHopf}

\begin{\de}
Una biálgebra $B$ en $\Vat$ es un álgebra y una coálgebra donde las dos estructuras son compatibles en el sentido de que los siguientes cuatro diagramas conmutan.

1.
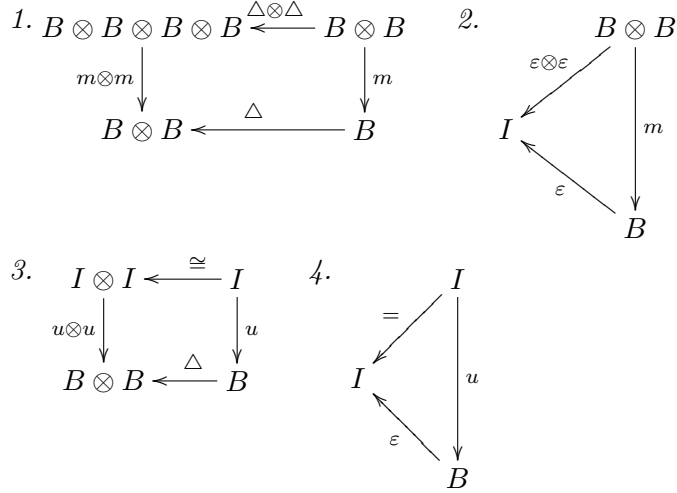
\xymatrix{ B \otimes B \otimes B \otimes B \ar[d]_{m \otimes m} & B \otimes B \ar[l]_>>>>{\triangle \otimes \triangle} \ar[d]^{m} \\
					 B \otimes B & B \ar[l]_{\triangle} }
\ \ \ \ 2. \xymatrix{ & B \otimes B \ar[dl]_{\varepsilon \otimes \varepsilon} \ar[dd]^{m} \\
				I \\
					 & B \ar[ul]^{\varepsilon} }

3. \xymatrix{ I \otimes I \ar[d]_{u \otimes u} & I \ar[l]_<<<<{\cong} \ar[d]^{u} \\
					 B \otimes B & B \ar[l]_<<<<{\triangle} } 
\ \ \ \ 4. \xymatrix{ & I \ar[dl]_{=} \ar[dd]^{u} \\
				I \\
					 & B \ar[ul]^{\varepsilon} }
\end{\de}

\begin{remark}
Si $\Vat$ es simétrica, la conmutatividad de los cuatro diagramas es equivalente a cualquiera de las siguientes condiciones.
\begin{enumerate}
	\item $m$ y $u$ son morfismos de coálgebras. 
	\item $\triangle$ y $\varepsilon$ son morfismos de álgebras.
\end{enumerate}

Luego, en el caso en que $\Vat$ es simétrica, una biálgebra es, por definición, un monoide en la categoría tensorial $Coalg_\Vat$ (o un comonoide en $Alg_\Vat$).
\end{remark}

\begin{\ej}
	Sea $(M, \cdot, 1_M)$ un monoide. Usaremos la notación de la delta de Kronecker para dos elementos $m,m'$ de $M$, o sea $\delta_{m,m'}$ será igual a $1$ si $m=m'$ y a $0$ en caso contrario. Tomemos $K = \{0,1\}$. Entonces $\Z$ es una $K$-álgebra con la unidad y multiplicación usuales y $\Z M = \oplus_M \Z$ es una \linebreak $K$-coálgebra con la estructura de $K$-e.v. dada por sumar en cada coordenada y la comultiplicación $\triangle$ definida en la base $\{ e_m \}_{m \in M} $, donde $(e_m)_{m'} = \delta_{m,m'}$, como $\triangle (e_m) = \sum_{m_1 \cdot m_2 = m} m_1 \otimes m_2$. La counidad es $e_m \mapsto \delta_{m,1_M }$
	
	Entonces $Hom_K (\Z M, \Z) = \{ f: M \rightarrow \Z$ de soporte finito$\}$ tiene una estructura de $K$-álgebra con el producto de convolución dado por \linebreak $f*g(e_m) = \sum_{m_1 \cdot m_2 = m} f(e_{m_1}) g(e_{m_2})$. Notemos que $f*g$ viene dada por la composición $\Z M \stackrel{\triangle}{\rightarrow}\Z M \otimes \Z M \stackrel{f \otimes g}{\rightarrow} \Z \otimes \Z \stackrel{m}{\rightarrow} \Z$. De forma similar se obtienen los ejemplos clásicos de convolución de polinomios y de funciones integrables. 
	
	Esta construcción se generaliza a las $K$-álgebras (es decir álgebras en la categoría tensorial $Vec_K$) de la siguiente forma: Si $C$ es una $K$-coálgebra y $A$ es una $K$-álgebra entonces el $K$-e.v. de las transformaciones $K$-lineales $Hom_K(C,A)$ es un álgebra con multiplicación $*$ y unidad $u \circ \varepsilon$, donde $f*g$ es la composición $C \stackrel{\triangle}{\rightarrow} C \otimes C \stackrel{f \otimes g}{\rightarrow} A \otimes A \stackrel{m}{\rightarrow} A$.
	En particular, si $B$ es una $K$-biálgebra entonces $[B,B]$ es una $K$-álgebra con multiplicación $*$. Una \linebreak $K$-biálgebra $B$ es una $K$-álgebra de Hopf si $id_B \in [B,B]$ es inversible, es decir si existe $a \in [B,B]$ (llamada antípoda) tal que $id_B * a = u \circ \varepsilon = a * id_B$.
\end{\ej}

El ejemplo anterior motiva la siguiente generalización de estas definiciones a categorías tensoriales cualesquiera.

\begin{\de}
Sea $\Vat$ una categoría tensorial simétrica con hom internos, $C$ una coálgebra y $A$ un álgebra en $\Vat$. Entonces el objeto de $\Vat$ $Hom(C,A)$ es un álgebra tomando la unidad $I \rightarrow Hom(C,A)$ dada por la ley exponencial a partir de la flecha $C \stackrel{\varepsilon}{\rightarrow} I \stackrel{u}{\rightarrow} A$ y la convolución $$Hom(C,A) \otimes Hom(C,A) \stackrel{*}{\rightarrow} Hom(C,A)$$ dada por la ley exponencial a partir de la flecha 

 \xymatrix{ Hom(C,A) \otimes Hom(C,A) \otimes C \ar@/^/[r]^{Hom(C,A) \otimes Hom(C,A) \otimes \triangle} & Hom(C,A) \otimes Hom(C,A) \otimes C \otimes C \cdots }
 \xymatrix{ \cdots \ar[rr]^<<<<<<<<{Hom(C,A) \otimes \psi \otimes C} & & Hom(C,A) \otimes C \otimes Hom(C,A) \otimes C \ar[r]^<<<<<{ev \otimes ev} & A \otimes A \ar[r]^{m} & A }
\end{\de}
 
Al igual que antes, si $B$ es una biálgebra entonces $End (B) = Hom(B,B)$ es un álgebra con multiplicación $*$. 

Para poder definir las álgebras de Hopf en $\Vat$. debemos poder interpretar qué significa que $id_B \in End (B)$ sea inversible, ya que la convolución no necesariamente es ahora una función definida sobre elementos sino una flecha abstracta en $\Vat$. Para ello notemos que la ley exponencial nos da la biyección

\begin{center}
\underline{$\ \ I \rightarrow End(B) \ \ $} \\
$B \rightarrow B$
\end{center}
que nos permite pensar que los elementos de $End(B)$ son las flechas \linebreak $I \rightarrow End(B)$. Entonces dada una flecha $a: B \rightarrow B$ tenemos flechas \linebreak $I \rightarrow End(B)$ en $\Vat$ dadas a partir de $id_B$ y $a$, que las llamamos $e$ y $\hat{a}$ respectivamente y que cumplen 

\begin{equation} \label{cumplen}
\xymatrix@C=-0.7pc @R=0.5pc { &   & & B \ar@2{-}[dd] & & & 						  	& & & & 				  & & & B \ar@2{-}[dd] & & \\
               & & & & & & B \ar@2{-}[dd] & & & & 				  & & & & & & B {\ar@<4pt>@{-}'+<0pt,-6pt>[dd] \ar@<-4pt>@{-}'+<0pt,-6pt>[dd]^{a}}\\
										& End(B) {\ar@<4pt>@{-}'+<-4pt,8pt>[uu] \ar@<-4pt>@{-}'+<4pt,8pt>[uu]^{\textcolor{white}{e}e\textcolor{white}{e}}} \ar@{-}[ddr] & \ar@{}[dd]|{ev} & B \ar@<4pt>@{-}[ddl] 			 & \ \  \ \  \ \   = \ \  \ \  \ \   & &  							& & \ \  \ \  \ \  \ \  \ \  \ \   \hbox{y} \ \  \ \  \ \  \ \  \ \  \ \   & &  & End(B) {\ar@<4pt>@{-}'+<-4pt,8pt>[uu] \ar@<-4pt>@{-}'+<4pt,8pt>[uu]^{\textcolor{white}{e}\hat{a}\textcolor{white}{e}}} \ar@{-}[ddr] & \ar@{}[dd]|{ev} & B \ar@<4pt>@{-}[ddl] 			 & \ \  \ \  \ \   = \ \  \ \  \ \   & &  \\
										& & & & & & B  & & & & 				  & & & & & & B \\
										& & B &																											 & & & 						 								& & & &					  & & B &	}
\end{equation}
										
Por lo tanto podemos aplicar la convolución a estos elementos de $End(B)$, mediante el diagrama

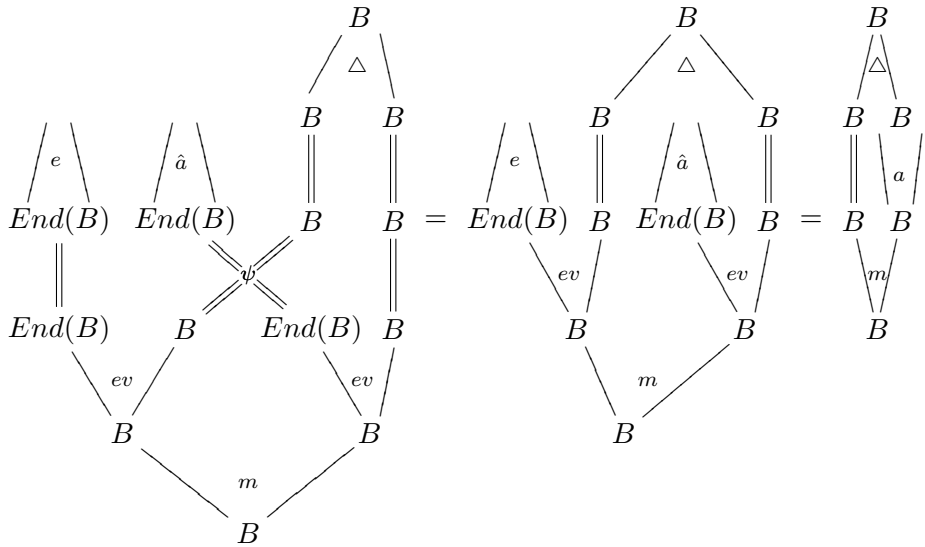
\xymatrix@C=-0.5pc { & & & & & & \! \! \! \! \!  B \ar@<-5pt>@{-}'+<-2pt,-10pt>[dl] \ar@<2pt>@{-}[dr] \ar@<-0.5em>@{}[d]|{\triangle} & & & & 			& & & & & B \ar@{-}[dll] \ar@{-}[drr] \ar@{}[d]|{\triangle} & & & & & &				& B \ar@{-}[dl] \ar@{-}[dr] \ar@{}[d]|{\triangle} & \\
										&  & & & & B \ar@2{-}[d] & & B \ar@2{-}[d]  & & & 		& & & B \ar@2{-}[d] & & & & B \ar@2{-}[d] & & & &				B \ar@2{-}[d] & & B {\ar@<4pt>@{-}'+<0pt,-6pt>[d] \ar@<-4pt>@{-}'+<0pt,-6pt>[d]^{a}} \\
										& End(B) {\ar@<4pt>@{-}'+<-4pt,8pt>[u] \ar@<-4pt>@{-}'+<4pt,8pt>[u]^{\textcolor{white}{e}e\textcolor{white}{e}}} \ar@2{-}[d] & & End(B) {\ar@<4pt>@{-}'+<-4pt,8pt>[u] \ar@<-4pt>@{-}'+<4pt,8pt>[u]^{\textcolor{white}{e}\hat{a}\textcolor{white}{e}}} \ar@2{-}[drr]|{\psi} & & B \ar@2{-}[dll]|{\psi} & & B \ar@2{-}[d] & \ \ = \ \ & & 													& End(B) {\ar@<4pt>@{-}'+<-4pt,8pt>[u] \ar@<-4pt>@{-}'+<4pt,8pt>[u]^{\textcolor{white}{e}e\textcolor{white}{e}}} \ar@{-}[dr] & \ar@<-0.3pc>@{}[d]|{ev} & B \ar@<2pt>@{-}[dl] & & End(B) {\ar@<4pt>@{-}'+<-4pt,8pt>[u] \ar@<-4pt>@{-}'+<4pt,8pt>[u]^{\textcolor{white}{e}\hat{a}\textcolor{white}{e}}} \ar@{-}[dr] & \ar@<-0.3pc>@{}[d]|{ev} & B \ar@<2pt>@{-}[dl] & & \ \ = \ \ & &			B \ar@{-}[dr] & \ar@{}[d]|{m} & B \ar@{-}[dl] \\
										& End(B) \ar@{-}[dr] & \ar@{}[d]|{ev} & B \ar@{-}[dl] & & End(B) \ar@{-}[dr] & \ar@<-0.3pc>@{}[d]|{ev} & B \ar@<2pt>@{-}[dl] & & & 			& & \! B \ar@{-}[drr] & & \ar@{}[d]|{\ \ \ \ \ m} & & \! B \ar@{-}[dll]	& & & & &  			& B		\\
										& & B \ar@{-}[drr] & & \ar@{}[d]|{m} & & \! B \ar@{-}[dll] & & & & 		& & & & B &	& & & & & & 		\\
										& & & & B & & & & & & }
										
La primer igualdad es por la naturalidad de $\psi$ (ver \eqref{swap}), y la segunda por \eqref{cumplen}. Luego, como la ley exponencial es una biyección, la igualdad entre los elementos de $End(B)$ que teníamos de la definición de las $K$-álgebras de Hopf la podemos traducir a una igualdad entre las flechas $B \rightarrow B$ que es la que nos permite hacer la siguiente definición.

\begin{\de}
	Sea $\Vat$ una categoría tensorial. Una biálgebra $B$ es un álgebra de Hopf en $\Vat$ si existe una flecha $a: B \rightarrow B$ en $\Vat$ (llamada antípoda) tal que el siguiente diagrama (y su análogo multiplicando por $a$ al otro lado) conmutan.
	
\ \ \ \ \ \ \ \ \ \ \ \ \ \ \ \ \ \ \ \ \ \ \ \ \ \ \xymatrix @C=0.5pc {	& B \otimes B \ar[rr]^{B \otimes a} & & B \otimes B \ar[dr]^{m} \\
						B \ar[ur]^{\triangle} \ar[drr]_{\varepsilon} & & & & B \\
						& & I \ar[urr]_{u}  }
\end{\de}

\subsection{$A$-módulos y $C$-comódulos}

\begin{\de}
	Sea $A$ un álgebra en $\Vat$. Un $A$-módulo (a izquierda) es un objeto $M$ de $\Vat$ con una flecha $\gamma: A \otimes M \rightarrow M$ llamada producto escalar tal que respeta la multiplicación y la unidad, es decir tal que los siguientes diagramas conmutan:
	
	Resp. mult.
	
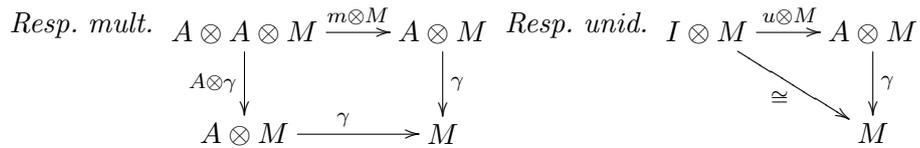
\xymatrix{ A \otimes A \otimes M \ar[r]^>>>>>{m \otimes M} \ar[d]_{A \otimes \gamma} & A \otimes M \ar[d]^{\gamma} \\
							A \otimes M \ar[r]^{\gamma} & M }
	\ Resp. unid.
	\xymatrix { I \otimes M \ar[r]^{u \otimes M} \ar[dr]_{\cong} & A \otimes M \ar[d]^{\gamma} \\
								& M }
	
\end{\de}

\begin{\de}
	Sea $C$ una coálgebra en $\Vat$. Un $C$-comódulo (a izquierda) es un objeto $M$ de $\Vat$ con una flecha $\rho: M \rightarrow C \otimes M$ llamada coproducto escalar tal que respeta la comultiplicación y la counidad, es decir tal que los siguientes diagramas conmutan:
	
	Resp. com.
	\xymatrix{ C \otimes C \otimes M & C \otimes M \ar[l]_>>>>>{C \otimes \rho} \\
						 C \otimes M \ar[u]^{\triangle \otimes M} & M \ar[u]_{\rho} \ar[l]_{\rho} }
	Resp. coun.
	\xymatrix{ I \otimes M & C \otimes M \ar[l]_{\varepsilon \otimes M} \\
							  & M \ar[ul]^{\cong} \ar[u]_{\rho} }
							  
\end{\de}

\begin{\de}
	Sean $A$ un álgebra en $\Vat$, $M$ y $M'$ dos $A$-módulos. Una flecha $f: M \rightarrow M'$ de $\Vat$ es un morfismo de $A$-módulos si el siguiente cuadrado conmuta:
	
	\ \ \ \ \ \ \ \ \ \ \ \ \xymatrix { A \otimes M \ar[r]^{A \otimes f} \ar[d]_{\gamma_M} & A \otimes M' \ar[d]^{\gamma_{M'}} \\
							M \ar[r]^{f} & M' }
\end{\de}

\begin{\de}
	Sean $C$ una coálgebra en $\Vat$, $M$ y $M'$ dos $C$-comódulos. Una flecha $f: M \rightarrow M'$ de $\Vat$ es un morfismo de $C$-comódulos si el siguiente cuadrado conmuta:

	\ \ \ \ \ \ \ \ \ \ \ \ \xymatrix { C \otimes M \ar[r]^{C \otimes f} & C \otimes M' \\
							M \ar[u]^{\rho_M} \ar[r]^{f} & M' \ar[u]_{\rho_{M'}} }
\end{\de}

\pagebreak
\section{Ends y Coends} \label{sec:Ends}

\subsection{Definiciones}

Seguimos en las definiciones a Mac Lane (\cite{ML}, IX, 4-6, pp.214-222).

\begin{\de}
	Sean $\Cat, \Bat$ dos categorías y $S, T : \Cat^{op} \times \Cat \rightarrow \Bat$ funtores. Una transformación di-natural $\alpha: S \stackrel{..}{\rightarrow} T$ es tener para cada objeto $C$ de $\Cat$ una flecha $\alpha_C : S(C,C) \rightarrow T(C,C)$ en $\Bat$, de forma tal que para cada flecha $f: C \rightarrow C'$ de $\Cat$ el siguiente diagrama conmute
	
	\ \ \ \ \ \ \ \ \ \ \ \ 	\xymatrix{
		& S(C,C) \ar[r]^{\alpha_C} & T(C,C) \ar[dr]^{T(1,f)} \\
		S(C',C) \ar[ur]^{S(f,1)} \ar[dr]_{S(1,f)} & & & T(C,C') \\
		& S(C',C') \ar[r]^{\alpha_{C'}} & T(C',C') \ar[ur]_{T(f,1)}	}
\end{\de}

Nos interesan para las definiciones de end y coend, las transformaciones dinaturales cuando uno de los dos funtores involucrados es constante.

\begin{\de} \label{def:wedge}
	Sean $\Cat, \Bat$ dos categorías, $T : \Cat^{op} \times \Cat \rightarrow \Bat$ un funtor y $B$ un objeto de $\Bat$. Un di-cono (llamado wedge en \cite{ML}) de $B$ a $T$ es una transformación di-natural $\alpha: S \stackrel{..}{\rightarrow} T$ donde $S$ es el funtor constante $B$. Notamos $\alpha: B \stackrel{..}{\rightarrow} T$
\end{\de}

El diagrama queda en este caso, para toda flecha $f: C \rightarrow C'$ en $\Cat$

	\ \ \ \ \ \ \ \ \ \ \ \ 	\xymatrix{
		B \ar[r]^{\alpha_C} \ar[d]_{\alpha_{C'}} & T(C,C) \ar[d]^{T(1,f)} \\
		T(C',C') \ar[r]^{T(f,1)} & T(C,C') }
		
pues el lado izquierdo del hexágono se "`colapsa"' a $B$. Tenemos la definición análoga

\begin{\de} \label{def:cowedge}
	Sean $\Cat, \Bat$ dos categorías, $S : \Cat^{op} \times \Cat \rightarrow \Bat$ un funtor y $B$ un objeto de $\Bat$. Un di-cono (también llamado wedge en \cite{ML}) de $S$ a $B$ es una transformación di-natural $\alpha: S \stackrel{..}{\rightarrow} T$ donde $T$ es el funtor constante $B$. Notamos $\alpha: S \stackrel{..}{\rightarrow} B$
\end{\de}

El diagrama queda en este caso, para toda flecha $f: C \rightarrow C'$ en $\Cat$

	\ \ \ \ \ \ \ \ \ \ \ \ 	\xymatrix{
		S(C',C) \ar[r]^{S(1,f)} \ar[d]_{S(f,1)} &	 S(C',C') \ar[d]^{\alpha_{C'}} \\
		S(C,C) \ar[r]^{\alpha_C} & B }

pues ahora el lado derecho del hexágono se "`colapsa"' a $B$.

Ahora, dado un funtor $T: \Cat^{op} \times \Cat \rightarrow \Bat$ podemos definir su end como su di-cono universal:

\begin{\de}
	Sean $\Cat, \Bat$ dos categorías y $T : \Cat^{op} \times \Cat \rightarrow \Bat$ un funtor. Un end $(B, \alpha)$ de $T$ es un par compuesto por un objeto $B$ de $\Bat$ y un di-cono \linebreak $\alpha: B \stackrel{..}{\rightarrow} T$ tal que para todo otro di-cono $\alpha': B' \stackrel{..}{\rightarrow} T$ existe una única flecha $h : B' \rightarrow B$ en $\Cat$ tal que $\alpha_C' =  \alpha_C \circ h$ para todo $C$ en $\Cat$.
	
	En forma de diagrama, tenemos:
		
		\ \ \ \ \ \ \ \ \ \ \ \ 	\xymatrix{ B  \ar[r]^{\alpha_C} & T(C,C) \\
																				 B' \ar@{.>}[u]^{h} \ar[ur]_{\alpha_C'} }	

	Como todo objeto definido por propiedades universales, el end de un funtor es único salvo isomorfismos que preservan la estructura. Luego podemos referirnos al end de $T$: cuando hagamos eso nos estaremos refiriendo al objeto $B$, y lo notaremos $B = \int_\Cat T(C,C)$
\end{\de}

Análogamente, dado un funtor $S: \Cat^{op} \times \Cat \rightarrow \Bat$ podemos definir su coend como el end en la categoría dual:

\begin{\de}
	Sean $\Cat, \Bat$ dos categorías y $S : \Cat^{op} \times \Cat \rightarrow \Bat$ un funtor. Un coend $(B, \alpha)$ de $S$ es un par compuesto por un objeto $B$ de $\Bat$ y un di-cono \linebreak $\alpha: S \stackrel{..}{\rightarrow} B$ tal que para todo otro di-cono $\alpha': S \stackrel{..}{\rightarrow} B'$ existe una única flecha $h : B \rightarrow B'$ en $\Cat$ tal que $\alpha_C' =  h \circ \alpha_C$ para todo $C$ en $\Cat$.
	
	En forma de diagrama, ahora tenemos:
		
		\ \ \ \ \ \ \ \ \ \ \ \ 	\xymatrix{ S(C,C) \ar[dr]_{\alpha_C} \ar[r]^{\alpha_C'} & B'\\
																				 & B \ar@{.>}[u]_{h} }

	También como antes, por estar definido por una propiedad universal podemos referirnos al coend de $S$: cuando hagamos eso nos estaremos refiriendo al objeto $B$, y lo notaremos $B = \int^\Cat S(C,C)$
\end{\de}

De forma análoga a como se prueba en \cite{ML}, V.4.Theorem 1, p.112 que el funtor $Hom(C,-): \Cat \rightarrow \Ens$ preserva los límites y el funtor $Hom(-,C): \Cat \rightarrow \Ens$ transforma colímites en límites, se puede demostrar la siguiente propiedad.

\begin{\prop} \label{coendaend}
	Sean $\Cat, \Bat$ dos categorías, $T : \Cat^{op} \times \Cat \rightarrow \Bat$ un funtor y $B$ un objeto de $\Bat$. Si existen el end $\int_\Cat T(C,C)$ y/o el coend $\int^\Cat T(C,C)$, entonces existe también el respectivo miembro derecho de las igualdades siguientes y vale que
	$$[B,\int_\Cat T(C,C)] = \int_\Cat [B,T(C,C)] \hbox{\ \ \ y/o \ \ \ } [\int^\Cat T(C,C),B] = \int_\Cat [T(C,C),B]$$
\end{\prop}

\subsection{Nat como un end}

Veremos ahora dos ejemplos, un end y un coend, que nos servirán luego. El primero son las transformaciones naturales entre funtores.

\begin{\prop} \label{NatEnd}
	Sean $\Cat$ y $\Xat$ dos categorías, con $\Cat$ pequeña, y $U,V: \Cat \rightarrow \Xat$ dos funtores. Entonces $$nat[U,V] = \int_\Cat [UC,VC]$$
\end{\prop}
\begin{proof}
	Es estándar el hecho de que $[U-,V-]: \Cat^{op} \times \Cat \rightarrow \Ens$ es un funtor. Tenemos el di-cono $\alpha : nat[U,V] \stackrel{..}{\rightarrow} [U-,V-]$ definido según $\alpha_C(\theta) = \theta_C$. La dinaturalidad de $\alpha$ es la naturalidad de $\theta$. Falta ver la universalidad de este di-cono. Si tengo otro di-cono $\alpha' : B' \stackrel{..}{\rightarrow} [U-,V-]$, tengo entonces para cada $C$ de $\Cat$ la flecha $\alpha_C' : B' \rightarrow [UC,VC]$ tal que el cuadrado 

	\ \ \ \ \ \ \ \ \ \ \ \ 	\xymatrix{
		B' \ar[r]^{\alpha_C'} \ar[d]_{\alpha_{C'}'} & [UC,VC] \ar[d]^{V(f)} \\
		[UC',VC'] \ar[r]^{U(f)} & [UC,VC'] }

conmuta para toda flecha $f: C \rightarrow C'$ en $\Cat$. Eso quiere decir, que, para todo $b' \in B'$, el cuadrado 

	\ \ \ \ \ \ \ \ \ \ \ \ \ \ \ \ \ \ 	\xymatrix{
		UC \ar[r]^{\alpha_C'(b')} \ar[d]_{U(f)} & VC \ar[d]^{V(f)} \\
		UC' \ar[r]^{\alpha_{C'}'(b')} & VC'		}
		
conmuta. Tenemos que mostrar que existe una única $h: B' \rightarrow nat[U,V]$ tal que $\alpha_C' =  \alpha_C \circ h$ para todo $C$ en $\Cat$. Luego debe ser $$(h(b'))_C = \alpha_C(h(b')) = \alpha_C'(b')$$ 

Resta observar que el cuadrado anterior equivale con esta definición precisamente a que $h(b')$ sea una transformación natural.
\end{proof}

	Recordemos que si $\Vat$ es una categoría tensorial, una categoría enriquecida en $\Vat$ es tener
\begin{itemize}
	\item una familia $\Cat$ de objetos
	\item para cada par de objetos $C,D$ de $\Cat$ un objeto $Hom(C,D)$ de $\Vat$
	\item para cada terna de objetos $C,D,E$ de $\Cat$ una flecha en $\Vat$ asociativa $Hom(D,E) \otimes Hom(C,D) \stackrel{\circ}{\rightarrow} Hom(C,E)$
	\item para cada objeto $C$ una flecha $I \rightarrow Hom(C,C)$ en $\Vat$ neutro para $\circ$ a ambos lados.
\end{itemize}
	 
	 Dada una categoría enriquecida en $\Vat$, se obtiene una categoría subyacente aplicando el funtor $[I,-]$. Se puede pensar también que se comienza con una categoría $\Cat$, y que esta es una categoría enriquecida en $\Vat$ si su bifuntor $[-,-]$ admite un levantamiento a $\Vat$ en el siguiente sentido:

\ \ \ \ \ \ \ \ \ \ \ \ \ \ \ \ \ \ \ \ 	 \xymatrix@C=4pc{ & \Vat \ar[d]^{[I,-]} \\
	 				\Cat \otimes \Cat^{op} \ar@{.>}[ru]^{Hom(-,-)} \ar[r]^>>>>>>>{[-,-]} & \Ens}

\begin{remark} \label{Vautoenrique}
	Que una categoría tensorial $\Vat$ tenga hom internos equivale a que esté enriquecida en $\Vat$.
\end{remark}

	 	 Las categorías enriquecidas (o $\Vat$-categorías) son sólo la "`puerta de entrada"' al mundo enriquecido. Funtores, transformaciones naturales, límites, adjunciones pueden enriquecerse sobre una categoría tensorial. La sección "`terminology"' de \cite{tesisdubuc} es una buena introducción compacta al tema. Se tiene en particular la siguiente "`versión enriquecida"' de la propiedad \ref{coendaend}:
	 	 
\begin{\prop} \label{coendaendenrique}
	Sean $\Cat, \Bat$ dos categorías, con $\Bat$ enriquecida en $\Vat$, $T : \Cat^{op} \times \Cat \rightarrow \Bat$ un funtor y $B$ un objeto de $\Bat$. Si existen el end $\int_\Cat T(C,C)$ y/o el coend $\int^\Cat T(C,C)$, entonces existe también el respectivo miembro derecho de las igualdades siguientes y vale que
	$$Hom(B,\int_\Cat T(C,C)) = \int_\Cat Hom(B,T(C,C)) \hbox{\ \ \ y/o \ \ \ } $$
	$$Hom(\int^\Cat T(C,C),B) = \int_\Cat Hom(T(C,C),B)$$
	Notemos que el end y el coend de la izquierda se calculan sobre $\Bat$ y los de la derecha sobre $\Vat$, considerando para estos los bifuntores $Hom(B,T(-,-))$ y $Hom(T(-,-),B)$.
\end{\prop}	 	 
	
\begin{remark} \label{coendaendenV}
	Por la observación \ref{Vautoenrique}, la propiedad anterior también se aplica si $\Vat$ es una categoría tensorial con hom internos y $T : \Cat^{op} \times \Cat \rightarrow \Vat$ un funtor.
\end{remark}
	 
\begin{\de} \label{cocat}
	Una co-categoría enriquecida en $\Vat$ es una categoría enriquecida en $\Vat^{op}$
\end{\de}

	Sean $\Cat$ una categoría pequeña, $\Vat$ una categoría tensorial cocompleta, $\Bat$ una categoría enriquecida en $\Vat$ y $F, G: \Cat \rightarrow \Bat$ dos funtores. Podemos armar entonces el bifuntor $Hom(F-,G-): \Cat^{op} \otimes \Cat \rightarrow \Vat$ y considerar el end $\int_\Cat Hom(FC,GC)$ en $\Vat$, al que llamaremos $Nat(F,G)$ pues es un objeto de $\Vat$ que representa a las transformaciones naturales entre $F$ y $G$ en el siguiente sentido (recordemos que el funtor $[I,-]$ nos permite recuperar la categoría subyacente): usando para la segunda igualdad la propiedad \ref{coendaend}, se tiene
	
	$$[I,Nat(F,G)] = [I,\int_\Cat Hom(FC,GC)] = \int_\Cat [I,Hom(FC,GC)] = $$ $$ = \int_\Cat [FC,GC] = nat[F,G]$$

Es decir, el conjunto subyacente de $Nat(F,G)$ es $nat[F,G]$.
	
\subsection{Producto tensorial entre funtores}

El segundo ejemplo, este de un coend, es la extensión del producto tensorial entre $R$-módulos al producto tensorial entre dos funtores, uno contravariante y uno covariante, sobre cualquier categoría tensorial.

Recordemos lo siguiente:
\begin{itemize}
	\item Si $A$ es un $R$-módulo a derecha y $B$ es un $R$-módulo a izquierda, se define su producto tensorial $A \otimes_R B$ como $R$-módulos identificando en el producto tensorial entre $A$ y $B$ como grupos abelianos a $a \cdot r \otimes b$ con $a \otimes r \cdot b$.
	Queda definido también como el colímite en $Ab$ del diagrama
	\begin{equation} \label{rmodulos}
	\xymatrix { A \otimes B \ar[r]^{A \otimes (r \cdot - ) } \ar[d]^{(- \cdot r ) \otimes B } & A \otimes B \ar[ddr] \\
							A \otimes B	\ar[drr] \\
							& & A \otimes_R B }
	\end{equation}
	indexado por los elementos $r$ de $R$.
	\item Todo monoide $M$ se puede interpretar como una categoría con un solo objeto, al que también llamamos $M$, donde las flechas son los elementos $m \in M$ y la composición es el producto de $M$.
	\item Una categoría $\Cat$ es aditiva si $\forall X,Y \in \Cat$ los conjuntos $[X,Y]$ admiten cada uno una estructura de grupo abeliano tal que la composición es bilineal, es decir se tiene $$[Y,Z] \otimes [X,Y] \stackrel{\circ}{\rightarrow} [X,Z]$$
	 Un funtor $F$ entre categorías aditivas es aditivo si visto como función entre los conjuntos $[X,Y]$ y $[FX,FY]$ es un morfismo de grupos.
\end{itemize}

Luego, de forma similar a como se puede hacer para las representaciones de grupos, si a un anillo $(R,+,\cdot)$ lo miramos como monoide $(R,\cdot)$ y lo interpretamos como categoría $\Rat$, los axiomas de su suma se corresponden con que $\Rat$ sea una categoría aditiva. Una vez hecho esto, notamos que un $R$-módulo a derecha $A$ es un funtor aditivo $\Rat^{op} \rightarrow Ab$, que manda el objeto simbólico de $\Rat$ a $A$, y análogamente un $R$-módulo a izquierda $B$ es un funtor aditivo $\Rat \rightarrow Ab$, que manda el objeto simbólico de $\Rat$ a $B$ (poder multiplicar elementos de $A$ (o $B$) con escalares de $R$ se corresponde con tener para cada flecha $r$ de $\Rat$ un endomorfismo de $A$ (o $B$), y los axiomas de módulo se corresponden con la aditividad del funtor). 

El coproducto $A \otimes B$ tiene entonces una estructura de $R$-módulo a derecha inducida por la estructura de $A$ ($(a \otimes b) \cdot r = (a \cdot r) \otimes b)$ y análogamente una estructura de $R$-módulo a izquierda inducida por la estructura de $B$. Esto equivale a que $\Rat^{op} \times \Rat \rightarrow Ab$, $(\overline{R},R) \mapsto A \otimes B$ es un funtor. 

\begin{\prop}
	$\int^\Rat A \otimes B = A \otimes_R B$

\end{\prop}

\begin{proof}
	Como $\Rat$ tiene un solo objeto, la propiedad universal de ser coend coincide con la de ser el límite del diagrama \eqref{rmodulos}.
\end{proof}

Esta igualdad nos permite definir el produto tensorial entre dos funtores, uno contravariante y uno covariante, sobre cualquier categoría tensorial, coincidiendo con el producto usual de $R$-módulos en el caso recién visto. Recordemos que una categoría es cocompleta si tiene todos los colímites (y por lo tanto todos los coends) pequeños.

\begin{\de}
	Sean $\Cat$ una categoría pequeña, $\Vat$ una categoría tensorial cocompleta y $A : \Cat^{op} \rightarrow \Vat$, $B : \Cat \rightarrow \Vat$ dos funtores. Se define el producto tensorial entre $A$ y $B$ sobre $\Cat$ como el objeto de $\Vat$ $$ A \otimes_\Cat B = \int^\Cat AC \otimes BC  $$
\end{\de}

	Definimos también el producto tensorial externo entre los funtores covariantes $B: \Cat \rightarrow \Vat$ e $Y: \Dat \rightarrow \Vat$ como un nuevo funtor covariante \linebreak $B \ \underline{\otimes} \  Y : \Cat \times \Dat \rightarrow \Vat$ que se define en los objetos como \linebreak $(B \ \underline{\otimes} \  Y) (C,D) = BC \otimes YD$ y en las flechas como \linebreak $(B \ \underline{\otimes} \  Y) (f,g) = B(f) \otimes Y(g)$. Queda definido también entre funtores contravariantes. La siguiente propiedad relaciona ambos productos.

\begin{\prop} \label{dobleproducto}
	Sean $\Cat$ y $\Dat$ dos categorías pequeñas, $\Vat$ una categoría tensorial cocompleta y $A : \Cat^{op} \rightarrow \Vat$, $B : \Cat \rightarrow \Vat$, $X : \Dat^{op} \rightarrow \Vat$ e $Y : \Dat \rightarrow \Vat$ cuatro funtores. Entonces, si $\Vat$ es simétrica y los funtores $(-)\otimes V : \Vat \rightarrow \Vat$ preservan coends para todo objeto $V$ de $\Vat$, vale la igualdad
	$$ (A \otimes_\Cat B) \ \underline{\otimes} \  (X \otimes_\Dat Y) \cong (A \ \underline{\otimes} \  X) \otimes_{\Cat \times \Dat} (B \ \underline{\otimes} \  Y)$$
\end{\prop}

\begin{proof}
	Como los funtores $(-) \otimes V$, con $V \in \Vat$ preservan coends, tomando $V = \int^\Dat XD \otimes YD$ tenemos
	\begin{eqnarray*}
	 (\int^\Cat AC \otimes BC) \otimes (\int^\Dat XD \otimes YD) & = & \\
	 \int^\Cat \int^\Dat AC \otimes BC \otimes  XD \otimes YD & \cong & \\
	 \int\int^{\Cat \times \Dat} (A  \ \underline{\otimes} \  X)(C,D) \otimes (B  \ \underline{\otimes} \  Y)(C,D) & = & \\
	 (A  \ \underline{\otimes} \  X) \otimes_{\Cat \times \Dat} (B  \ \underline{\otimes} \  Y) 
	\end{eqnarray*}

Para pasar de la segunda a la tercera expresión usamos la simetría de $\Vat$ y la proposición análoga a Fubini que dice que la doble integral de coends se calcula iteradamente (ver \cite{ML}, IX, 8. Iterated Ends and Limits, p.226).	
\end {proof}

Observemos que la condición para $\Vat$ de que los funtores $(-) \otimes V : \Vat \rightarrow \Vat$ preserven coends (y colímites en general) se tiene en particular si $\Vat$ tiene hom internos, pues en ese caso estos funtores tienen adjuntos a derecha.

\begin{\prop} \label{c}
	Sean $\Cat$ una categoría pequeña, $\Vat$ una categoría tensorial cocompleta con hom internos y $A: \Cat^{op} \rightarrow \Vat$ un funtor. Entonces se tiene la adjunción
	
\ \ \ \ \ \ \ \ \ \ \ \ \ \ \ \ \ \ \ \ \ \ \ \	\ \ \ \ \ \ \ \ \ \ \ \ \xymatrix{ {\Vat}^{\Cat} \ar@/^/[r]^{(-) \otimes_\Cat A}_\bot & \Vat \ar@/^/[l]^{Hom(A,-)} }
	
	dada por la biyección entre las flechas
\begin{center}
	\underline{$B \Rightarrow Hom(A,V)$} \\
	$B \otimes_\Cat A \rightarrow V$
\end{center}	
\end{\prop}

\begin{proof}

	Aclaremos primero que $Hom(A,V): \Cat \rightarrow \Vat$ es el funtor $Hom(A,V) (C) = Hom(AC, V)$, $Hom(A,V)(C \stackrel{c}{\rightarrow} C')=A(c)^*$ dada por "`componer primero con $A(c)$ "' (no confundir esta ${}^*$ con la del dual en una categoría rígida).
	
	Veremos entonces la biyección natural entre las flechas $B \otimes_\Cat A \rightarrow V$ y las transformaciones naturales $B \Rightarrow Hom(A,V)$. En efecto, una flecha $B \otimes_\Cat A \stackrel{f}{\rightarrow} V$, como un coend es un di-cono universal, es lo mismo que tener flechas $BC \otimes AC \stackrel{f_C}{\rightarrow} V$ tales que para toda flecha $c: C \rightarrow C'$ en $\Cat$ el diagrama 
	
	\ \ \ \ \ \ \ \ \ \ \ \ \xymatrix { BC \otimes AC' \ar[r]^{BC \otimes A(c)} \ar[d]_{B(c) \otimes AC'} & BC \otimes AC \ar[d]^{f_C} \\
							BC' \otimes AC' \ar[r]^{f_{C'}} & V }
							
	conmuta. Pero por la ley exponencial, esto es lo mismo que tener flechas $BC \stackrel{\tilde{f_C}}{\rightarrow} Hom (AC,V)$. La naturalidad de la $B \stackrel{\tilde{f}}{\Rightarrow} Hom(A,V)$ inducida viene dada por el diagrama
	
	\ \ \ \ \ \ \ \ \ \ \ \ \ \ \ \ \ \ \xymatrix { BC \ar[r]^>>>>>{\tilde{f_C}} \ar[d]_{Bc} & Hom(AC,V) \ar[d]^{A(c)^*}	\\
							BC' \ar[r]^>>>>>{\tilde{f_{C'}}} & Hom(AC',V) }
													
	Debemos ver que la conmutatividad de ambos diagramas es equivalente. Ahora bien, la naturalidad en $A$ de la ley exponencial para \linebreak $(-)\otimes AC \dashv Hom(AC,-)$ nos da la equivalencia
	
\xymatrix { {}  \ar@<-4ex>@{-}[rrrr] & BC \otimes AC' \ar[r]^{BC \otimes A(c)} \ar@/^2pc/[rr]^{\triangle} & BC \otimes AC \ar[r]^{f_C} & V & {} \\
						{} & BC \ar[r]^{\tilde{f_C}} \ar@/_2pc/[rr]_{\nabla} & Hom(AC,V) \ar[r]^{A(c)^*} & Hom(AC',V) & {} }
	 
	mientras que la naturalidad en $B$ de $(-)\otimes AC' \dashv Hom(AC',-)$ nos da 
	
\xymatrix { {}  \ar@<-4ex>@{-}[rrrr] & BC \otimes AC' \ar[r]^{B(c) \otimes AC} \ar@/^2pc/[rr]^{\triangle} & BC' \otimes AC' \ar[r]^{f_{C'}} & V & {} \\
						{} & BC \ar[r]^{B(c)} \ar@/_2pc/[rr]_{\nabla} & BC' \ar[r]^{\tilde{f_{C'}}} & Hom(AC',V) & {} }	

	Luego, la conmutatividad del primer cuadrado equivale a la igualdad de las dos flechas marcadas con $\triangle$, que equivale a la igualdad de las dos flechas marcadas con $\nabla$, que equivale a la conmutatividad del segundo cuadrado.
\end{proof}

\pagebreak
\section{Nat Predual} \label{sec:NatPredual}

\subsection{Definición y predualidad}

A partir de esta sección, $\Vat$ será una categoría tensorial cocompleta con hom internos, $\Vat_0$ una subcategoría tensorial de $\Vat$ en la cual todo objeto tiene un dual a derecha y $\Cat$ será siempre pequeña.

El ejemplo motivador de $\Vat$ y $\Vat_0$, aunque simétrico, son los $K$-e.v. de dimensión finita dentro de la categoría de todos los $K$-e.v. Podemos realizar la siguiente definición (cuya idea original es de Joyal).

\begin{\de} \label{nuestradef}
	Sean $F, G: \Cat \rightarrow \Vat$ dos funtores tales que $G$ toma valores en $\Vat_0$. Definimos $Nat^{\lor}(F,G)$ como el coend $\int^\Cat FC \otimes (GC)^{\wedge}$
\end{\de}
	
	Notemos que el funtor $G$ define el funtor contravariante $G^\wedge$ según \linebreak $G^\wedge(C) = (GC)^\wedge$, $G^\wedge(f) = G(f)^\wedge$ (que se obtiene componiendo este $G$ con el funtor contravariante dado por la dualidad a derecha, también llamado $G$ en la sección \ref{subsec:DualidadGlobal}), que es el que estamos utilizando para armar el coend. 
	
\begin{remark} \label{a}
	$Nat^\lor(F,G)$ resulta igual al producto tensorial entre los funtores $F$ y $G^\wedge$: $$Nat^\lor(F,G) = F \otimes_\Cat G^\wedge$$
\end{remark}

Nos interesa comparar esta definición con la de \cite{JS}, 3, p.432. Allí se define (aunque sin hacer mención al coend) $Nat^\lor(F,G) = \int^\Cat Hom(FC,GC)^{\lor}$, y se muestra en la página 439 un isomorfismo $G^\lor \otimes_\Cat F \cong Nat^\lor(F,G)$. Nosotros realizamos una definición de $Nat^\lor(F,G)$ similar al miembro izquierdo del isomorfismo, pues este existe con mayor generalidad, y mostraremos a continuación que nuestra definición es equivalente, con una hipótesis adicional, a una definición análoga a la de \cite{JS}.

Notemos que aunque $GC$ caiga en $\Vat_0$, no necesariamente $Hom(FC,GC)$ va a tener un dual a derecha, pues si $FC$ no tiene un dual a izquierda no podemos afirmar que $Hom(FC,GC) = GC \otimes FC^\lor$. En cambio, si $FC$ tiene un dual a izquierda entonces
$$ FC \otimes GC^\wedge = {FC^\lor}^\wedge \otimes GC^\wedge = (GC \otimes FC^\lor)^\wedge = Hom(FC,GC)^\wedge $$

Luego tenemos

\begin{remark} \label{compatibJoyal}
	Sean $F, G: \Cat \rightarrow \Vat$ dos funtores tales que $G$ toma valores en $\Vat_0$ y $FC$ tiene un dual a izquierda para todo $C$. Entonces	$$Nat^{\lor}(F,G) = \int^\Cat Hom(FC,GC)^{\wedge}$$
\end{remark}

La idea de la predualidad aparece en la siguiente propiedad, que luego probaremos también para la definición \ref{nuestradef}.

\begin{remark} \label{predualidadJoyal}
	Bajo las hipótesis de la observación anterior, existe el end $Nat(F,G)$ en $\Vat$ y $$Hom(\int^\Cat Hom(FC,GC)^{\wedge},I) = Nat(F,G)$$ 
\end{remark}

\begin{proof}
	Como $Hom(FC,GC)$ tiene un dual a derecha \linebreak $Hom(FC,GC)^\wedge$, este objeto tiene como dual a izquierda a $Hom(FC,GC)$; y como toda dualidad a izquierda viene dada por hom internos tenemos 
	$$Hom(Hom(FC,GC)^{\wedge}, I) = Hom(FC,GC)^{{\wedge}^{\lor}} = Hom(FC,GC)$$
	
	Luego, por la propiedad \ref{coendaendenrique}, tenemos la siguiente cadena de igualdades que nos da además la existencia de $Nat(F,G)$ 
	$$Hom(\int^\Cat Hom(FC,GC)^{\wedge}, I) = \int_\Cat Hom(Hom(FC,GC)^{\wedge}, I) = $$
	$$ \ \ \ \ \ \ \ \ \ \ \ = \int_\Cat Hom(FC,GC) = Nat(F,G) $$
\end{proof}

Veremos ahora que la predualidad también es válida sin la hipótesis de que $FC$ tenga un dual a izquierda.

\begin{\prop} \label{predualidad}
	Sean $F, G: \Cat \rightarrow \Vat$ dos funtores tales que $G$ toma valores en $\Vat_0$. Entonces existe el end $Nat(F,G)$ en $\Vat$ y $$Hom(Nat^\lor(F,G),I) = Nat(F,G)$$ 
\end{\prop}

\begin{proof}
	Tenemos la siguiente biyección
	
\vspace{2ex}	
	
\xymatrix@C=1pc @R=0pc { & A \otimes FC \otimes GC^\wedge \ar[r] & B \\
						 \ar@{-}[rrr] & & & \hbox{adjunción propiedad \ref{CadjC*}} \\ 
						  			 & A \otimes FC \ar[r] & B \otimes GC \\
						 \ar@{-}[rrr] & & & \hbox{ley exponencial para } FC \\ 			 
						  			 & A \ar[r] & Hom(FC, B \otimes GC) }

\vspace{2ex}
	
	que nos dice que $Hom(FC \otimes GC^\wedge,-) = Hom(FC, (-) \otimes GC)$, y en particular que $Hom(FC \otimes GC^\wedge,I) = Hom(FC,GC)$. Luego, similarmente a como demostramos la observación \ref{predualidadJoyal}, tenemos
	
	$$Hom(Nat^\lor(F,G),I) = Hom(\int^\Cat FC \otimes GC^\wedge, I) = $$
	$$ = \int_\Cat Hom(FC \otimes GC^{\wedge}, I) = \int_\Cat Hom(FC,GC) = Nat(F,G) $$
\end{proof}

\begin{remark}
	La propiedad anterior \textbf{no} dice que $Nat(F,G)$ sea un dual a izquierda de $Nat^\lor(F,G)$ en el sentido de la sección \ref{sec:Dualidad}. Sin embargo sí nos dice (pues de haber dualidad esta viene dada por el hom interno) que, si $Nat^\lor(F,G)$ tiene un dual a izquierda, este debe ser $Nat(F,G)$; o equivalentemente, que si $Nat(F,G)$ tiene un dual a derecha, este debe ser $Nat^\lor(F,G)$.
\end{remark}

\begin{remark} \label{casonoenrique}
	Aplicando el funtor $[I,-]$ se obtiene de la propiedad anterior la igualdad de conjuntos $[Nat^{\lor}(F,G),I] = nat[F,G]$.
\end{remark}

	Definimos también el producto tensorial interno entre los funtores covariantes $F_1, F_2: \Cat \rightarrow \Vat$ como un nuevo funtor covariante $F_1 \otimes F_2 : \Cat \rightarrow \Vat_0$ que se define en los objetos como $(F_1 \otimes F_2) (C) = F_1(C) \otimes F_2(C)$ y en las flechas como $(F_1 \otimes F_2) (f) = F_1(f) \otimes F_2(f)$. No debe confundirse este producto con $\underline{\otimes}$ ni con $\otimes_\Cat$, si bien se tiene la igualdad $\otimes = \underline{\otimes} \circ \delta$, donde $\delta: \Cat \rightarrow \Cat \times \Cat$ es la diagonal.

Observemos ahora que la igualdad $$Hom(GC^\wedge, V) = V \otimes {GC^\wedge}^\lor = V \otimes GC$$ nos da la igualdad entre los funtores

\begin{equation} \label{b}
	Hom(G^\wedge,V) = V \otimes G
\end{equation}

Damos a continuación una adjunción para $Nat^\lor(-,G)$ análoga a la de la propiedad \ref{c}, y que es consecuencia directa de esta propiedad y de \eqref{b}:

\begin{\prop} \label{adjuncionpredual}
	Sean $\Cat$ una categoría, y $G : \Cat \rightarrow \Vat$ un funtor que toma valores en $\Vat_0$. Se tiene la siguiente adjunción
	
		\ \ \ \ \ \	\ \ \ \ \ \ \ \ \ \ \ \ \ \ \ \ \ \ \ \ \ \ \ \ \ \ \ \ \ \ \ \ \ \ \ \ \ \ \ \ \xymatrix { {\Vat}^{\Cat} \ar@/^/[r]^{Nat^{\lor}(-,G)}_\bot & \Vat \ar@/^/[l]^{(-) \otimes G} }
	
	dada por la biyección entre las flechas
	
\begin{center}
	\underline{$Nat^\lor(F,G) \rightarrow V$} \\
	$F \Rightarrow V \otimes G$
\end{center}
	
\end{\prop}

\begin{proof} 

	Verificamos primero que $Nat^{\lor}(-,G)$ es un funtor. Dada una flecha $\theta: F \Rightarrow H$ en $\Vat^\Cat$, debemos construir $$Nat^\lor(F,G) \stackrel{Nat^\lor(-,G)(\theta)}{\longrightarrow} Nat^\lor(H,G)$$ Tenemos las flechas $FC \otimes GC^\wedge \stackrel{\theta_C \otimes (GC)^\wedge}{\longrightarrow} HC \otimes (GC)^\wedge \stackrel{\lambda_C}{\longrightarrow} Nat^\lor(H,G)$, para que induzcan la flecha $Nat^\lor(F,G) \stackrel{Nat^\lor(-,G)(\theta)}{\longrightarrow} Nat^\lor(H,G)$ debemos ver que para toda $f: C \rightarrow C'$ en $\Cat$ el diagrama
	
\vspace{1ex}
	
\xymatrix @C=1.4pc { & FC \otimes (GC)^\wedge \ar[rr]^{\theta_C \otimes (GC)^\wedge} & & HC \otimes (GC)^\wedge \ar[dr]^{\lambda_C} \\
					FC \otimes (GC')^\wedge \ar[ur]^{FC \otimes G(f)^\wedge} \ar[dr]_{F(f) \otimes (GC')^\wedge \ \ \ } & & & & Nat^\lor(H,G) \\
					  & FC' \otimes (GC')^\wedge \ar[rr]^{\theta_{C'} \otimes (GC')^\wedge} & & HC' \otimes (GC')^\wedge \ar[ur]_{\lambda_{C'}} }

\vspace{1ex}
	
	conmuta. Para eso intercalamos en el diagrama $HC \otimes (GC')^\wedge$ y obte- \linebreak nemos

\vspace{1ex}	

\xymatrix @C=0pc { & FC \otimes (GC)^\wedge \ar[rr]^{\theta_C \otimes (GC)^\wedge} & & HC \otimes (GC)^\wedge \ar[dr]^{\lambda_C} \\
					FC \otimes (GC')^\wedge \ar[ur]^{FC \otimes G(f)^\wedge} \ar[dr]_{F(f) \otimes (GC')\wedge \ \ \ } \ar[rr]^{\theta_C \otimes (GC')^\wedge} & & HC \otimes (GC')^\wedge \ar[ur]_{HC \otimes G(f)^\wedge} \ar[dr]^{H(f) \otimes (GC')^\wedge} & & Nat^\lor(H,G) \\
					  & FC' \otimes (GC')^\wedge \ar[rr]^{\theta_{C'} \otimes (GC')^\wedge} & & HC' \otimes (GC')^\wedge \ar[ur]_{\lambda_{C'}} }

\vspace{1ex}

	El rombo de la derecha conmuta por definición de coend, el paralelogramo superior es el producto tensorial de dos paralelogramos trivialmente conmutativos, y el inferior es el cuadrado que expresa la naturalidad de $\theta$ en la flecha $f$ multiplicado a derecha por $(GC')^\wedge$. Dejamos como ejercicio verificar que $Nat^\lor(-,G)$ respeta composiciones e identidades.
	
	Luego, para ver la adjunción, se obtiene la biyección natural entre las flechas $Nat^\lor(F,G) \rightarrow V$ y las flechas $F \Rightarrow V \otimes G$. Por la observación \ref{a}, las flechas $Nat^\lor(F,G) \rightarrow V$ son las flechas $F \otimes_\Cat G^\wedge \rightarrow V$ , que por la propiedad \ref{c} están en biyección natural con las flechas $F \Rightarrow Hom(G^\wedge, V)$, que por \eqref{b} son las flechas $F \Rightarrow V \otimes G$.
\end{proof}

\begin{remark}
	La biyección obtenida, en el caso $V=I$, nos da la biyección entre los conjuntos $[Nat^\lor(F,G),I]$ y $nat[F,G]$ como está expresada en la observación \ref{casonoenrique}.
\end{remark}

\subsection{Coevaluación}

	Así como la counidad de la ley exponencial es la evaluación, la unidad de la adjunción de la propiedad anterior es una transformación natural \linebreak $F \stackrel{\eta}{\Rightarrow} Nat^\lor (F,G) \otimes G$ a la que llamamos coevaluación.
	
	Nos interesa entender un poco quién es $\eta$. Como ya mencionamos anteriormente, en cualquier adjunción $\eta$ es la flecha que corresponde a $id_{Nat^\lor(F,G)}$ en la biyección entre las flechas $Nat^\lor(F,G) \rightarrow Nat^\lor(F,G)$ y \linebreak $F \Rightarrow Nat^\lor(F,G) \otimes G$. Por lo tanto, para obtener una descripción más precisa de $\eta$ debemos perseguir a $id_{Nat^\lor(F,G)}$ a lo largo de todas las biyecciones entre flechas con las que hemos armado esta última adjunción, que son las de la demostración de la propiedad \ref{c} y la igualdad \eqref{b}.

\vspace{1ex}

	\underline{$Nat^\lor(F,G) \stackrel{id}{\rightarrow} Nat^\lor(F,G)$}
		
	\underline{$FC \otimes GC^\wedge \stackrel{\lambda_C}{\rightarrow} Nat^\lor(F,G)$ cumpliendo trivialmente propiedad del coend}
	
	\underline{$FC \stackrel{{\eta}_C}{\rightarrow} Hom(GC^\wedge,Nat^\lor(F,G))$ natural}
	
	\underline{$F \stackrel{\eta}{\Rightarrow} Hom(G^\wedge, Nat^\lor(F,G))$}
	
	$F \stackrel{\eta}{\Rightarrow} Nat^\lor(F,G) \otimes G$
	
\vspace{1ex}

	Aquí observamos entonces que $\eta_C$ es la flecha que corresponde a $\lambda_C$ (la inclusión en el coend) vía la ley exponencial. Ahora, como $GC^\wedge$ tiene a $GC$ como dual a izquierda, tenemos que su hom interno viene dado por esta dualidad, y que la biyección de la ley exponencial se calcula explícitamente con el $\varepsilon$ y $\eta$ de la dualidad como fue descripto en \eqref{leyexpcondual}. Luego obtenemos que ${\eta}_C$ es la flecha 
	
\begin{equation} \label{eta}
	FC \stackrel{FC \otimes \eta}{\rightarrow} FC \otimes GC^\wedge \otimes GC \stackrel{\lambda_C \otimes GC}{\rightarrow} Nat^\lor(F,G) \otimes GC
\end{equation}	

\subsection {$Nat^\lor$ da una co-categoría enriquecida en $\Vat$} \label{sec:cocat}

	Explicitemos ahora qué es tener una co-categoría enriquecida como fue dicho en la definición \ref{cocat}

\begin{remark}
	Sea $\Vat$ una categoría tensorial. Una co-categoría enriquecida en $\Vat$ es tener
\begin{itemize}
	\item una familia $\Cat$ de objetos
	\item para todo par de objetos $C,D$ de $\Cat$ un objeto $Hom^\lor(C,D)$ de $\Vat$
	\item para toda terna de objetos $C,D,E$ de $\Cat$ una flecha en $\Vat$ coasociativa $Hom^\lor(C,E) \stackrel{\triangle}{\rightarrow} Hom^\lor(C,D) \otimes Hom^\lor(D,E)$ 
	\item para todo objeto $C$ una flecha $Hom^\lor(C,C) \rightarrow I$ en $\Vat$ neutro para $\triangle$ a ambos lados.
\end{itemize}
\end{remark}

	El objetivo ahora es ver que la categoría de los funtores $F: \Cat \rightarrow \Vat_0$ obtiene con $Nat^\lor$ una estructura de co-categoría enriquecida en $\Vat$.

	Para eso, dados $F,G,H : \Cat \rightarrow \Vat_0$ funtores queremos construir una flecha $Nat^\lor (F,H) \stackrel{\triangle}{\rightarrow} Nat^\lor(F,G) \otimes Nat^\lor (G,H)$ dual a la composición (en verdad $F$ puede tomar valores en $\Vat$). 
	
	Análogamente a como se puede construir la flecha composición con la ley exponencial "`evaluando dos veces"', construimos la transformación natural
	$$\alpha: F \stackrel{\eta}{\Rightarrow} Nat^\lor (F,G) \otimes G \stackrel{Nat^\lor (F,G) \otimes \eta}{\Longrightarrow} Nat^\lor (F,G) \otimes Nat^\lor(G,H) \otimes H $$
	
	(notar que cada $\eta$ corresponde a una adjunción distinta) que por la biyección entre las flechas de la adjunción de la proposición \ref{adjuncionpredual} (usada en la otra dirección que antes) nos da la flecha $$ Nat^\lor (F,H) \stackrel{\triangle}{\rightarrow} Nat^\lor(F,G) \otimes Nat^\lor (G,H) $$	a la que llamamos cocomposición. 
	
	Para entender quién es $\triangle$, debemos recorrer las biyecciones en el otro sentido. Comenzamos con las flechas $\alpha_C$, que obtenemos de \eqref{eta}:

\vspace{1ex}
	
\xymatrix{FC \ar[r]^>>>>>{FC \otimes \eta} &
					FC \otimes GC^\wedge \otimes GC \ar[r]^>>>>>{\lambda_C \otimes GC} &
					Nat^\lor(F,G) \otimes GC \ar[rrr]^>>>>>>>>>>>>>>>{Nat^\lor(F,G) \otimes GC \otimes \eta} & & & \cdots}

\xymatrix{ \cdots Nat^\lor(F,G) \otimes GC \otimes HC^\wedge \otimes HC \ar@/_/[r]_{Nat^\lor(F,G) \otimes \lambda_C \otimes HC} &
					Nat^\lor(F,G) \otimes Nat^\lor(G,H) \otimes HC }

\vspace{1ex}
					
Y ahora, al mirar la otra $\theta$ de la ley exponencial, vemos que debemos multiplicar esta flecha a derecha por $HC^\wedge$ y componerla luego con $Nat^\lor(F,G) \otimes Nat^\lor(G,H) \otimes \varepsilon$. La forma más sencilla de trabajar con esta flecha es con el cálculo de ascensores, de la siguiente forma:

\ \ \ \ \ \ \ \ \ \ \ \ \ \ \ \ \ \ \ \ \xymatrix@C=-2pc{ FC \ar@2{-}[d] & & & \ar@{-}[rd] \ar@{-}[ld] \ \ \ \ \ \ar@{}[d]|{\eta} \ \ \ \ \ & & & & & & & HC^\wedge \ar@2{-}[d] & \\
								 FC \ar@{-}[rd] & \ar@{}[d]|{\lambda_C} & GC^\wedge \ \ \ \ar@{-}[ld] & & GC \ar@2{-}[d] & & & & & & HC^\wedge \ar@2{-}[d] & \\	
								 & Nat^\lor(F,G) \ar@2{-}[d] & & & GC \ar@2{-}[d] & & & \ar@{-}[rd] \ar@{-}[ld] \ \ \ \ \ \ar@{}[d]|{\eta} \ \ \ \ \ & & & HC^\wedge \ar@2{-}[d] & \\							 
								 & Nat^\lor(F,G) \ar@2{-}[d] & & & GC \ar@{-}[rd] & \ar@{}[d]|{\lambda_C} & HC^\wedge \ar@{-}[ld] & & HC \ar@2{-}[d] & & HC^\wedge \ar@2{-}[d] & \\				
								 & Nat^\lor(F,G) \ar@2{-}[d] & & & & Nat^\lor(G,H) \ar@2{-}[d] & & & HC \ar@{-}[rd] & \ \ \ \ \ \ar@{}[d]|{\varepsilon} \ \ \ \ \ & HC^\wedge \ar@{-}[ld] & \\							 
								 & Nat^\lor(F,G) & & & & Nat^\lor(G,H) & & & & & & }

Moviendo ahora flechas hacia arriba y hacia abajo obtenemos

\xymatrix@C=-2.5pc{ FC \ar@2{-}[d] & & & \ar@{-}[rd] \ar@{-}[ld]  \ \ \ \ \ \ \ \ \ar@{}[d]|{\eta} \ \ \ \ \ \ \ \ & & & & \ar@{-}[rd] \ar@{-}[ld] \ \ \ \ \ \ \ \ \ar@{}[d]|{\eta} \ \ \ \ \ \ \ \ & & & HC^\wedge \ar@2{-}[d] & & 				FC \ar@2{-}[d] & & & \ar@{-}[rd] \ar@{-}[ld] \ \ \ \ \ \ \ \ \ar@{}[d]|{\eta} \ \ \ \ \ \ \ \ & & & HC^\wedge \ar@2{-}[d]	\\	
								 FC \ar@2{-}[d] & & GC^\wedge \ar@2{-}[d] & & GC \ar@2{-}[d] & & HC^\wedge \ar@2{-}[d] & & HC \ar@{-}[rd] & \ \ \ \ \ \ \ \ \ar@{}[d]|{\varepsilon} \ \ \ \ \ \ \ \ & HC^\wedge \ar@{-}[ld] & \ \ \ \ \ \ \ \ \ \ = \ \ \ \ \ \ \ \ \ \ &					FC \ar@{-}[rd] & \ar@{}[d]|{\lambda_C} & GC^\wedge \ar@{-}[ld] & & GC  \ar@{-}[rd] & \ar@{}[d]|{\lambda_C} & HC^\wedge \ar@{-}[ld]	\\							 			 FC \ar@{-}[rd] & \ar@{}[d]|{\lambda_C} & GC^\wedge \ar@{-}[ld] & & GC \ar@{-}[rd] & \ar@{}[d]|{\lambda_C} & HC^\wedge \ar@{-}[ld] & & & & & & 				& Nat^\lor(F,G) & & & & Nat^\lor(G,H) 	\\						 
								 & Nat^\lor(F,G)					& 											& & & Nat^\lor(G,H) & & & & & & }

\vspace{1ex}

La igualdad es válida por la primer igualdad triangular de la dualidad $HC \dashv HC^\wedge$. Luego, obtenemos que las flechas $\triangle_C$ que inducen $\triangle$ son

\vspace{1ex}

\xymatrix{FC \otimes HC^\wedge \ar@/^/[r]^>>>{FC \otimes \eta \otimes HC^\wedge} & FC \otimes GC^\wedge \otimes GC \otimes HC^\wedge \ar[r]^{\lambda_C \otimes \lambda_C} & Nat^\lor(F,G) \otimes Nat^\lor(G,H) }

\vspace{1ex}
						
La siguiente propiedad expresa la coasociatividad de la cocomposición.

\begin{\prop}
Para todos $F,G,H,I: \Cat \rightarrow \Vat_0$, el siguiente diagrama conmuta

\xymatrix{ Nat^\lor(F,G) \otimes Nat^\lor(G,I) \otimes Nat^\lor(I,H) & Nat^\lor(F,I) \otimes Nat^\lor(I,H) \ar[l]_<<<<{\triangle \otimes 1} \\
					 Nat^\lor(F,G) \otimes Nat^\lor(G,H) \ar[u]^{1 \otimes \triangle} & Nat^\lor(F,H) \ar[l]^{\triangle} \ar[u]_{\triangle} }
\end{\prop}

\begin {proof}
Para probarla ponemos dentro de este cuadrado otro cuadri-\\látero de la siguiente forma

\tiny

\xymatrix @C=-1.7pc { Nat^\lor(F,G) \otimes Nat^\lor(G,I) \otimes Nat^\lor(I,H) & & & Nat^\lor(F,I) \otimes Nat^\lor(I,H) \ar[lll]_{\triangle \otimes 1} \\
						& FC \otimes GC^\wedge \otimes GC \otimes IC^\wedge \otimes IC \otimes HC^\wedge \ar[ul]^{\lambda_C \otimes \lambda_C \otimes \lambda_C}  \\
							& &	FC \otimes IC^\wedge \otimes IC \otimes HC^\wedge \ar[ul]_{FC \otimes \eta \otimes IC^\wedge \otimes IC \otimes HC^\wedge} \ar[uur]^{\lambda_C \otimes \lambda_C} 	\\
						& FC \otimes GC^\wedge \otimes GC \otimes HC^\wedge \ar[uu]^{FC \otimes GC^\wedge \otimes GC \otimes \eta \otimes HC^\wedge} \ar[dl]^{\lambda_C \otimes \lambda_C} & FC \otimes HC^\wedge \ar[u]^{FC \otimes \eta \otimes HC^\wedge} \ar[l]^<<<<{FC \otimes \eta \otimes HC^\wedge} \ar[dr]^{\lambda_C} \\ 
					 Nat^\lor(F,G) \otimes Nat^\lor(G,H) \ar[uuuu]^{1 \otimes \triangle} & & & Nat^\lor(F,H) \ar[lll]^{\triangle} \ar[uuuu]_{\triangle} }

\normalsize

Luego los cuatro cuadriláteros exteriores conmutan por la descripción que hicimos de $\triangle$, y el cuadrilátero interior también lo hace por la funtorialidad del producto tensorial.
\end{proof}

Buscamos ahora una counidad $Nat^\lor(F,F) \stackrel{\varepsilon}{\rightarrow} I$, la obtenemos de la siguiente forma: tenemos las flechas $FC \otimes FC^\wedge \stackrel{\varepsilon}{\rightarrow} I$ correspondientes a la dualidad de $FC$. Para obtener a partir de ellas una flecha $Nat^\lor(F,F) \stackrel{\varepsilon}{\rightarrow} I$, debemos verificar que para toda flecha $c: C \rightarrow C'$ los cuadrados

\ \ \ \ \ \ \ \ \ \ \ \ \xymatrix@C=3pc { FC \otimes FC'^\wedge \ar[r]^{FC \otimes F(c)^\wedge} \ar[d]_{F(c) \otimes FC'^\wedge} & FC \otimes FC^\wedge \ar[d]^{\varepsilon} \\
							FC' \otimes FC'^\wedge \ar[r]^{\varepsilon} & I }
							
	conmutan. Esto ocurre por la propiedad \ref{wedgeswitch}.
	
	Ver que $\varepsilon$ es neutro para $\triangle$ es verificar que el siguiente triángulo conmuta (el otro es análogo).
	
\ \ \ \ \ \ \ \ \ \ \ \ \xymatrix {  Nat^\lor(G,F) \otimes Nat^\lor(F,F) \ar[rrr]^{Nat^\lor(G,F) \otimes \varepsilon} & & & Nat^\lor(F,F) \otimes I \\
							Nat^\lor(G,F) \ar[u]^{\triangle} \ar[urrr]_{ \cong } }

\vspace{1ex}
	
	Similarmente a como venimos haciendo, metemos en el triángulo la flecha $GC \otimes \eta \otimes FC^\wedge$ y obtenemos
	
 \xymatrix@C=1pc {  Nat^\lor(G,F) \otimes Nat^\lor(F,F) \ar[rr]^{Nat^\lor(G,F) \otimes \varepsilon} & & Nat^\lor(G,F) \otimes I \\
							& GC \otimes FC^\wedge \otimes FC \otimes FC^\wedge \ar[ul]_{\lambda_C \otimes \lambda_C} \ar[ur]^{\lambda_C \otimes \varepsilon} \\
							& GC \otimes FC^\wedge \ar[u]^{GC \otimes \eta \otimes FC^\wedge}	\ar[dl]^{\lambda_C} \ar@/_2pc/[uur]^{\cong \circ \lambda_C} \\
							Nat^\lor(G,F) \ar[uuu]^{\triangle} \ar@/_6pc/[uuurr]_{\cong} }

\begin{verbatim}

\end{verbatim}

donde el trapecio de la izquierda conmuta igual que antes, los "`triángulos"' pequeños de arriba y abajo conmutan trivialmente y el de la derecha lo hace pues la composición $(\lambda_C \otimes \varepsilon) \circ (GC \otimes \eta \otimes FC^\wedge)$ es, bajando la flecha $\lambda_C$
	
	\ \ \ \ \ \ \ \ \ \ \ \xymatrix@C=-2.3pc{ 
GC \ar@2{-}[d] & 											& 											& \ar@{-}[ld] \ar@{-}[rd]  \ \ \ \ \  \  \ar@{}[d]|{\eta}    \ \ \ \ \ \ &				 &												& FC^\wedge \ar@2{-}[d]\\
GC \ar@2{-}[d] & 											& FC^\wedge \ar@2{-}[d] & 																				& FC \ar@{-}[rd] &    \ \ \ \ \ \ \ar@{}[d]|{\varepsilon}    \ \ \ \ \ \ & FC^\wedge \ar@{-}[ld]		& \ \ \ \ \ \ \ \ \ \ \ar@{}[d]|{=} \ \ \ \ \ \ \ \ \ \ &  GC \ar@{-}[rd] & \ar@{}[d]|{\lambda_C}& FC^\wedge \ar@{-}[ld] \\	
GC \ar@{-}[rd] & \ar@{}[d]|{\lambda_C}& FC^\wedge \ar@{-}[ld]	& & & &    & & 	& Nat^\lor(G,F)				\\
							   & Nat^\lor(G,F)				& 											& 																				& 							 & 											  &								}

\vspace{2ex}
									  
La igualdad es válida por la segunda igualdad triangular de la dualidad $FC \dashv FC^\wedge$.

\pagebreak	
\section {End Predual} \label{sec:EndPredual}

\subsection{$End^\lor(F)$ es una coálgebra}

	Sean $\Cat$ una categoría, y $F : \Cat \rightarrow \Vat_0$ un funtor.  Notemos $$End^\lor (F) = Nat^\lor(F,F)$$
	
		Así como en una categoría $\Cat$ (o en una categoría enriquecida en $\Vat$) los conjuntos $[C,C]$ (o los objetos $Hom(C,C)$) obtienen una estructura de monoide (respecto del producto tensorial de $\Vat$) con la composición, la cocomposición en el caso $F=G=H$ nos da una flecha \linebreak $End^\lor (F) \stackrel{\triangle}{\rightarrow} End^\lor (F) \otimes End^\lor (F)$ que enriquece a $End^\lor(F)$ con una estructura de coálgebra en $\Vat$ cuya counidad es $End^\lor (F) \stackrel{\varepsilon}{\rightarrow} I$.
	
Antes de seguir avanzando, ya podemos construir un morfismo co- \linebreak rrespondiente al del enunciado 1 de Galois.

\begin{\prop} \label{morfdecomp}
	Sea $C$ una coálgebra en $\Vat$ y $U: Comod_0C \rightarrow \Vat_0$ el funtor de olvido, donde $Comod_0C$ es la categoría de los $C$-comódulos que están en $\Vat_0$. Entonces tenemos un morfismo de coálgebras $\tilde{\rho}: End^\lor (U) \rightarrow C$
\end{\prop}

\begin{remark} \label{obs1}
	Más adelante veremos que en el caso $\Vat = Vec_K$, $\Vat_0 = Vec_K^{<\infty}$, el morfismo resulta un isomorfismo de $K$-coálgebras.
\end{remark}

\begin{proof}
	Cada $C$-comódulo $M$ tiene un coproducto escalar \linebreak $\rho_M : M \rightarrow M \otimes C$. Juntándolos obtenemos una transformación natural $\rho: U \Rightarrow U \otimes C$, interpretando a $C$ como el funtor constante $C \in \Vat_0$ (la naturalidad de $\rho$ para una $f: M \rightarrow M'$ coincide con la condición pedida para que $f$ sea morfismo de $C$-comódulos). 
	
	Vía la biyección obtenida para demostrar la propiedad \ref{adjuncionpredual}, obtenemos una flecha $\tilde{\rho}: End^\lor(U) \rightarrow C$ en $\Vat$. Debemos verificar que $\tilde{\rho}$ es un morfismo de coálgebras.
	
	Para entender qué hace $\tilde{\rho}$, debemos seguir a la flecha $\rho$ a través de las distintas biyecciones naturales entre flechas:
	
	\xymatrix @R=0pc { & U \ar@2[r]^{\rho} & U \otimes C \ar@2{~}[dd] \\
						 \ar@{-}[rrr] & & & {\eqref{b}} \\ 
						  			 & U \ar@2[r]^{\rho} & Hom(U^\lor,C) \\
						 \ar@{-}[rrr] & & & \hbox{Propiedad \ref{c} (ley exponencial)} \\ 			 
						  			 & U \otimes_\Cat U^\lor \ar[r]^{\tilde{\rho}} \ar@2{-}[dd] & C \\
						 \ar@{-}[rrr] & & & \hbox{obs \ref{a}} \\						  			 
						 				 & End^\lor(U) \ar[r]^{\tilde{\rho}} & C  }
	
	Luego, recordando cómo es la ley exponencial cuando los hom internos vienen dados por una dualidad, lo cual fue expuesto en \eqref{leyexpcondual} (en este caso tenemos que $M$ es dual a izquierda de $M^\wedge$), tenemos el diagrama

\begin{equation} \label{defderho}
	\xymatrix { & M \otimes M^\wedge \ar[dl]_{\lambda_M} \ar[r]^{\rho_M \otimes M^\wedge} & C \otimes M \otimes M^\wedge \ar[dr]^{C \otimes \varepsilon} \\
						End^\lor(U) \ar[rrr]^{\tilde{\rho}} & & & C }
\end{equation}
						
	Podemos construir entonces el diagrama
	
\xymatrix@C=0.5pc	{ C \otimes C & & & C \ar[lll]_{\triangle} \\
						& C \otimes M \otimes M^\wedge \otimes C \otimes M \otimes M^\wedge \ar[ul]^{C \otimes \varepsilon \otimes C \otimes \varepsilon} & C \otimes M \otimes M^\wedge\ar[ur]^{C \otimes \varepsilon} \\
						& M \otimes M^\wedge \otimes M \otimes M^\wedge \ar[u]^{\rho_M \otimes M^\wedge \otimes \rho_M \otimes M^\wedge} \ar[dl]^{\lambda_M \otimes \lambda_M} & M \otimes M^\wedge \ar[l]^>>>>>>>{M \otimes \eta \otimes M^\wedge} \ar[u]^{\rho_M \otimes M^\wedge} \ar[dr]^{\lambda_M} \\
						End^\lor(U) \otimes End^\lor(U) \ar[uuu]^{\tilde{\rho} \otimes \tilde{\rho}} & & & End^\lor(U) \ar[lll]_{\triangle} \ar[uuu]_{\tilde{\rho}} }
						
	del que nos resta verificar que la parte superior es conmutativa (la parte inferior conmuta por la descripción que hicimos para $\triangle$). En efecto, la composición que va hacia la izquierda y luego sube se corresponde con el diagrama
	
\ \ \ \ \ \ \xymatrix@C=0pc { & & M \ar@2{-}[d] & & & & \ar@{-}[dll] \ar@{}[d]|{\eta} \ar@{-}[drr] & & & & M^\wedge \ar@2{-}[d] & & 
											& & M \ar@{-}[d] \ar@{}[ld]|>>>>>>{\ \rho_M } \ar@{-}[dll] \\
									& & M \ar@{-}[d] \ar@{}[ld]|>>>>>>{\ \rho_M } \ar@{-}[dll] & & M^\wedge \ar@2{-}[d] & & & & M \ar@{-}[d] \ar@{}[ld]|>>>>>>{\ \rho_M } \ar@{-}[dll] & & M^\wedge \ar@2{-}[d] & & 										C \ar@2{-}[d] & & M \ar@2{-}[d] & & & \ar@{-}[dl] \ar@{}[d]|{\eta} \ar@{-}[dr] & & & M^\wedge \ar@2{-}[d] \\
									C \ar@2{-}[d] & & M \ar@{-}[dr] & \ar@{}[d]|{\varepsilon} & M^\wedge \ar@{-}[dl] & & C \ar@2{-}[d] & & M \ar@{-}[dr] & \ar@{}[d]|{\varepsilon} & M^\wedge \ar@{-}[dl] & \ \ \ \ = \ \ \ \ & 			C \ar@2{-}[d] & & M \ar@{-}[dr] & \ar@{}[d]|{\varepsilon} & M^\wedge \ar@{-}[dl] & & M \ar@2{-}[d] & & M^\wedge \ar@2{-}[d] \\
									C & & & & & & C & & & &  		& & 
											C \ar@2{-}[d] & & & & & & M \ar@{-}[d] \ar@{}[ld]|>>>>>>{\ \rho_M } \ar@{-}[dll] & & M^\wedge \ar@2{-}[d] \\
& & & & & & & & & &  		& & 				C \ar@2{-}[d] & & & & C \ar@2{-}[d] & & M \ar@{-}[dr] & \ar@{}[d]|{\varepsilon} & M^\wedge \ar@{-}[dl] \\
& & & & & & & & & &  		& & 				C & & & & C & & &  & }

\vspace{1ex}
			
	donde la igualdad se obtiene subiendo y bajando flechas. Luego, por la segunda igualdad triangular de la dualidad $M \dashv M^\wedge$, el diagrama queda
	
\ \ \ \ \ \ \ \ \ \ \ \ \xymatrix@C=0pc { & & & & M \ar@{-}[d] \ar@{}[dl]|{\rho_M} \ar@{-}[dllll] & & M^\wedge \ar@2{-}[d] & & 			& & & & M \ar@{-}[d] \ar@{}[dl]|{\rho_M} \ar@{-}[dlll] & & M^\wedge \ar@2{-}[d] \\	
									C \ar@2{-}[d] & & & & M \ar@{-}[d] \ar@{}[ld]|>>>>>>{\ \rho_M } \ar@{-}[dll] & & M^\wedge \ar@2{-}[d] & \ \ \ \ \ar@{}[d]|{=} \ \ \ \ & 		& C \ar@{-}[dr] \ar@{}[d]|{\triangle} \ar@{-}[dl] & & & M \ar@2{-}[d] & & M^\wedge \ar@2{-}[d] \\
									C \ar@2{-}[d] & & C \ar@2{-}[d] & &  M \ar@{-}[dr] & \ar@{}[d]|{\varepsilon} & M^\wedge \ar@{-}[dl] & & 			C \ar@2{-}[d] & & C \ar@2{-}[d] & &  M \ar@{-}[dr] & \ar@{}[d]|{\varepsilon} & M^\wedge \ar@{-}[dl]	\\
									C & & C & & & &  		& & C & & C & & & & & & 		}

\vspace{1ex}
									
	donde la igualdad es válida pues $\rho_M$ es morfismo de $C$-comódulos. Este último diagrama se corresponde con la otra composición.
\end{proof}

\subsection{El levantamiento de $F$}
\label{sec:levdeF}	
	
	La coevaluación en este caso $F=G=H$ es una flecha $F \stackrel{\eta}{\Rightarrow} End^\lor (F) \otimes F$ que, como veremos a continuación, para cada objeto $C \in \Cat$, le da a $FC$ una estructura de $End^\lor(F)$-comódulo a izquierda. Verifiquemos que $\eta_C$ respeta la comultiplicación $\triangle$ de $End^\lor(F)$. Para ello debemos ver que el siguiente diagrama, que ya hemos completado en su interior con la descripción de cada una de las flechas, es conmutativo.
	
\small	
\xymatrix@C=-4pc{ End^\lor(F) \otimes 	End^\lor(F) \otimes FC & & & End^\lor(F) \otimes FC \ar[lll]_{End^\lor(F) \otimes \eta_C} \ar[dl]|{End^\lor(F) \otimes FC \otimes \eta_C} \\
						& & End^\lor(F) \otimes FC \otimes FC^\wedge \otimes FC \ar[ull]|{End^\lor(F) \otimes \lambda_C \otimes FC} \\
						& FC \otimes FC^\wedge \otimes FC \otimes FC^\wedge \otimes FC \ar[uul]|{\lambda_C \otimes \lambda_C \otimes FC} \\
						& & FC \otimes FC^\wedge \otimes FC \ar[ul]|{FC \otimes \eta \otimes FC^\wedge \otimes FC} \ar@/_3pc/[uuur]^{\lambda_C \otimes FC} \ar[dll]_{\lambda_C \otimes FC} \\
						End^\lor(F) \otimes FC \ar@/^/[uuuu]^{\triangle \otimes FC} & & & FC \ar[lll]_{\eta_C} \ar@/_3pc/[uuuu]^{\eta_C} \ar[ul]^{FC \otimes \eta} }
\normalsize	

Para ello basta ver que las dos composiciones "`internas"' que comienzan en el $FC$ de abajo a la derecha y llegan al $End^\lor(F) \otimes End^\lor(F) \otimes FC$ de arriba a la izquierda son la misma, y eso se verifica escribiendo estas composiciones con el cálculo de ascensores y subiendo y bajando flechas (lo dejamos como ejercicio).
	
	Notando con $Comod_0(End^\lor (F))$ a la categoría de los $End^\lor(F)$-comódulos de $\Vat_0$, construimos así un "`levantamiento"' de $F$
	
	\ \ \ \ \ \ \ \ \ \ \ \ \ \ \ \ \ \ \ \ \ \ \ \ \xymatrix {  & Comod_0(End^\lor (F)) \ar[d]^{U} \\
							{\Cat} \ar[ur]^{\tilde{F}} \ar[r]^{F} & \Vat_0 }
							
	donde $U$ es el funtor de olvido. La definición en las flechas es \linebreak $\tilde{F}(C \stackrel{f}{\rightarrow} C') = F(f)$, que veremos a continuación que es un morfismo de $End^\lor(F)$-comódulos. Debemos verificar que el cuadrado
	
	\ \ \ \ \ \ \ \ \ \ \ \ \ \ \ \ \ \ \ \ \ \ \ \xymatrix { FC \ar[r]^>>>>>{\eta_C} \ar[d]_{F(f)} & End^\lor(F) \otimes FC \ar[d]^{End^\lor(F) \otimes F(f)} \\
							FC' \ar[r]^>>>>>{\eta_{C'}} & End^\lor(F) \otimes FC' }
							
	conmuta. Por definición de la coevaluación, esto equivale a que conmute el diagrama
	
\ \ \ \ \ \ \ \xymatrix@C=3pc { FC \ar[r]^>>>>>>>>>{FC  \otimes \eta} \ar[d]_{F(f)} & FC \otimes FC^\wedge \otimes FC \ar[r]^{\lambda_C \otimes FC} & End^\lor(F) \otimes FC \ar[d]^{End^\lor(F) \otimes F(f)} \\
							FC' \ar[r]^>>>>>>>>>{FC' \otimes \eta} & FC' \otimes (FC')^\wedge \otimes FC' \ar[r]^>>>>>>>{\lambda_{C'} \otimes FC'} & End^\lor(F) \otimes FC' }
	
	En forma gráfica, queremos ver que 
	
	\begin{equation} \label{qvq}
	\ \ \ \ \ \ \ \ \ \ \ \ \ \ \ \xymatrix@C=-1.5pc { FC \ar@2{-}[d] & & & \ \ \ \ \ \ar@{-}[dr] \ar@{-}[dl] \ar@{}[d]|{\eta} \ \ \ \ \ &									&&	
	FC \ar@<7pt>@{-}'+<0pt,-6pt>[d] \ar@<-7pt>@{-}'+<0pt,-6pt>[d]^{F(f)} \\
											 FC \ar@{-}[dr] & \ar@{}[d]|{\lambda_C} & FC^\wedge \ar@{-}[dl] & & FC \ar@2{-}[d]										& \ \ \ \ \ \ \ \ = \ \ \ \ \ \ \ \ &		
	FC' \ar@2{-}[d] & & & \ar@{-}[dr] \ar@{-}[dl] \ \ \ \ \ \ar@{}[d]|{\eta} \ \ \ \ \ \\
											 & End^\lor(F) \ar@2{-}[d] & & & FC \ar@<7pt>@{-}'+<0pt,-6pt>[d] \ar@<-7pt>@{-}'+<0pt,-6pt>[d]^{F(f)}	&&	
	FC' \ar@{-}[dr] & \ar@{}[d]|{\lambda_{C'}} & (FC')^\wedge \ar@{-}[dl] & & FC' \ar@2{-}[d] \\
											 & End^\lor(F) & & & FC'																																							&&	
	& End^\lor(F) & & & FC' \\}
	\end{equation}
	
	Para ello usaremos la conmutatividad del siguiente rombo que es válida por la definición de coend
	
	\ \ \ \ \ \ \ \ \ \ \ \ \ \ \ \xymatrix{ & FC \otimes FC^\wedge \ar[dr]_{\lambda_C} \\
						FC \otimes (FC')^\wedge \ar[ru]_{FC \otimes F(f)^\wedge} \ar[rd]^{F(f) \otimes (FC')^\wedge} & & End^\lor(F) \\
						 & FC' \otimes (FC')^\wedge \ar[ru]^{\lambda_{C'}} }
	
	En forma gráfica, recordando la definición del dual de una flecha de la sección \ref{subsec:DualidadGlobal}, el rombo nos queda
	
	\ \ \ \ \ \ \ \ \ \ \ \ \ \ \xymatrix@C=-1.5pc{ FC \ar@<7pt>@{-}'+<0pt,-6pt>[d] \ar@<-7pt>@{-}'+<0pt,-6pt>[d]^{F(f)} & & (FC')^\wedge \ar@2{-}[d] 	&& 
	FC \ar@2{-}[d] & & & \ar@{-}[dr] \ar@{-}[dl] \ \ \ \ \ \ar@{}[d]|{\eta} \ \ \ \ \ & & & (FC')^\wedge \ar@2{-}[d] \\
											FC' \ar@{-}[dr] & \ar@{}[d]|{\lambda_{C'}} & (FC')^\wedge \ar@{-}[dl]																&& 
	FC \ar@2{-}[d] & & FC^\wedge \ar@2{-}[d] & & FC \ar@<7pt>@{-}'+<0pt,-6pt>[d] \ar@<-7pt>@{-}'+<0pt,-6pt>[d]^{F(f)} & & (FC')^\wedge \ar@2{-}[d] \\
											& End^\lor(F) & 																																										& \ \ \ \ \ \ \ \ = \ \ \ \ \ \ \ \  & 
	FC \ar@2{-}[d] & & FC^\wedge \ar@2{-}[d] & & FC' \ar@{-}[dr] & \ \ \ \ \ \ar@{}[d]|{\varepsilon} \ \ \ \ \ & (FC')^\wedge \ar@{-}[dl] \\
	&&&& FC \ar@{-}[dr] & \ar@{}[d]|{\lambda_C} & FC^\wedge \ar@{-}[dl] & & & &  \\		
	&&&& & End^\lor(F)  }
								
	Luego, usando esta igualdad en el diagrama de la derecha de \eqref{qvq} (previa subida de $\eta$ y bajada de $F(f)$), este queda igual a
	
	\ \ \ \ \ \ \ \ \ \ \ \ \ \ \ \ \ \ \ \ \ \ \xymatrix@C=-1.5pc{ & & & & & & & \ar@{-}[dr] \ar@{-}[dl] \ \ \ \ \ \ar@{}[d]|{\eta} \ \ \ \ \ & \\
											FC \ar@2{-}[d] & & & \ar@{-}[dr] \ar@{-}[dl] \ \ \ \ \ \ar@{}[d]|{\eta} \ \ \ \ \ & & & (FC')^\wedge \ar@2{-}[d] & & FC' \ar@2{-}[d] \\
											FC \ar@2{-}[d] & & FC^\wedge \ar@2{-}[d] & & FC \ar@<7pt>@{-}'+<0pt,-6pt>[d] \ar@<-7pt>@{-}'+<0pt,-6pt>[d]^{F(f)} & & (FC')^\wedge \ar@2{-}[d] & & FC' \ar@2{-}[d] \\
											FC \ar@2{-}[d] & & FC^\wedge \ar@2{-}[d] & & FC' \ar@{-}[dr] & \ \ \ \ \ \ar@{}[d]|{\varepsilon} \ \ \ \ \ & (FC')^\wedge \ar@{-}[dl] & & FC' \ar@2{-}[d] \\
											FC \ar@{-}[dr] & \ar@{}[d]|{\lambda_C} & FC^\wedge \ar@{-}[dl] & & & & & & FC' \ar@2{-}[d] \\
											& End^\lor(F) & & & & & & & FC' }
	
	que, juntando la $\eta$ y $\varepsilon$ de la dualidad $FC' \dashv (FC')^\wedge$ abajo a la derecha y usando la segunda igualdad triangular de esta dualidad, queda si subimos a $\lambda_C$ y bajamos a $F(f)$ como el diagrama de la izquierda de \eqref{qvq}.	

\begin{remark} \label{obs2}
Más adelante veremos que en el caso $\Vat = Vec_K$, $\Vat_0 = Vec_K^{<\infty}$, si $\Cat$ es abeliana y enriquecida en $Vec_K$, y $F$ es enriquecido, exacto y fiel, entonces $\tilde{F}$ es una equivalencia de categorías. 
\end{remark}

	Enunciamos ahora nuestros siguientes pasos:

\begin{enumerate}
	\item Si $\Cat$ y $F$ son tensoriales, y $\Vat$ es simétrica entonces $End^{\lor}(F)$ es una biálgebra en $\Vat$.
	\item En ese caso, si todo objeto de $\Cat$ tiene un dual a derecha entonces $End^{\lor}(F)$ es una álgebra de Hopf en $\Vat$.
\end{enumerate}

\subsection{$End^\lor$ es un funtor tensorial}

	A partir de esta sección $\Vat$ será simétrica.

	Lo primero que haremos es darle a la construcción $End^\lor$ estructura funtorial. Para ello necesitamos la categoría $Cat / \Vat_0$ de categorías sobre $\Vat_0$. Recordamos que sus objetos son los funtores $\Cat \stackrel{F}{\rightarrow} \Vat_0$ y una flecha $(f,\alpha)$ entre dos objetos $\Cat \stackrel{F}{\rightarrow} \Vat_0$  y $\Dat \stackrel{G}{\rightarrow} \Vat_0$ es un funtor $\Cat \stackrel{f}{\rightarrow} \Dat$ tal que el triángulo 
	\xymatrix @C=0.5pc { {\Cat} \ar[rr]^{f} \ar[dr]_{F} & & {\Dat} \ar[dl]^{G} \\
								& {\Vat_0} }
	conmuta salvo isomorfismo, es decir existe un isomorfismo na- \linebreak tural $\alpha : F \Rightarrow G \circ f$.
	
	Notemos también que $Cat / \Vat_0$ es una categoría tensorial definiendo \linebreak $\Cat \times \Dat \stackrel{F \otimes G}{\longrightarrow} \Vat_0$ según $(F \otimes G) (C,D) = F(C) \otimes G(D)$.
	
\begin {\prop}
	$End^\lor : Cat / \Vat_0 \rightarrow Coalg_\Vat$ es un funtor tensorial.
\end {\prop}

\begin{proof}
	Debemos asignarle a una flecha $(f,\alpha)$ entre dos objetos \linebreak $\Cat \stackrel{F}{\rightarrow} \Vat_0$  y $\Dat \stackrel{G}{\rightarrow} \Vat_0$, un morfismo de coálgebras $$End^\lor(f): End^\lor(F) \rightarrow End^\lor(G)$$
	
	Tenemos las flechas $FC \otimes FC^\wedge \stackrel{\alpha_C \otimes {(\alpha^{-1})_C}^\wedge}{\longrightarrow} Gf(C) \otimes Gf(C)^\wedge \rightarrow End^\lor(G)$. Para que induzcan una flecha $End^\lor(F) \rightarrow End^\lor(G)$ tenemos que ver que el diagrama
	
	\xymatrix @C=0.7pc { & FC \otimes FC^\wedge \ar[rrrr]^<(.25){\alpha_C \otimes {(\alpha^{-1})_C}^\wedge} & & & & Gf(C) \otimes Gf(C)^\wedge \ar[dr]^{\lambda_{f(C)}} \\
						 FC \otimes FC'^\wedge \ar[ur]^{FC \otimes F^\wedge(g)} \ar[dr]_{F(h) \otimes FC'^\wedge} & & & & & & \int^\Dat GD \otimes GD^\wedge \\
						  & FC' \otimes FC'^\wedge \ar[rrrr]^<(.25){\alpha_{C'} \otimes {(\alpha^{-1})_{C'}}^\wedge} & & & & Gf(C') \otimes Gf(C')^\wedge \ar[ur]_{\lambda_{f(C')}} \\
	}
	
	conmuta. Para ello agregamos $Gf(C') \otimes Gf(C)^\wedge$ en el medio

\footnotesize	
	\xymatrix @C=0.1pc { & FC \otimes FC^\wedge \ar[rrrr]^{\alpha_C \otimes {(\alpha^{-1})_C}^\wedge} & & & & Gf(C) \otimes Gf(C)^\wedge \ar[dr]^{\lambda_{f(C)}} \\
						FC' \otimes FC^\wedge \ar[ur]^{FC \otimes F^\wedge(g)} \ar[dr]_{F(h) \otimes FC'^\wedge} \ar[rrr]^{\alpha_C \otimes {(\alpha^{-1})_{C'}}^\wedge} & & & Gf(C') \otimes Gf(C)^\wedge \ar[urr]_{\ \ \ Gf(C) \otimes (Gf)^\wedge(h)} \ar[drr]^{\ \ \ Gf(h) \otimes Gf(C)^\wedge} & & & \int^\Dat GD \otimes GD^\wedge \\
						  & FC' \otimes FC'^\wedge \ar[rrrr]^{\alpha_{C'} \otimes {(\alpha^{-1})_{C'}}^\wedge} & & & & Gf(C') \otimes Gf(C')^\wedge \ar[ur]_{\lambda_{f(C')}} \\
	}
\normalsize

	y la conmutatividad de los paralelogramos de arriba y abajo se obtienen por la naturalidad de $\alpha^{-1}$ 
	 y $\alpha$ respectivamente, mientras que el rombo de la derecha es por definición del coend.
	
	La tensorialidad se obtiene de la propiedad \ref{dobleproducto} tomando \linebreak $(A,B,X,Y)= (F,F^{\wedge},G,G^{\wedge})$, y es aquí donde se usa la simetría de $\Vat$.

\end{proof}

\subsection{$C$ y $F$ tensoriales nos dan $m$ y $u$} \label{sec:myu}

	Observemos ahora que el hecho de que $\Cat$ sea una categoría tensorial \linebreak (estricta) y que $F$ sea un funtor tensorial (no necesariamente estricto) equi- vale a tener un monoide $\Cat \stackrel{F}{\rightarrow} \Vat_0$ en la categoría $Cat / \Vat_0$.
	
		En efecto, si $\otimes_\Cat: \Cat \times \Cat \rightarrow \Cat$ es el producto tensorial con neutro $I_\Cat \in \Cat$, $s: F \otimes F \Rightarrow F \circ \otimes_\Cat$ es el isomorfismo natural que dice que $F$ respeta a $\otimes$, y $f : I \stackrel{\cong}{\rightarrow} F(I_\Cat)$, obtenemos la multiplicación
	
	\ \ \ \ \ \ \ \ \ \ \ \ \xymatrix {{\Cat} \ar[dr]_{F} & & {\Cat \times \Cat} \ar[ll]_{\otimes_\Cat} \ar[dl]^{F \otimes F} \\
								& \Vat_0 }
	
	donde el triángulo conmuta salvo el isomorfismo $s$, y la unidad
	
	\ \ \ \ \ \ \ \ \ \ \ \ \xymatrix { {\{*\}} \ar[rr]^{I_\Cat} \ar[dr]_{I} & & {\Cat} \ar[dl]^{F} \\
								& \Vat_0 }
								
	donde el triángulo conmuta salvo el isomorfismo $f$.
	
	La asociatividad y unidad del producto tensorial se corresponden con las de la multiplicación.
	
	Luego, como un funtor tensorial manda monoides en monoides pues \linebreak respeta las composiciones, las identidades y al producto tensorial, tenemos que si $\Cat$ es una categoría tensorial (estricta) y $F$ es un funtor tensorial entonces $End^{\lor}(F)$ obtiene una estructura de biálgebra en $\Vat$ (pues resulta un monoide en la categoría de las coálgebras).
	
	Nos interesa entender un poco mejor qué hacen la multiplicación $m$ y la unidad $u$ de $End^{\lor}(F)$.
	
	$m$ viene dada por aplicar el funtor $End^\lor$ a la flecha $({\otimes}_\Cat,s)$, luego está inducida por las flechas (volvemos a llamar $\otimes$ a todos los productos tensoriales)
	
	\small
	$FC \otimes FD \otimes (FC \otimes FD)^\wedge \stackrel{s_{C,D} \otimes (s_{C,D}^{-1})^\wedge}{\longrightarrow} F(C \otimes D) \otimes (F(C \otimes D))^\wedge \stackrel{\lambda_{C \otimes D}}{\rightarrow} End^\lor(F)$.
	\normalsize
	
	Vía el isomorfismo entre $End^\lor(F) \otimes End^\lor(F)$ y $End^\lor(F \otimes F)$ inducido por las propiedades \ref{dobleproducto} y \ref{dualXporY} obtenemos el diagrama
	
	\small
		\xymatrix @C=0.1pc {End^\lor(F) \otimes End^\lor(F) \ar[r]^{\cong}  \ar@/^2pc/[rr]^{m} & End^\lor(F \otimes F) \ar[r] & End^\lor(F) \\
						FC \otimes FC^\wedge \otimes FD \otimes FD^\wedge \ar[u]^{\lambda_C \otimes \lambda_D} \ar[rd]^{FC \otimes \psi_{FC^\wedge,FD \otimes FD^\wedge} } &  & F(C \otimes D) \otimes (F(C \otimes D))^\wedge \ar[u]^{\lambda_{C \otimes D}} \\
						& FC \otimes FD \otimes (FC \otimes FD)^\wedge \ar[ru]^{s_{C,D} \otimes (s_{C,D}^{-1})^\wedge} \ar[uu]^{\lambda_{C,D}}}
	\normalsize
	
	Llamamos $m_{C,D}$ a la flecha 
	
	\xymatrix {FC \otimes FC^\wedge \otimes FD \otimes FD^\wedge \ar[rrr]^{FC \otimes \psi_{FC^\wedge,FD \otimes FD^\wedge}} & & &
						FC \otimes FD \otimes (FC \otimes FD)^\wedge  ... }
	\xymatrix {& & & & & ... \ar[rr]^>>>>>>>>>{s_{C,D} \otimes ({s_{C,D}^{-1}})^\wedge} & & F(C \otimes D) \otimes (F(C \otimes D))^\wedge }
	
	En cuanto a $u$, al aplicar el funtor $End^\lor$ a la flecha ($I_\Cat,f$) obtenemos que está inducida por la flecha $I \otimes I^\wedge \stackrel{f \otimes (f^{-1})^\wedge}{\longrightarrow} F(I_\Cat) \otimes F(I_\Cat)^\wedge \stackrel{\lambda_I}{\rightarrow} End^\lor(F)$,	que junto con el isomorfismo $I \cong I \otimes I^\wedge$ nos da el diagrama
	
	\ \ \ \ \ \ \ \ \ \ \ \ \xymatrix @C=4pc { I \ar[d]_{\cong} \ar[r]^{u} &  End^\lor(F) \\
							 I \otimes I^\wedge \ar[r]^{f \otimes (f^{-1})^\wedge} & F(I_\Cat) \otimes F(I_\Cat)^\wedge \ar[u]_{\lambda_I} }

	Para obtener el resultado enunciado en el item 1, es decir que $End^\lor(F)$ es una biálgebra aunque $\Cat$ no sea estricta, debemos reemplazar $\Cat$ por su categoría tensorial estricta equivalente $st(\Cat)$ (ver sección \ref{sub:estrictas}). Para ello, a partir de $\Cat \stackrel{F}{\rightarrow} \Vat_0$ definimos $st(\Cat) \stackrel{st(F)}{\rightarrow} \Vat_0$ vía $st(F)(C_1...C_n) = F \circ s$, y obtenemos así el diagrama 
	
	\ \ \ \ \ \ \ \ \ \ \ \ \xymatrix @C=0.5pt { {\Cat} \ar@/^/[rr]^{i} \ar[dr]_{F} & & {St(\Cat)} \ar[dl]^{st(F)} \ar@/^/[ll]^{s}\\
								& \Vat_0 }
	
	que conmuta salvo los isomorfismos que permiten realizar los productos tensoriales en cualquier orden en $\Vat$. Este diagrama muestra entonces que $\Cat$ y $St(\Cat)$ son equivalentes como categorías sobre $\Vat_0$, es decir son objetos isomorfos en $Cat / \Vat_0$. Luego, las coálgebras $End^\lor(F)$ y $End^\lor(st(F))$ son isomorfas, y $End^\lor(F)$ obtiene la estructura de biálgebra que tiene $End^\lor(st(F))$.
	
\begin{remark}
	La estructura de álgebra de $End^\lor(F)$ es un caso particular de la multiplicación que siempre se tiene en $Nat^\lor(F,G)$ si $F$ y $G$ son funtores monoidales.
\end{remark}

\subsection{Cuándo $End^\lor(F)$ es un álgebra de Hopf}

Veremos ahora el segundo enunciado 

\begin{\prop} \label{EndesHopf}
	Sea $\Cat$ una categoría tensorial en la que todo objeto tiene un dual a derecha y $F: \Cat \rightarrow \Vat_0$ un funtor tensorial. Entonces $End^\lor(F)$ es un álgebra de Hopf en $\Vat$.
\end{\prop}

\begin{proof}
	
	Para todo objeto $C$ de $\Cat$ tenemos las flechas $C \otimes C^\lor \stackrel{\varepsilon}{\rightarrow} I_\Cat$ y $I_\Cat \stackrel{\eta}{\rightarrow} C^\lor \otimes C$. Aplicando el funtor tensorial $F$ obtenemos los diagramas
	
	1. \xymatrix { F(C \otimes C^\wedge) \ar[r]^>>>>>{F(\varepsilon)} & F(I_\Cat) \\
							FC \otimes F(C^\wedge) \ar[u]^{s} \ar[r]^>>>>>{\varepsilon} & I \ar[u]_{f} } \ \ \ 
	y 2. \xymatrix { F(I_\Cat) \ar[r]^>>>>>{F(\eta)} & F(C^\wedge \otimes C) \\
							I \ar[u]^{f} \ar[r]^>>>>>{\eta} & F(C^\wedge) \otimes FC \ar[u]_{s} }
	
	que dicen que $F(C^\wedge)$ es un dual (a derecha) de $F(C)$. 

	Por otro lado, $F(C \otimes C^\wedge)$ es isomorfo a $FC \otimes F(C^\wedge)$ vía $s$. Luego, por la propiedad \ref{dualXporY}, y el hecho de que como $\Vat$ es simétrica ${F(C^\wedge)}^\wedge = FC$, tenemos que $(FC \otimes F(C^\wedge))^\wedge = FC \otimes F(C^\wedge)$ y luego también \linebreak $F(C \otimes C^\wedge)^\wedge = F(C \otimes C^\wedge)$ con la unidad $\eta$ dada por 
	\footnotesize
	$$ I \stackrel{\psi \circ \eta}{\rightarrow} FC \otimes F(C^\wedge) \stackrel{FC \otimes \eta \otimes F(C^\wedge)}{\rightarrow} FC \otimes F(C^\wedge) \otimes FC \otimes F(C^\wedge) \stackrel{s \otimes s}{\rightarrow} F(C \otimes C^\wedge) \otimes F(C \otimes C^\wedge)$$
	\normalsize
	
	Debemos hallar la antípoda $a : End^\lor(F) \rightarrow End^\lor(F)$
	
	Gracias a la simetría del producto tensorial de $\Vat$, tenemos las flechas $FC \otimes F(C^\wedge) \stackrel{\psi}{\rightarrow} F(C^\wedge) \otimes FC$ que como $FC$ es dual de $F(C^\wedge)$ las podemos componer con $\lambda_{C^\wedge}$
	
	Para que induzcan a $a$, debemos verificar que para toda flecha $f:C \rightarrow C'$ en $\Cat$ el siguiente diagrama conmuta
	
	\xymatrix@C=1.5pc { & FC \otimes F(C^\wedge) \ar[rr]^{\psi} & & F(C^\wedge) \otimes FC \ar[dr]^{\lambda_{C^\wedge}} \\
					FC \otimes F(C'^\wedge) \ar[ur]^{FC \otimes F(f)^\wedge} \ar[dr]_{F(f) \otimes F(C'^\wedge)} & & & & End^\lor(F) \\
					& F(C') \otimes F(C'^\wedge) \ar[rr]^{\psi} & &  F(C'^\wedge) \otimes F(C') \ar[ur]_{\lambda_{C'^\wedge} } }

	Para ver eso, completamos el diagrama y miramos

	\xymatrix@C=0.2pc { & FC \otimes F(C^\wedge) \ar[rr]^{\psi} & & F(C^\wedge) \otimes FC \ar[dr]^{\lambda_{C^\wedge}} \\
					FC \otimes F(C'^\wedge) \ar[ur]^{FC \otimes F(f)^\wedge} \ar[dr]_{F(f) \otimes F(C'^\wedge)} \ar[rr]^{\psi} & & F(C'^\wedge) \otimes FC \ar[ur]_{F(c)^\wedge \otimes FC} \ar[dr]^{F(C'^\wedge) \otimes F(c)} & & End^\lor(F) \\
					& F(C') \otimes F(C'^\wedge) \ar[rr]^{\psi} & &  F(C'^\wedge) \otimes F(C') \ar[ur]_{\lambda_{C'^\wedge} } }
	
y las conmutatividades de los paralelogramos son triviales y la del rombo por coend (pues $F(c)={F(c)^\wedge}^\wedge$). 
	
	Debemos verificar ahora que el siguiente diagrama (y el análogo multiplicando por $a$ al otro lado) conmuta
	
	
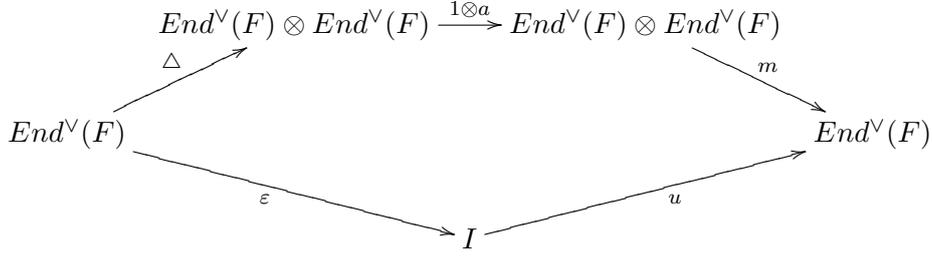
\xymatrix @C=0.5pc {	& End^\lor(F) \otimes End^\lor(F) \ar[rr]^{1 \otimes a} & & End^\lor(F) \otimes End^\lor(F) \ar[dr]^{m} \\
						End^\lor(F) \ar[ur]^{\triangle} \ar[drr]_{\varepsilon} & & & & End^\lor(F) \\
						& & I \ar[urr]_{u}  }
	
	Una última vez, gracias al trabajo que ya hicimos entendiendo qué hacen todas estas operaciones, "`completamos"' el diagrama con otro pentágono adentro de la siguiente forma
	
	\tiny
		
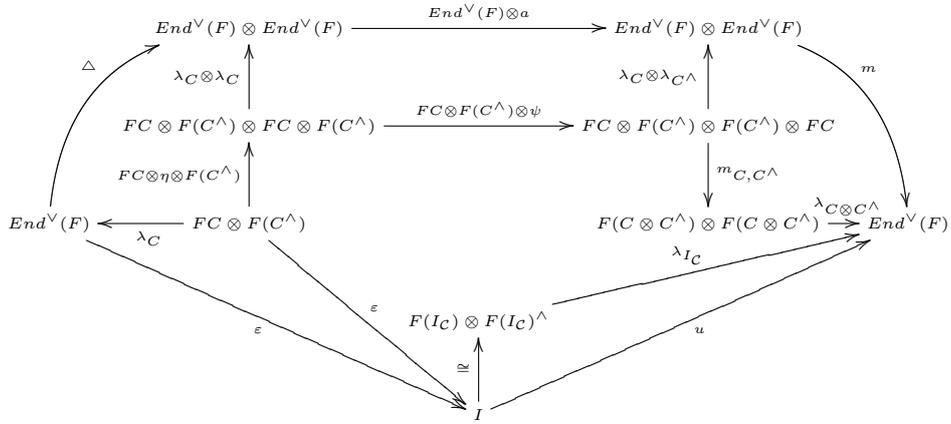
\xymatrix @C=0.5pc {	& End^\lor(F) \otimes End^\lor(F) \ar[rr]^{End^\lor(F) \otimes a} & & End^\lor(F) \otimes End^\lor(F) \ar@/^2pc/[ddr]^{m} \\
								& FC \otimes F(C^\wedge) \otimes FC \otimes F(C^\wedge) \ar[u]^{\lambda_C \otimes \lambda_C} \ar[rr]^{FC \otimes F(C^\wedge) \otimes \psi} & & FC \otimes F(C^\wedge) \otimes F(C^\wedge)	\otimes FC \ar[u]^{\lambda_C \otimes \lambda_{C^\wedge}}	\ar[d]^{m_{C,C^\wedge}}	\\
						End^\lor(F) \ar@/^2pc/[uur]^{\triangle} \ar[ddrr]_{\varepsilon} & FC \otimes F(C^\wedge) \ar[l]^{\lambda_C} \ar[u]^{FC \otimes \eta \otimes F(C^\wedge)} \ar[rdd]^{\varepsilon} & & F(C \otimes C^\wedge) \otimes F(C \otimes C^\wedge) \ar[r]^<(0.25){\lambda_{C \otimes C^\wedge}} & End^\lor(F) \\
						& & F(I_\Cat) \otimes F(I_\Cat)^\wedge \ar[urr]^{\lambda_{I_\Cat}}					\\
						& & I \ar[uurr]_{u} \ar[u]^{\cong} }
	\normalsize
	
	Luego, la conmutatividad de los cinco diagramas pequeños que involucran a $\triangle$, $m$, $\varepsilon$, $u$ y $a$, viene dada por la construcción y descripción que hemos hecho de cada una de esas operaciones.
	
	Para verificar la conmutatividad del "`hexágono"' interior, como está llegando a un coend agregamos el objeto $F(C \otimes C^\wedge)$ y el hexágono queda
	
	\tiny
	
	
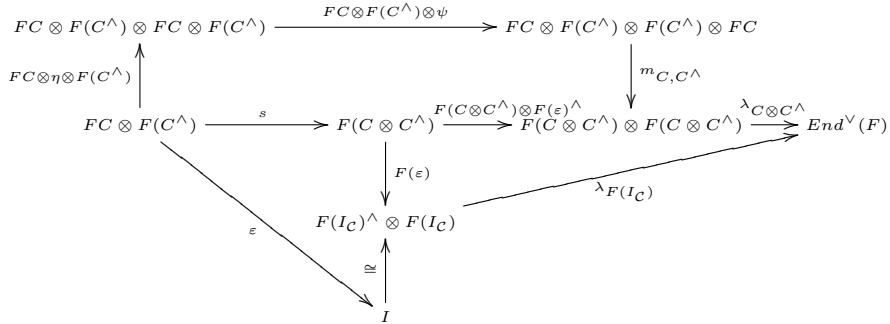
\xymatrix @C=1pc { FC \otimes F(C^\wedge) \otimes FC \otimes F(C^\wedge) \ar[rr]^{FC \otimes F(C^\wedge) \otimes \psi} & & FC \otimes F(C^\wedge) \otimes F(C^\wedge)	\otimes FC \ar[d]^{m_{C,C^\wedge}}	\\
						 FC \otimes F(C^\wedge) \ar[u]^{FC \otimes \eta \otimes F(C^\wedge)} \ar[ddr]_{\varepsilon} \ar[r]^{s} & F(C \otimes C^\wedge) \ar[r]^{F(C \otimes C^\wedge) \otimes F(\varepsilon)^\wedge} \ar[d]^{F(\varepsilon)} & F(C \otimes C^\wedge) \otimes F(C \otimes C^\wedge) \ar[r]^>>>>{\lambda_{C \otimes C^\wedge}} & End^\lor(F) \\ 
						 	 & F(I_\Cat)^\wedge \otimes F(I_\Cat) \ar[urr]_{\lambda_{F(I_\Cat)}}			\\
						 	 & I \ar[u]^{\cong}		}
	
	\normalsize

	donde el triángulo inferior izquierdo es el diagrama 1 del comienzo de la demostración y el derecho conmuta por coend. Resta ver la conmutatividad del rectángulo superior. Comenzamos por la composición superior; para el cálculo de ascensores esta es, recordando la definición de $m_{C,D}$ que dimos en la sección \ref{sec:myu}:
	
	\begin{equation} \label{compsup}
\xymatrix@C=-2pc{ FC \ar@2{-}[d] & &  & \ar@{-}[ld] \ar@{-}[dr] \ \ \ \ \ \ \ar@{}[d]|{\eta} \ \ \ \ \ \ & & & F(C^\wedge) \ar@2{-}[d] & & & & & \\
								 FC \ar@2{-}[d] & & F(C^\wedge) \ar@2{-}[d] & & FC \ar@2{-}[rrd]|{\psi} & & F(C^\wedge) \ar@2{-}[lld]|{\psi} & & & & & \\	
								 FC \ar@2{-}[d] & & F(C^\wedge) \ar@2{-}[rrrrd]|{\psi} & & F(C^\wedge) \ar@2{-}[lld] & \ar@{}[dd]|>>>>>>>>{({s_{C,C^\wedge}^{-1}})^\wedge} & FC \ar@2{-}[lld] & & & & & \\							 
								 FC \ar@{-}[dr] & \ar@{}[d]|>>>>>>>>{s} & F(C^\wedge) \ar@{-}[ld] & & FC \ar@{-}[dr] & & F(C^\wedge) \ar@{-}[ld] & & & & & \\							 
								 & F(C \otimes C^\wedge) & & & & F(C \otimes C^\wedge) & & & & & & }
\end{equation}
										
	Lo primero que haremos es ver que $({s_{C,C^\wedge}^{-1}})^\wedge = s$. Por definición del dual de una flecha, teniendo en cuenta cómo son los $\eta$ y $\varepsilon$ de estas dualidades, $({s_{C,C^\wedge}^{-1}})^\wedge$ está dada por el diagrama
	
\ \ \ \ \ \ \ \ \ \ \ \ \ \ \ \xymatrix@C=-2pc{ & & & \ar@{-}[llld] \ar@{-}[rrrd] \ar@{}[d]|{\eta} & & & & & FC \ar@2{-}[d] & \ \ \ \ \ \ \ \ \ \ \ \ & F(C^\wedge) \ar@2{-}[d] & \\
								 F(C^\lor) \ar@2{-}[rrrrrrd]|{\psi} & & & & & & FC \ar@2{-}[lllllld]|{\psi} & & FC \ar@2{-}[d] & & F(C^\wedge) \ar@2{-}[d] & \\	
								 FC \ar@2{-}[d] & & & \ar@{-}[ld] \ar@{-}[rd] \ \ \ \ \ \ \ar@{}[d]|{\eta} \ \ \ \ \ \ & & & F(C^\wedge) \ar@2{-}[d] & & FC \ar@2{-}[d] & & F(C^\wedge) \ar@2{-}[d] & \\					
								 FC \ar@{-}[rd] & \ar@{}[d]|>>>>{s} & F(C^\wedge) \ar@{-}[ld] & & FC \ar@{-}[rd] & \ar@{}[d]|>>>>{s} & F(C^\wedge) \ar@{-}[ld] & & FC \ar@2{-}[d] & & F(C^\wedge) \ar@2{-}[d] & \\
								 & F(C \otimes C^\wedge) & & & & F(C \otimes C^\wedge) \ar@{-}[ld] \ar@{-}[rd] \ar@{}[d]|<<<<<<{s^{-1}} & & & FC \ar@2{-}[d] & & F(C^\wedge) \ar@2{-}[d] & \\			
								 & & & & FC \ar@2{-}[d] & & F(C^\wedge) \ar@2{-}[rrd]|{\psi} & & FC \ar@2{-}[lld]|{\psi} & & F(C^\wedge) \ar@2{-}[d] & \\							 
								 & & & & FC \ar@2{-}[d] & & FC \ar@{-}[rd] & \ \ \ \ \ \ \ar@{}[d]|{\varepsilon} \ \ \ \ \ \  & F(C^\wedge) \ar@{-}[ld] & & F(C^\wedge) \ar@2{-}[d] & \\							 
								 & & & & FC \ar@{-}[rrrd] & & & \ar@{}[d]|{\varepsilon} & & & F(C^\wedge) \ar@{-}[llld] & \\							 
								 & & & & & & & & & & & }
	
	Primero simplificamos $s$ con $s^{-1}$:

\ \ \ \ \ \ \ \ \ \ \ \ \ \ \ 
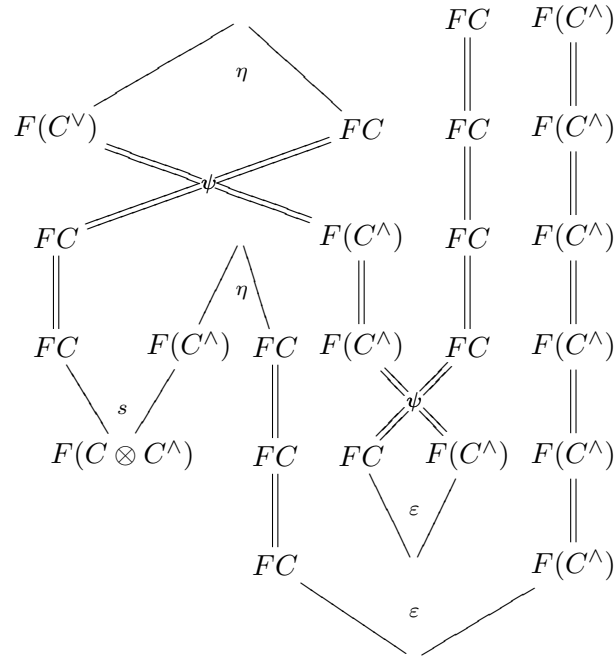
\xymatrix@C=-2pc{ & & & \ar@{-}[llld] \ar@{-}[rrrd] \ar@{}[d]|{\eta} & & & & & FC \ar@2{-}[d] & & F(C^\wedge) \ar@2{-}[d] & \\
								 F(C^\lor) \ar@2{-}[rrrrrrd]|{\psi} & & & & & & FC \ar@2{-}[lllllld]|{\psi} & & FC \ar@2{-}[d] & & F(C^\wedge) \ar@2{-}[d] & \\	
								 FC \ar@2{-}[d] & & & \ar@{-}[ld] \ar@{-}[rd] \ \ \ \ \ \ \ar@{}[d]|{\eta} \ \ \ \ \ \ & & & F(C^\wedge) \ar@2{-}[d] & & FC \ar@2{-}[d] & & F(C^\wedge) \ar@2{-}[d] & \\					
								 FC \ar@{-}[rd] & \ar@{}[d]|>>>>{s} & F(C^\wedge) \ar@{-}[ld] & & FC \ar@2{-}[d] & \ \ \ \ \ \ \ \ \ \ \ \ & F(C^\wedge) \ar@2{-}[rrd]|{\psi} & & FC \ar@2{-}[lld]|{\psi} & \ \ \ \ \ \ \ \ \ \ \ \ & F(C^\wedge) \ar@2{-}[d] & \\							 
								 & F(C \otimes C^\wedge) & & & FC \ar@2{-}[d] & & FC \ar@{-}[rd] & \ \ \ \ \ \ \ar@{}[d]|{\varepsilon} \ \ \ \ \ \ & F(C^\wedge) \ar@{-}[ld] & & F(C^\wedge) \ar@2{-}[d] & \\
								 & & & & FC \ar@{-}[rrrd] & & & \ar@{}[d]|{\varepsilon} & & & F(C^\wedge) \ar@{-}[llld] & \\							 
								 & & & & & & & & & & & }
										
	Luego juntamos la $\eta$ y $\varepsilon$ de la dualidad $C^\wedge \dashv C$ (ver propiedad \ref{DualEnSimetrica}).

\ \ \ \ \ \ \ \ \ \ \ \ \ \ \ 
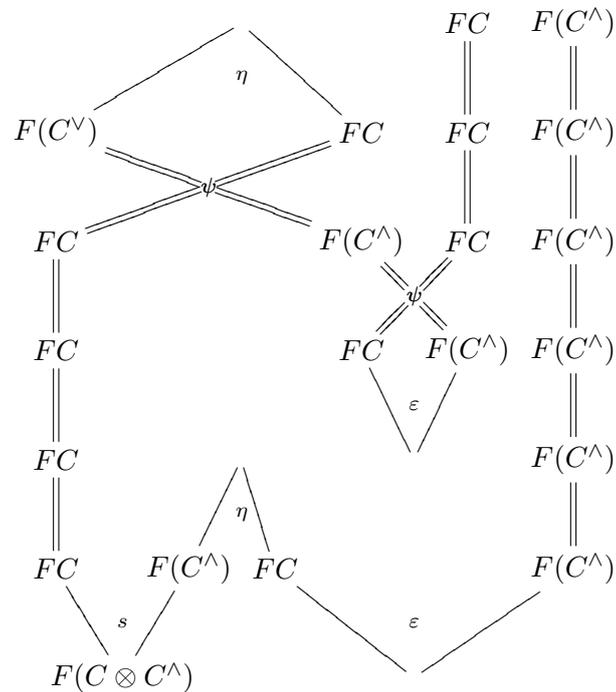
\xymatrix@C=-2pc{ & & & \ar@{-}[llld] \ar@{-}[rrrd] \ar@{}[d]|{\eta} & & & & & FC \ar@2{-}[d] & & F(C^\wedge) \ar@2{-}[d] & \\
								 F(C^\lor) \ar@2{-}[rrrrrrd]|{\psi} & & & & & & FC \ar@2{-}[lllllld]|{\psi} & & FC \ar@2{-}[d] & & F(C^\wedge) \ar@2{-}[d] & \\	
								 FC \ar@2{-}[d] & & & & & & F(C^\wedge) \ar@2{-}[rrd]|{\psi} & & FC \ar@2{-}[lld]|{\psi} & & F(C^\wedge) \ar@2{-}[d] & \\							 
								 FC \ar@2{-}[d] & & & & & & FC \ar@{-}[rd] & \ \ \ \ \ \ \ar@{}[d]|{\varepsilon} \ \ \ \ \ \ & F(C^\wedge) \ar@{-}[ld] & \ \ \ \ \ \ \ \ \ \ \ \ & F(C^\wedge) \ar@2{-}[d] & \\
								 FC \ar@2{-}[d] & & & \ar@{-}[ld] \ar@{-}[rd] \ \ \ \ \ \ \ar@{}[d]|{\eta} \ \ \ \ \ \ & & & & & & & F(C^\wedge) \ar@2{-}[d] & \\
								 FC \ar@{-}[rd] & \ar@{}[d]|{s} & F(C^\wedge) \ar@{-}[ld] & & FC \ar@{-}[rrrd] & \ \ \ \ \ \ \ \ \ \ \ \ & & \ar@{}[d]|{\varepsilon} & & & F(C^\wedge) \ar@{-}[llld] & \\							 							 & F(C \otimes C^\wedge) & & & & & & & & & & }
											
	Entonces por la primer igualdad triangular de esa dualidad el diagrama se colapsa en 

\ \ \ \ \ \ \ \ \ \ \ \ \ \ \ \ \ \ \ \ \xymatrix@C=-2pc{ FC \ar@2{-}[d] & & & \ar@{-}[ld] \ar@{-}[rd] \ \ \ \ \ \ \ar@{}[d]|{\eta} \ \ \ \ \ \ & & & F(C^\wedge) \ar@2{-}[d] & & 						FC \ar@{-}[rd] & \ar@{}[d]|{s} & F(C^\wedge) \ar@{-}[ld] & & & \\
								 FC \ar@{-}[rd] & \ar@{}[d]|{s} & F(C^\wedge) \ar@{-}[ld] & & FC \ar@{-}[rd] & \ \ \ \ \ \ \ar@{}[d]|{\varepsilon} \ \ \ \ \ \ & F(C^\wedge) \ar@{-}[ld] & \ \ \ \ \ \ \ \ \ = \ \ \ \ \ \ \ \ \ & 					& F(C \otimes C^\wedge)  & & \\	
								 & F(C \otimes C^\wedge) & & & & & & & & & & }
								 
	La igualdad vale por la primer igualdad triangular de la dualidad $FC \dashv F(C^\lor)$. Volviendo entonces a \eqref{compsup}, por el lema \ref{etapsi} nos queda

\ \ \ \ \ \ \ \ \ \ \ \ \ \ \ \ \ \ \ \ \xymatrix@C=-2pc{ FC \ar@2{-}[d] & & F(C^\wedge) \ar@2{-}[d] & & & \ar@{-}[ld] \ar@{-}[dr] \ar@{}[d]|{\eta} & & & & & & \\
								 FC \ar@2{-}[d] & & F(C^\wedge) \ar@2{-}[rrd]|{\psi} & \ \ \ \ \ \ \ \ \ \ \ \ & F(C^\wedge) \ar@2{-}[lld]|{\psi} & & FC \ar@2{-}[d] & & & & & \\	
								 FC \ar@2{-}[d] & & F(C^\wedge) \ar@2{-}[rrrrd]|{\psi} & & F(C^\wedge) \ar@2{-}[lld] & & FC \ar@2{-}[lld] & & & & & \\							 
								 FC \ar@{-}[dr] & \ar@{}[d]|{s} & F(C^\wedge) \ar@{-}[ld] & & FC \ar@{-}[dr] & \ar@{}[d]|{s} & F(C^\wedge) \ar@{-}[ld] & & & & & \\							 
								 & F(C \otimes C^\wedge) & & & & F(C \otimes C^\wedge) & & & & & & }
	
	que por la coherencia para las $\psi$ se reduce a

\ \ \ \ \ \ \ \ \ \ \ \ \ \ \ \ \ \ \ \ \xymatrix@C=-2pc{ FC \ar@2{-}[d] & & F(C^\wedge) \ar@2{-}[d] & & & \ar@{-}[ld] \ar@{-}[dr] \ar@{}[d]|{\eta} & & & & & & & \\
								 FC \ar@{-}[dr] & \ar@{}[d]|{s} & F(C^\wedge) \ar@{-}[ld] & \ \ \ \ \ \ \ \ \ \ \ \ \ & F(C^\wedge) \ar@2{-}[rrd]|{\psi} & & FC \ar@2{-}[lld]|{\psi} & & & & & & \\	
								 & F(C \otimes C^\wedge) & & & FC \ar@{-}[dr] & \ar@{}[d]|{s} & F(C^\wedge) \ar@{-}[ld] & & & & & & \\							 
								 & & & & & F(C \otimes C^\wedge) & & & & & & & }

	Luego nos resta ver que $F(\varepsilon)^\wedge$ coincide con la parte derecha de este último diagrama. Recordando la unidad de la dualidad de $F(C \otimes C^\wedge)$ que describimos al principio de la demostración, la definición del dual de una flecha, y que $F(\varepsilon) = \varepsilon \circ s^{-1}$, tenemos que $F(\varepsilon)^\wedge$ es el diagrama

\ \ \ \ \ \ \ \ \ \ \ \ \ \ \ \ \ \ \ \ 
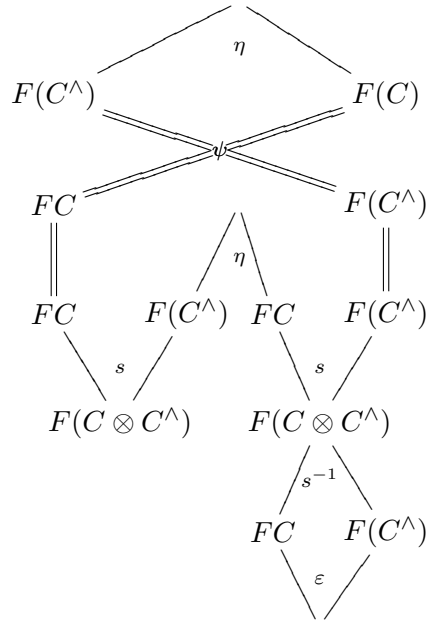
\xymatrix@C=-2pc{ & & & \ar@{-}[llld] \ar@{-}[rrrd] \ar@{}[d]|{\eta} & & & & & & & & \\
								 F(C^\wedge) \ar@2{-}[rrrrrrd]|{\psi} & & & & & & F(C) \ar@2{-}[lllllld]|{\psi} & & & & & \\	
								 FC \ar@2{-}[d] & & & \ar@{-}[ld] \ar@{-}[dr] \ \ \ \ \ \ \ar@{}[d]|{\eta} \ \ \ \ \ \ & & & F(C^\wedge) \ar@2{-}[d] & & & & & \\							 
								 FC \ar@{-}[dr] & \ar@{}[d]|{s} & F(C^\wedge) \ar@{-}[ld] & & FC \ar@{-}[rd] & \ar@{}[d]|{s} & F(C^\wedge) \ar@{-}[ld] & & & & & \\							 
								 & F(C \otimes C^\wedge) & & & & F(C \otimes C^\wedge) \ar@{-}[ld] \ar@{-}[dr] \ar@{}[d]|{s^{-1}} & & & & & & \\							 
								 & & & & FC \ar@{-}[rd] & \ar@{}[d]|{\varepsilon}  & F(C^\wedge) \ar@{-}[ld] & & & & & \\							 
								 & & & & & & & & & & & }
											
	Simplificamos $s$ con $s^{-1}$ y por la primer igualdad triangular de la dualidad de $FC$ obtenemos lo deseado.
\end{proof}

\pagebreak
\section{El caso $\Vat = Vec_K$} \label{sec:Vec}

En esta sección probaremos lo anticipado en las observaciones \ref{obs1} y \ref{obs2}

\subsection{Comenzando a trabajar con elementos}

En esta sección trabajaremos el caso $\Vat = Vec_K$, $\Vat_0 = Vec_K^{<\infty}$, $\otimes = \otimes_K$. Es decir, $F, G: \Cat \rightarrow Vec_K$ y $dim_K(GC)<\infty$ para todo $C$. Como esta categoría es simétrica, duales a derecha y a izquierda coinciden, y son por supuesto el dual usual de un espacio vectorial. En la sección \ref{sub:objdual} mostramos detalladamente cómo son $\eta$ y $\varepsilon$ para esta dualidad. Como ya mencionamos anteriormente, la definición $Nat^\lor(F,G)= \int^\Cat FC \otimes (GC)^\wedge$ coincidirá con $Nat^\lor(F,G)= \int^\Cat Hom(FC,GC)^\wedge$ si $F$ también toma valores en los espacios vectoriales de dimensión finita.

Dentro del $K$-espacio vectorial $Nat^\lor(F,G)$ tendremos los elementos $\lambda_C (v \otimes w^\lor)$, donde $v \in FC$ y $w^\lor \in (GC)^\lor = GC^*$. Seguiremos a veces la notación de \cite{JS}, usando $[v \otimes w^\lor]$ para referirnos a la clase de estos elementos, sin especificar el objeto $C$. Por la construcción de los coends (similar a la de los colímites) en la categoría $Vec_K$, como el coend lo que debe hacer es que para cada flecha $f: C \rightarrow C'$ en $\Cat$ el siguiente rombo

\ \ \ \ \ \ \ \ \ \ \ \ \ \ \ \xymatrix{ & FC \otimes FC^\lor \ar[dr]_{\lambda_C} \\
						FC \otimes (FC')^\lor \ar[ru]_{FC \otimes F(f)^\lor} \ar[rd]^{F(f) \otimes (FC')^\lor} & & End^\lor(F) \\
						 & FC' \otimes (FC')^\lor \ar[ru]^{\lambda_{C'}} }
	
conmute, tenemos la siguiente caracterización del $K$-e.v. $Nat^\lor(F,G)$ (comparar con \cite{JS}, p.433):

\begin{\prop}
	Sean $F,G: \Cat \rightarrow Vec_K$ dos funtores tales que $dim_K(GC)<\infty$ para todo $C$. Sean, para cada $C$, $\{v_i^C\}_{i \in I_C}$ y $\{^Cw_j^\lor \}_{j \in J_C}$ bases de $FC$ y de $GC^\lor$ respectivamente. Entonces, $Nat^\lor(F,G)$ es el $K$-e.v. generado por los elementos $\{ \lambda_C (v_i^C \otimes {}^Cw_j^\lor) \}$, variando $C$, $i$ y $j$, cuya notación se alivia escribiendo $\lambda_C (v \otimes w^\lor)$ o bien $[v \otimes w^\lor]$, y que están sujetos para cada \linebreak $f: C \rightarrow C'$ a la relación (que se obtiene observando el rombo anterior)
	$$ \lambda_C (v \otimes G(f)^\lor(w^\lor)) = \lambda_{C'} (F(f)(v) \otimes w^\lor) $$
	
	Cada inclusión $\lambda_C$ queda así definida en una base, y su linealidad nos da la forma de calcular combinaciones lineales entre $\lambda_C (v_1 \otimes w_1^\lor)$ y $\lambda_C (v_2 \otimes w_2^\lor)$ (cuando se tiene el mismo $C$ para las dos).
	
\end{\prop}

Lo primero que veremos es una fórmula para la coevaluación \linebreak $\eta: F \Rightarrow Nat^\lor(F,G) \otimes G$. Recordemos que hemos descripto a sus componentes $\eta_C$ en \eqref{eta}. Ahora, en el caso $Vec$, podemos dar el valor de $\eta_C$ en un elemento $v$ por

\begin{equation} \label{etaenvec}
	{\eta}_C (v) = \sum_{i=1}^n \lambda_C (v \otimes w_i^\lor) \otimes w_i
\end{equation}

Observemos que esta fórmula es análoga a la $\gamma_C$ definida en \cite{JS}, p.439. Ahora vamos a probar la suryectividad de $\tilde{\rho}$ definida en la propiedad \ref{morfdecomp}.

\begin{\prop} \label{rhosury}
	Sea $C$ una $K$-coálgebra y $U: Comod_{<\infty}C \rightarrow Vec_K^{<\infty}$ el funtor de olvido. Entonces $\tilde{\rho}: End^\lor(U) \rightarrow C$ es suryectivo
\end{\prop}

\begin {proof}
	Sea $v \in C$. $C$ es un $C$-comódulo con el coproducto dado por la comultiplicación $\triangle$, y en cualquier comódulo, dado un elemento de él, podemos obtener un subcomódulo de dimensión finita que contiene a ese elemento. En efecto, si $\{a_i\}_{i \in I}$ es una base de $C$, y $\triangle(v)= \sum_{j=1}^n a_{i_j} \otimes v_j$, con $a_{i_j} \in C$ consideramos el subespacio $V$ de $C$ generado por $v$ y por los $v_j$. Para ver que es un subcomódulo, debemos verificar que $\triangle(v_j) \in C \otimes V \ \forall j$.
	
	Para ello notemos que el diagrama que dice que el coproducto respeta la comultiplicación (o en este caso el diagrama que expresa la coasociatividad de la comultiplicación), aplicado al elemento $v$, dice que $$\sum_{j=1}^n \triangle(a_{i_j}) \otimes v_j = \sum_{j=1}^n a_{i_j} \otimes \triangle(v_j)$$
	
	Escribamos $\triangle(a_{i_j}) = \sum_{k=1}^m a_{i_k} \otimes \alpha_{i_j}^k$ (donde se entiende que $a_{i_j}$ no tiene por qué coincidir con $a_{i_k}$ aunque $j$ sea igual a $k$), luego 
	$$ \sum_{j=1}^n a_{i_j} \otimes \triangle(v_j) = \sum_{j=1}^n ( \sum_{k=1}^m a_{i_k} \otimes \alpha_{i_j}^k ) \otimes v_j = \sum_{k=1}^m a_{i_k} \otimes (\sum_{j=1}^n \alpha_{i_j}^k \otimes v_j) $$
	Como los $a_i$ son una base, obtenemos que 
	$$ \triangle(v_j) = \left\{ \begin{array}{rl} \sum_{l=1}^n \alpha_{i_l}^k \otimes v_l & \text{si } a_{i_j} = a_{i_k} \text{ para algún } k\\
																								0 & \hbox{ en caso contrario }\end{array} \right. $$
	
	Afirmamos ahora que $\tilde{\rho} ( \lambda_V (v \otimes \varepsilon|_V) ) = v$. Notemos primero para ver esto que el diagrama \xymatrix {& C \otimes C \ar[dl]_{C \otimes \varepsilon}  & \\
						 C \otimes K & C \ar[u]_{\triangle } \ar[l]_{\cong}  } aplicado a $v$ nos da la igualdad $\sum_{j=1}^n a_{i_j} \cdot \varepsilon(v_j) = v$.
	
	Luego, usando el diagrama \eqref{defderho} tenemos que 
	$$\tilde{\rho} ( \lambda_V (v \otimes \varepsilon|_V) ) = (C \otimes \varepsilon) \circ (\triangle \otimes V^\wedge) (v \otimes \varepsilon|_V) = $$
	$$(C \otimes \varepsilon) (\sum_{j=1}^n a_{i_j} \otimes v_j \otimes \varepsilon|_V) = \sum_{j=1}^n a_{i_j} \cdot \varepsilon(v_j) = v $$
	
	Notemos que hay dos $\varepsilon$ distintas, una la counidad $\varepsilon: C \rightarrow K$ la cual se restringe a $V$, y otra la de la dualidad de $V$ que corresponde a la evaluación $\varepsilon: V \otimes V^\lor \rightarrow K$.
\end{proof}	

\subsection{Los minimal models y algunos lemas}

Desarrollaremos en esta sección los conceptos y lemas necesarios para poder demostrar la inyectividad de $\tilde{\rho}$ y el teorema fundamental de la teoría de Tannaka como está enunciado en el item 3 (en la sección \ref{sec:levdeF}). En todo este desarrollo es donde seguiremos más de cerca a \cite{JS}. Primero recordaremos varias definiciones y construcciones que se pueden hacer en una categoría abeliana (para la definición de categoría abeliana ver \cite{ML}, VIII, 3).

\begin{itemize}
	\item Si $A$ es un objeto de $\Cat$, un subobjeto $B$ de $A$ es un monomorfismo \xymatrix { B \hbox{  } \ar@{>->}[r]^{s} & A}
	
	\item La categoría $\Cat$ se dice artiniana si cualquier cadena decreciente (o sea hacia la izquierda) de subobjetos se estabiliza.

	\item La intersección de dos subobjetos es el subobjeto más grande contenido en ambos, y se obtiene como el pull-back de las dos flechas. La unión de dos subobjetos es el subobjeto más chico que contiene a ambos y se obtiene gracias a la artinianidad.
	
	\item Si $A$ es un objeto de $\Cat$, un cociente $B$ de $A$ es un epimorfismo \xymatrix@1 { A \ar@{->>}[r]^{c} & B}
	
	\item La categoría $\Cat$ se dice noetheriana si $\Cat^{op}$ es artiniana, o sea si cualquier cadena decreciente (o sea hacia la derecha) de cocientes se estabiliza.
	
	\item Recordemos que en una categoría abeliana cualquier flecha $A \stackrel{f}{\rightarrow} B$ admite una factorización única salvo isomorfismo
	
	\ \ \ \ \ \ \ \ \ \ \ \ \ \ \ \ \ \ \ \ \ \ \ \ \ \ \ \ \ \ \xymatrix { A \ar[rr]^{f} \ar@{->>}[rd]_{e} & & B \\
							& C \hbox{  } \ar@{>->}[ru]_{m} }
							
	$f=me$ donde $m$ es mono y $e$ es epi (ver \cite{ML}, VIII, 3, proposition 1).

	\item Además, en una categoría abeliana, no sólo los pull-backs de los mono-\\morfismos son monomorfismos, sino que además los pull-backs de epimorfismos son epimorfismos. Como los axiomas de categoría abeliana son autoduales, también los push-outs de monos son monos y los de epis son epis.

\end{itemize}

De aquí en adelante supondremos en esta sección que $\Cat$ es abeliana. Consideremos ahora los funtores $F: \Cat \rightarrow Vec_K^{<\infty}$ y $F^\wedge: \Cat^{op} \rightarrow Vec_K^{<\infty}$. Observemos que si $F$ es fiel, $\Cat$ resulta artiniana y noetheriana. Asumamos que $F$ es exacto a izquierda, entonces $F$ preserva sucesiones exactas a izquierda, monomorfismos y límites. Recordemos en primer lugar a la categoría $\Gamma_F$ llamada diagrama de $F$, cuyos objetos son pares $(x,C)$ donde $C$ es un objeto de $\Cat$ y $x \in FC$, y cuyas flechas $(x,C) \stackrel{f}{\rightarrow} (x',C')$ son flechas $f: C \rightarrow C'$ de $\Cat$ tales que $F(f)(x) = x'$. Vía Yoneda, tenemos que los objetos de $\Gamma_F$ son transformaciones naturales $[C,-] \stackrel{x}{\Rightarrow} F$, y las flechas son $C \stackrel{f}{\rightarrow} C'$ que hacen conmutar al diagrama

\ \ \ \ \ \ \ \ \ \ \ \ \ \ \ \ \ \ \ \ \ \ \ \ \ \ \ \ \ \ \xymatrix { [C,-] \ar@2{->}[dr]^{x} \\
						& F \\
						[C',-] \ar@2{->}[ur]_{x'} \ar@2{->}[uu]^{f^*} }
						
$F$ es el colímite de su diagrama. Dados $(x,C) \in \Gamma_F$ y un subobjeto \xymatrix { B \hbox{  } \ar@{>->}[r]^{s} & C}, diremos que $x$ está contenido en el subobjeto $B$ si $x$ está en la imagen de $FB \stackrel{F(s)}{\rightarrow} FC$. Como $F$ preserva pull-backs, si $x$ está contenido en dos subobjetos $B$ y $B'$ entonces está contenido en su intersección (pull-back) $B \cap B'$: por construcción de los pull-backs en $Vec_K$, tenemos

\ \ \ \ \xymatrix@C=0.2pc@R=0.5pc { & & & & b \ar@{|->}[dddd] 											  	& & & & & &   {\exists b''} \ar@{|.>}[rrrr] \ar@{|.>}[dddd] & & & & b \ar@{|->}[dddd] \\
						& F(B \cap B') \ar@{>->}[rr] \ar@{>->}[dd] & & FB \ar@{>->}[dd] &   & & & & & &   & F(B \cap B') \ar@{>->}[rr] \ar@{>->}[dd] & & FB \ar@{>->}[dd]\\
						& & p.b. & & 																						& \ar@2{->}[rrr] & & & & & & & p.b.  \\
						& FB' \ar@{>->}[rr] & & FA & 																			& & & & & &   & FB' \ar@{>->}[rr] & & FA & \\
						b' \ar@{|->}[rrrr] & & & & x 																			& & & & & &   b' \ar@{|->}[rrrr] & & & & x }
						
Luego, si $\Cat$ es artiniana, tomar sucesivas intersecciones se estabiliza y obtenemos que $x$ debe estar contenido en un mínimo objeto de $\Cat$ al que llamamos el objeto generado por $x$ y notamos $< \! \! x \! \! >$. Consideramos la subcategoría plena $Span(F)$ de $\Gamma_F$ compuesta de aquellos pares $(C,x)$ donde $C= \ < \! \! x \! \! >$. Como tenemos (ver el diagrama anterior)

\ \ \ \ \ \ \ \ \ \ \ \ \ \ \ \ \ \ \ \xymatrix { [B,-] \ar@/^/@2{->}[drrr]^{b} \ar@2{->}[dr] \\
						& [B \cap B',-] \ar@2{->}[rr]^{b''} & & F \\
						[B',-] \ar@/_/@2{->}[urrr]^{b'} \ar@2{->}[ur] }
						
$Span(F)$ resulta cofinal en $\Gamma_F$ y por lo tanto $F$ es también el colímite de su diagrama sobre $Span(F)$. Este nuevo colímite es filtrante, el diagrama anterior muestra la condición de ser filtrante para los objetos de $Span(F)$, y la condición para las flechas se satisface pues entre todo par de objetos de $Span(F)$ hay a lo sumo una flecha: supongamos que tenemos \linebreak $f,g: (C,x) \rightarrow (C',x')$. Entonces $f$ y $g$ son flechas $f,g: C \rightarrow C'$ tales que $F(f)(x)=x'=F(g)(x)$. Consideramos el núcleo de $f-g$

\ \ \ \ \ \ \ \ \ \ \ \ \ \ \ \ \ \ \ \ \ \ \ \ \xymatrix {K \hbox{  } \ar@{>->}[r] & A \ar[rr]^{f-g} & & B}

Como $F$ es exacto a izquierda tenemos otro núcleo, ahora en $Vec_K$, y por construcción de los núcleos allí obtenemos

\ \ \ \ \ \ \ \ \ \ \ \ \ \ \ \ \ \ \ \ \ \ \ \ \xymatrix@R=0.5pc {FK \hbox{  } \ar@{>->}[r] & FA \ar[rr]^{F(f)-F(g)} & & FB \\
					 \exists x' \ar@{|.>}[r] & x \ar@{|->}[rr] & & 0}
					 
Luego $x$ está contenido en $K$, pero como $A= \ < \! \! x \! \! >$ debe ser $K=A$ y luego $f=g$. Estamos ahora en condiciones de probar

\begin{\prop} \label{colimfiltrante}
	Sean $\Cat$ una categoría abeliana y enriquecida en $Vec_K$, y \linebreak $F: \Cat \rightarrow Vec_K^{<\infty}$ un $Vec_K$-funtor exacto y fiel. Entonces los elementos del coend $F \otimes_\Cat F^\wedge = End^\lor(F)$ son $[x \otimes \phi]$ (recordemos que estos son $\lambda_C(x \otimes \phi)$), con $x \in FC$, $\phi \in FC^\wedge$ y $C$ un objeto de $\Cat$, y vale la igualdad $[x \otimes \phi] = [x' \otimes \phi']$ si y sólo si existen $(y,D) \in Span(F)$, $f: (y,D) \rightarrow (x,C)$ y $g: (y,D) \rightarrow (x',C')$ tales que $F^\wedge(f)(\phi) = F^\wedge(g)(\phi')$. 
\end{\prop}

\begin{proof}

Tenemos las flechas $[C,-] \stackrel{x}{\Rightarrow} F$ que nos dan el isomorfismo $\displaystyle \colim_{(x,C) \in \Gamma_F} [C,-] \stackrel{\cong}{\Rightarrow} F$. Multiplicando a esas flechas por $F^\wedge$ obtenemos flechas $[C,-] \otimes_\Cat F^\wedge \stackrel{\psi_x}{\Rightarrow} F \otimes_\Cat F^\wedge$ que están inducidas por

\ \ \ \ \ \ \ \ \ \ \ \ \ \ \ \ \ \ \ \xymatrix { {\int^{D \in \Cat} [C,D] \otimes F^\wedge D} \ar@{.>}[r]^{\psi_x} & {\int^{D \in \Cat} FD \otimes F^\wedge D} \\
						[C,D] \otimes F^\wedge D \ar[u]^{\lambda_D} \ar[r]^{x_D \otimes F^\wedge D} & FD \otimes F^\wedge D \ar[u]^{\lambda_D} }

Las flechas $\psi_x$ inducen un nuevo isomorfismo 
$$ \colim_{(x,C) \in \Gamma_F} [C,-] \otimes_\Cat F^\wedge \stackrel{\cong}{\Rightarrow} F \otimes_\Cat F^\wedge $$

Veremos que $[C,-] \otimes_\Cat F^\wedge = F^\wedge C$, y luego como $Span(F)$ es cofinal y cofiltrante se tendrá lo deseado. El siguiente diagrama muestra que $F^\wedge C$ es un di-cono, y dado otro di-cono $Z$ queremos construir la $\varphi$ que hace conmutar los triángulos. Dada una flecha $a: A \rightarrow A'$, tenemos el diagrama

\ \ \ \ \ \ \ \ \ \ \ \xymatrix { & F^\wedge A \otimes [C,A] \ar[dr]^{\lambda_A} \ar@/^0.5pc/[drrr]^{z_A} \\
					F^\wedge A' \otimes [C,A] \ar[ur]^{F^\wedge(a) \otimes [C,A]} \ar[dr]^{F^\wedge A' \otimes a_*} & & F^\wedge C \ar@{.>}[rr]^{\varphi} & & Z \\
						& F^\wedge A' \otimes [C,A'] \ar[ur]^{\lambda_{A'}} \ar@/_0.5pc/[urrr]^{z_{A'}} }

donde, por definición, $\lambda_A(x,\phi) = F^\wedge (\phi) (x)$ y $a_*$ es "`componer segundo con $a$"'. El primer rombo conmuta entonces por la funtorialidad de $F^\wedge$. Defi- nimos $\varphi(x)=z_C (x,id_C)$. Queremos verificar entonces que $\forall \ x \in F^\wedge A$ y para toda flecha $\phi: C \rightarrow A$, vale que $z_A (x, \phi) = z_C (F^\wedge(\phi)(x), id_C)$. Eso vale pues la condición de que $Z$ sea un di-cono para esta flecha se traduce a
	
\ \ \ \ \ \ \ \ \ \ \ \ \ \ \ \ \ \ \ \xymatrix { & F^\wedge C \otimes [C,C] \ar[dr]^{z_C}  \\
					F^\wedge A \otimes [C,C] \ar[ur]^{F^\wedge(\phi) \otimes [C,A]} \ar[dr]_{F^\wedge A' \otimes \phi_*} & & Z \\
						& F^\wedge A \otimes [C,A] \ar[ur]_{z_A} }
						
Persiguiendo a $(x,id_C)$ obtenemos lo deseado. Concluimos entonces, por la construcción de los colímites filtrantes, que los elementos de este colímite son clases de $(x,\phi)$, con la condición para la igualdad descripta en el enunciado de la proposición.
\end{proof}

La idea ahora es realizar la definición dual a la de $< \! \! x \! \! >$. Sean \linebreak $\phi \in FC^\wedge$, y \xymatrix@1 { C \ar@{->>}[r]^{c} & B \ } un cociente. Decimos que $B$ soporta a $\phi$ si $\phi$ está en la imagen de \xymatrix@1 { F^\wedge B \ar@{->>}[r]^{F^\wedge(c)} & F^\wedge C }. Usando la noetherianidad de $\Cat$, obtenemos que $\phi$ está soportada en un mínimo cociente \xymatrix@1 { C \ar@{->>}[r]^{c} & > \! \! \phi \! \! <} al que llamamos el cogenerado por $\phi$. Probaremos ahora los siguientes lemas que utilizaremos luego para demostrar el teorema fundamental.

\begin{lemma}
	Sean $\Cat$ una categoría abeliana y enriquecida en $Vec_K$, y \linebreak $F: \Cat \rightarrow Vec_K^{<\infty}$ un $Vec_K$-funtor exacto y fiel. Sean $C$ un objeto de $\Cat$, $x \in FC$ y $\phi \in F^\wedge C$ tales que $[x \otimes \phi] = 0$. Entonces existe un subobjeto \xymatrix@1 { B \hbox{  } \ar@{>->}[r]^{s} & C} tal que $x$ está contenido en $B$ y $F^\wedge(s)(\phi) = 0$.
\end{lemma}

\begin{proof}

 Consideramos la flecha \xymatrix@1 {< \! \! x \! \! > \hbox{  } \ar@{>->}[r]^{s} & C \ar@{->>}[r]^{c} & > \! \! \phi \! \! <} y su facto-\linebreak rización epi-mono
 
\ \ \ \ \ \ \ \ \ \ \ \ \ \ \ \ \ \ \ \xymatrix@1 {< \! \! x \! \! > \hbox{  } \ar@{->>}[rd]_{p} \ar@{>->}[r]^{s} & C \ar@{->>}[r]^{c} & > \! \! \phi \! \! < \\
 								& C' \hbox{  } \ar@{>->}[ru]_{i} }
 								
	Como $x$ está contenido en $< \! \! x \! \! >$, existe $x_1 \in F(< \! \! x \! \! >)$ tal que \linebreak $x=F(s)(x_1)$. Definimos $x' = F(p)(x_1)$. Análogamente, como $\phi$ está soportada en $> \! \! \phi \! \! <$, existe $\phi_1 \in F^\wedge(> \! \! \phi \! \! <)$ tal que $\phi = F^\wedge(c)(\phi_1)$. Definimos $\phi'=F^\wedge(i)(\phi_1)$. Verifiquemos que $[x' \otimes \phi'] = [x \otimes \phi]$:
	
	$$ [x' \otimes \phi'] = [F(p)(x_1) \otimes \phi'] = [x_1 \otimes F^\wedge(p)(\phi')] = [x_1 \otimes F^\wedge(ip)(\phi_1)] = \\ $$
	$$ = [x_1 \otimes F^\wedge(cs)(\phi_1)] = [x_1 \otimes F^\wedge(s)(\phi)] = [F(s)(x_1) \otimes \phi] = [x \otimes \phi] $$

Queremos ver ahora que $C'= \ < \! \! x' \! \! >$. En primer lugar, notemos que $< \! \! x_1 \! \! > \ = \ < \! \! x \! \! >$, pues se tiene

\ \ \ \ \ \ \ \ \ \ \ \ \ \ \ \ \ \ \ \ \ \ \ \ \xymatrix {< \! \! x_1 \! \! > \hbox{  } \ar@{>->}[r]^{s_1} & < \! \! x \! \! > \hbox{  } \ar@{>->}[r]^{s} & C}

tal que aplicando $F$ 

\ \ \ \ \ \ \ \ \ \ \ \ \ \ \ \ \ \ \ \ \ \ \ \ \xymatrix@R=0.5pc {F(< \! \! x_1 \! \! >) \hbox{  } \ar@{>->}[r]^{F(s_1)} & F(< \! \! x \! \! >) \hbox{  } \ar[r]^{F(s)} & FC \\
					 \exists x_2 \ar@{|->}[r] & x_1 \ar@{|->}[r] & x}
					 
Es decir, x está contenido en $< \! \! x_1 \! \! >$, pero como $< \! \! x \! \! >$ era mínimo con esta propiedad se tiene la igualdad. Consideramos entonces el pull-back

\xymatrix@C=0.3pc@R=0.5pc { & & & & 	 											  						  & & & & & &   {\exists x_P} \ar@{|.>}[rrrr] \ar@{|.>}[dddd] & & & & x'_1 \ar@{|->}[dddd] \\
						& P \ar@{->>}[rr] \ar@{>->}[dd] & & < \! \! x' \! \! > \ar@{>->}[dd] &  		  & & & & & &   & FP \ar[rr] \ar[dd] & & F(< \! \! x' \! \! >) \ar[dd] \\
						& & p.b. & & 																										& \ar@2{->}[rrr]^{F} & & & & & & & p.b.  \\
						& < \! \! x_1 \! \! > \ar@{->>}[rr] & & C' & 																	& & & & & &   & F(< \! \! x_1 \! \! >) \ar[rr] & & FC' & \\
						& & & & 		 																										& & & & & &   x_1 \ar@{|->}[rrrr] & & & & x' }

que nos dice en particular que $x_1$ está contenido en $P$, pero $< \! \! x_1 \! \! >$ era mínimo con esta propiedad, luego $P = \ < \! \! x_1 \! \! >$ y la flecha $P \rightarrow < \! \! x_1 \! \! > \rightarrow C'$ es un epimorfismo. Luego por unicidad de la factorización epi-mono, la flecha $< \! \! x' \! \! > \rightarrow C'$ debe ser un isomorfismo. 

De manera dual se demuestra que $C'= \ > \! \! \phi \! \! <$. Ahora, por la propiedad \ref{colimfiltrante}, como $[x' \otimes \phi'] = 0$, existen $(y,D) \in Span(F)$ y $f: D \rightarrow C'$, $g=0: A \rightarrow 0$ tales que $F(f)(y)=x'$ y $F^\wedge(f)(\phi')=F^\wedge(g)(0)=0$. Tomando la factorización epi-mono de $f$, y aplicándole $F$,

	\xymatrix { <y> \ar[rr]^{f} \ar@{->>}[rd]_{e} & & < \! \! x' \! \! >  & &  			F(<y>) \ar[rr]^{F(f)} \ar[rd]_{F(e)} & & F(< \! \! x' \! \! >) \\
							& B \hbox{  } \ar@{>->}[ru]_{m} & 					\ar@2{->}[rr]^{F} & & 							& FB \hbox{   } \ar@{>->}[ru]_{F(m)} }
							
Luego como $F(f)(y)=x'$ debe existir $b \in FB$ tal que $F(m)(b) = x'$, o sea $x'$ está contenido en $B$, pero como $< \! \! x' \! \! >$ era el mínimo con esta propiedad, entonces $B= \ < \! \! x' \! \! >$ y luego $f$ es epi.

Luego $F^\wedge(f)$ es mono, y como $F^\wedge(f)(\phi')=0$ obtenemos que $\phi'=0$ con lo cual $C' = \ > \! \! \phi' \! \! < \ = 0$. Luego la composición  \xymatrix@1 {< \! \! x \! \! > \hbox{  } \ar@{>->}[r]^{s} & C \ar@{->>}[r]^{c} & > \! \! \phi \! \! <} tam-\linebreak bién es $0$ y, tomando como $B$ a $< \! \! x \! \! >$, 
$$ F^\wedge(s)(\phi) = F^\wedge(s)(F^\wedge(c)(\phi')) = F^\wedge(c \circ s)(\phi') = F^\wedge(0)(\phi') = 0 $$
\end{proof}

\begin{remark}
	Siguiendo la demostración de esta última propiedad, se verá que hemos probado también que dado un elemento $\lambda_C(x \otimes \phi)$ cualquiera de $Nat^\lor(F,G)$, existe un $\lambda_C'(x' \otimes \phi') = \lambda_C(x \otimes \phi)$ tal que $C'=\ > \! \! \phi' \! \! < \ = <x>$. A este $\lambda_C'(x' \otimes \phi')$ se lo llama en \cite{JS} un "`minimal model"' de $\lambda_C(x \otimes \phi)$.
\end{remark}

\begin{lemma} \label{subobjeto}
	Sean $\Cat$ una categoría abeliana y enriquecida en $Vec_K$, y \linebreak $F: \Cat \rightarrow Vec_K^{<\infty}$ un $Vec_K$-funtor exacto y fiel. Sean $C$ un objeto de $\Cat$, y $E \subset \tilde{F}C$ un subcomódulo. Entonces existe un subobjeto \xymatrix {B \hbox{  } \ar@{>->}[r]^{s} & C} tal que $E=Im (F(s))$.
\end{lemma}

\begin{proof}

Que $E$ sea un subcomódulo quiere decir que si \linebreak $\eta: C \rightarrow End^\lor(F) \otimes C$ es el coproducto de $C$, entonces se restringe a $\eta: E \rightarrow End^\lor(F) \otimes E$. Sea $\{e_1,...,e_n\}$ una base de $FC$ como $K$-e.v. tal que $\{e_1,...,e_k\}$ es una base de $E$. Como $\eta$ está dada por (ver \eqref{etaenvec})
$$\eta(v) = \sum_{i=1}^n [v \otimes e_i^\wedge] \otimes e_i $$
que se restrinja correctamente quiere decir que $[e_j \otimes e_i^\wedge] = 0$ \linebreak $\forall 1 \leq j \leq k < i \leq n$. Usando el lema anterior, obtenemos que existen subobjetos \xymatrix {B_{ij} \hbox{  } \ar@{>->}[r]^{m_{ij}} & C} tales que $e_j \in Im(F(m_{ij})) = F(m_{ij})(B_{ij})$ y $F^\wedge(m_{ij})(e_i^\wedge) = 0$, es decir $e_i^\wedge (F(m_{ij})) = 0$.

Tomando $B = \bigcup_{j \leq k} \bigcap_{i > k} B_{ij}$, tenemos el diagrama

\ \ \ \ \ \ \ \ \ \ \ \ \ \ \ \ \ \ \ \xymatrix { & B_{ij} \ar@{>->}[d]^{m_{ij}}			\\
						\bigcap_{i > k} B_{ij} \hbox{  } \ar@{>->}[ru]^{p_{ij}} \ar@{>->}[r]^{m_j} \ar@{>->}[rd]_{i} & C					\\
						& B \ar@{>.>}[u]_{m} 						}

donde las flechas $p_{ij}$ y $m_j$ se obtienen del pull-back, $i$ es la inclusión en la unión, y $m$ se obtiene pues la unión es el subobjeto más chico que contiene a los uniendos. Usando que como $F$ preserva monomorfismos (subobjetos) entonces preserva uniones e intersecciones, obtenemos el diagrama

\ \ \ \ \ \ \ \ \ \ \ \ \ \ \ \ \ \ \ \xymatrix { & F(B_{ij}) \ar@{>->}[d]^{F(m_{ij})}			\\
						\bigcap_{i > k} F(B_{ij}) \hbox{  } \ar@{>->}[ru]^{F(p_{ij})} \ar@{>->}[r]^{F(m_j)} \ar@{>->}[rd]_{F(i)} & FC \ar[r]^{e_i^\wedge} & K				\\
						& FB 	= \oplus_{j \leq k} \bigcap_{i > k} F(B_{ij}) \ar@{>.>}[u]_{F(m)} 						}

Luego, como $e_j \in F(m_{ij})(B_{ij})$, $e_j \in F(m_j)(\bigcap_{i > k} F(B_{ij}))$ $\forall j \leq k$, y por lo tanto $e_j \in F(m)(FB)$ $\forall j \leq k$. Concluimos que $E \subset F(m)(FB)$.

Por otro lado, como $e_i^\wedge \circ F(m_{ij}) = 0$ $\forall 1 \leq j \leq k < i \leq n$, obtenemos que $e_i^\wedge \circ F(m_j) = 0$ $\forall j \leq k$ y por lo tanto $e_i^\wedge \circ F(m) = 0$ $\forall i > k$. Eso nos dice que $F(m)(FB) \subset \bigcap_{i>k} Ker(e_i^\wedge) = E$ y obtenemos la igualdad.
\end{proof}

\begin{\de}
	Una subcategoría plena de una categoría abeliana se dice repleta si es cerrada por sumas directas finitas, subobjetos y cocientes.
\end{\de}


\begin{lemma} \label{Dat}
	Dada una coálgebra $C$, la asignación $C' \mapsto Comod_{<\infty}C'$ es una biyección entre las subcoálgebras de $C$ y las subcategorías repletas de $Comod_{<\infty}C$.
\end{lemma}

\begin{proof}
	Sea $V \in Comod_{<\infty}C$, con comultiplicación $\alpha: V \rightarrow C \otimes V$. Por la ley exponencial $\alpha$ induce $V \otimes V^\wedge \stackrel{\tilde{\alpha}}{\rightarrow} C$ definida como la composición
		$$ \tilde{\alpha} : V \otimes V^\wedge \stackrel{\alpha \otimes V^\wedge}{\rightarrow} C \otimes V \otimes V^\wedge \stackrel{C \otimes \varepsilon}{\rightarrow} C $$
	
	$V \otimes V^\wedge$ es una coálgebra con counidad $\varepsilon$ y comultiplicación \linebreak $V \otimes V^\wedge \stackrel{V \otimes \eta \otimes V^\wedge}{\rightarrow} V \otimes V^\wedge \otimes V \otimes V^\wedge$ claramente asociativa. Las igualdades triangulares de la dualidad nos dicen que $\varepsilon$ es counidad. Veremos que $\tilde{\alpha}$ es un morfismo de coálgebras. $\alpha$ hacía conmutar al cuadrado
	
	\ \ \ \ \ \ \ \ \ \ \ \ \ \ \ \ \ \ \ \xymatrix {V \ar[d]_{\alpha} \ar[r]^{\alpha} & C \otimes V \ar[d]^{C \otimes \alpha} \\
						 C \otimes V \ar[r]^{\triangle \otimes V} & C \otimes C \otimes V }
						 
	Apliquemos entonces la ley exponencial a ambas composiciones. Al aplicarla a la flecha $(\triangle \otimes V) \circ \alpha$ obtenemos el diagrama
	
\ \ \ \ \ \ \ \ \ \ \ \ \ \ \ \ \ \ \ \ \ \ \ \ \ \ \ \ \ \ \xymatrix@C=-0.2pc{ & & & & V \ar@{-}[lld] \ar@{-}[d] \ar@{}[ld]|{\alpha} & & V^\wedge \ar@2{-}[d] & & & & & & \\
								 & & C \ar@{-}[lld] \ar@{-}[d] \ar@{}[ld]|{\triangle} & & V \ar@2{-}[d] & & V^\wedge \ar@2{-}[d] & & & & & & \\	
								 C \ar@2{-}[d] & & C \ar@2{-}[d] & & V \ar@{-}[rd] & \ar@{}[d]|{\varepsilon} & V^\wedge \ar@{-}[ld] & & & & & \\							 
								 C & & C & & & & & & & & & }
						
	que resulta equivalente (juntando $\alpha$ con $\varepsilon$) a $\triangle \circ \tilde{\alpha}$.
	
	Al aplicar la ley exponencial a la flecha $(C \otimes \alpha) \circ \alpha$, obtenemos
	
\ \ \ \ \ \ \ \ \ \ \ \ \ \ \ \ \ \ \ \ \ \ \ \ \ \ \ \ \ \ \xymatrix@C=-0.2pc{ & & & & V \ar@{-}[lllld] \ar@{-}[d] \ar@{}[lld]|{\alpha} & & V^\wedge \ar@2{-}[d] & & & & & \\
								 C \ar@2{-}[d] & & & & V \ar@{-}[lld] \ar@{-}[d] \ar@{}[ld]|{\alpha} & & V^\wedge \ar@2{-}[d] & & & & & \\	
								 C \ar@2{-}[d] & & C \ar@2{-}[d] & & V \ar@{-}[rd] & \ar@{}[d]|{\varepsilon} & V^\wedge \ar@{-}[ld] & & & & & \\							 
								 C & & C & & & & & & & & & }
	
	Insertamos en el diagrama la segunda igualdad triangular de la dualidad a derecha de $V$ y obtenemos (ya reemplazamos el segundo $\alpha$ y el $\varepsilon$ por $\tilde{\alpha}$)
	
\ \ \ \ \ \ \ \ \ \ \ \ \ \ \ \ \ \ \ \ \ \ \ \ \ \ \ \ \ \ \xymatrix@C=-0.2pc{ & & V \ar@{-}[lld] \ar@{-}[d] \ar@{}[ld]|{\alpha} & & & & & & V^\wedge \ar@2{-}[d] & & & \\
								 C \ar@2{-}[d] & & V \ar@2{-}[d] & & & \ar@{-}[ld] \ar@{-}[rd] \ar@{}[d]|{\eta} & & & V^\wedge \ar@2{-}[d] & & & \\	
								 C \ar@2{-}[d] & & V \ar@{-}[rd] & \ar@{}[d]|{\varepsilon} & V^\wedge \ar@{-}[ld] & & V \ar@2{-}[d] & & V^\wedge \ar@2{-}[d] & & & \\							 
								 C \ar@2{-}[d] & & & & & & V \ar@{-}[rd] & \ar@{}[d]|{\tilde{\alpha}} & V^\wedge \ar@{-}[ld] & & & \\							 
								 C & & & & & & & C & & & & }
											
	Que reemplazando nuevamente $\alpha$ y $\varepsilon$ por $\tilde{\alpha}$ nos queda
	
\ \ \ \ \ \ \ \ \ \ \ \ \ \ \ \ \ \ \ \ \ \ \ \ \ \ \ \ \ \ \xymatrix@C=-0.4pc{ V \ar@2{-}[d] & & & \ar@{-}[ld] \ar@{-}[rd] \ar@{}[d]|{\eta} & & & V^\wedge \ar@2{-}[d] & & & & & \\
								 V \ar@{-}[rd] & \ar@{}[d]|{\tilde{\alpha}} & V^\wedge \ar@{-}[ld] & & V \ar@{-}[rd] & \ar@{}[d]|{\tilde{\alpha}} & V^\wedge \ar@{-}[ld] & & & & & \\	
								 & C & & & & C & & & & & & }
																					
	Luego obtenemos la conmutatividad del siguiente cuadrado 
	
	\ \ \ \ \ \ \ \ \ \ \ \ \ \ \ \ \ \ \ \xymatrix {V \otimes V^\wedge \ar[d]_{\tilde{\alpha}} \ar[rr]^{V \otimes \eta \otimes V^\wedge} & & V \otimes V^\wedge \otimes V \otimes V^\wedge \ar[d]^{\tilde{\alpha} \otimes \tilde{\alpha}} \\
						 C  \ar[rr]^{\triangle} & & C \otimes C }
						 
	que nos dice que $\tilde{\alpha}$ respeta la comultiplicación. Dejamos por verificar que respeta la counidad.
	
	Sea ahora $\Dat$ una subcategoría repleta de $Comod_{<\infty}C$. Definimos la aplicación inversa a la del enunciado de la proposición como $\Dat \mapsto Im(\Dat)$, donde $Im(\Dat) = \oplus_{(V,\alpha) \in \Dat} Im(\tilde{\alpha})$.
	
	Si $C'$ es una subcoálgebra de $C$, $\forall (V,\alpha) \in \Dat$, se tiene \linebreak $\tilde{\alpha}(V \otimes V^\wedge) \subset C' \Longleftrightarrow \alpha: V \rightarrow C' \otimes V$, de donde se deduce que \linebreak $Im(\Dat) \subset C' \Longleftrightarrow D \subset Comod_{<\infty} C'$. Luego $Im(Comod_{<\infty}C') \subset C'$ y $\Dat \subset Comod_{<\infty} (Im(\Dat))$. Nos quedan por verificar las otras dos inclusiones.
	
	Observemos que cualquier coálgebra es la unión de sus subcoálgebras de dimensión finita: para probar esto consideramos a la coálgebra $C$ como un $C^{op} \otimes C$-comódulo, donde $C^{op}$ es la coálgebra opuesta (ver \cite{JS}, 7, p.454, Proposition 1), luego sus subcomódulos coinciden con las subcoálgebras de $C$ y ya hemos probado (al demostrar la propiedad \ref{rhosury}) que dado un elemento de un comódulo se tiene un subcomódulo de dimensión finita que lo contiene. (También se puede consultar \cite{Sweedler}, theorem 2.2.1, p.46 para una demostración "`independiente"' de ese hecho). Luego, para probar que $C' \subset Im(Comod_{<\infty}C')$, bastará con ver que cualquier subcoálgebra $V$ de dimensión finita de $C'$ está contenida en $Im(Comod_{<\infty} C')$.
	
	Supongamos que la estructura de coálgebra de $V$ viene dada por $\triangle$ y $\varepsilon_V$ (para no confundirlo con el $\varepsilon$ de la dualidad). Tomando una base $\{v_1,...,v_n\}$ de $V$, escribimos $\triangle(v)= \sum_{i,j=1}^n \alpha_{ij} \cdot v_i \otimes v_j$ y que $\varepsilon_V$ sea counidad se puede expresar como $v = \sum_{i,j=1}^n \alpha_{ij} \cdot \varepsilon_V (v_j) \cdot v_i$. Luego, si $\tilde{\triangle}: V \otimes V^\lor \rightarrow V$ se obtiene de $\triangle$ por ley exponencial,
	$$\tilde{\triangle}(v \otimes \varepsilon_V) = (id_V \otimes \varepsilon) (\sum_{i,j=1}^n \alpha_{ij} \cdot v_i \otimes v_j \otimes \varepsilon_V) = \sum_{i,j=1}^n \alpha_{ij} \cdot \varepsilon_V(v_j) \cdot v_i  = v $$
	Por lo tanto $\tilde{\triangle}$ es suryectiva y $V \subset Im(Comod_{<\infty} C')$.
	
	Veamos ahora que $Comod_{<\infty} (Im(\Dat)) \subset \Dat$. Llamemos $C' = Im(\Dat)$. Sea $(V,\alpha) \in \Dat$. Como 
	$$ \tilde{\alpha} : V \otimes V^\wedge \rightarrow C $$
	$$ v \otimes \phi \mapsto (id_C \otimes \phi)(\alpha(v)) $$
	
	$Im(\tilde{\alpha})$ está generado por las imágenes de $V \stackrel{\alpha}{\rightarrow} C \otimes V \stackrel{C \otimes \phi}{\rightarrow} C$, con $\phi \in V^\wedge$ (por ejemplo una base). Luego, como $\Dat$ es cerrado por cocientes y $V / Ker( (id_C \otimes \phi) \circ \alpha) \cong Im( (id_C \otimes \phi) \circ \alpha)$, obtenemos que $Im(\tilde{\alpha}) \in \Dat$ y como $C' = \oplus_{(V,\alpha) \in \Dat} Im (\tilde{\alpha})$ y $\Dat$ es cerrada por sumas directas finitas, cualquier subcomódulo de $C'$ de dimensión finita pertenece a $\Dat$, y eso basta para probar que $C' \in \Dat$. Como $\Dat$ es también cerrada por subobjetos, cualquier subcomódulo de $C'$ también está en $\Dat$.
	
	Sea ahora entonces $(W,\beta) \in Comod_{<\infty} C'$, y $\{ \phi_1,...,\phi_n \}$ una base de $W^\wedge$. Consideramos las imágenes $W_i$ de las flechas $W \stackrel{\beta}{\rightarrow} C' \otimes W \stackrel{C' \otimes \phi_i}{\rightarrow} C'$. Los $W_i$ pertenecen a $\Dat$ pues son subcomódulos de $C'$, factorizando estas flechas tenemos $W \stackrel{\rho_i}{\rightarrow} W_i$ las imágenes. Las juntamos todas en $W \stackrel{\oplus \rho_i}{\rightarrow} \oplus W_i$, luego como $\Dat$ es completa basta ver que $\oplus \rho_i$ es un monomorfismo. En efecto, si $w \mapsto \oplus 0$, entonces como $\{\phi_1,...\phi_n\}$ es una base de $W^\wedge$ obtenemos que $\beta(w) = 0$, pero como $\beta$ respeta la counidad de $C'$ tenemos que $\varepsilon \circ \beta$ es un isomorfismo lo que implica que $\beta$ es mono y que $w = 0$.	
\end{proof}

\subsection{Demostración de los teoremas}

En esta sección demostraremos lo que nos resta de los teoremas fundamentales en el caso $\Vat = Vec_K$, usando las herramientas y lemas de la sección anterior.

\begin{\te} 
	Sea $C$ una $K$-coálgebra y $U: Comod_{<\infty}C \rightarrow Vec_K^{<\infty}$ el funtor de olvido. Entonces $\tilde{\rho}: End^\lor(U) \rightarrow C$ es inyectivo. Como ya hemos probado que es un morfismo de $K$-coálgebras y que es suryectivo, resulta entonces un isomorfismo de $K$-coálgebras.
\end{\te}

\begin{proof}
	
	Consideremos $\lambda_V(v, \phi) \in End^\lor(U)$. Por definición, otro \linebreak $C$-comódulo $M$ soporta a $\phi$ si tenemos un epimorfismo de comódulos $V \stackrel{f}{\rightarrow} W$ tal que $\phi \in Im(M^\lor \stackrel{f^\lor}{\rightarrow}	V^\lor)$. Por definición de la adjunta, esto es si y sólo si existe $\psi \in M^\lor$ tal que $\phi = \psi \circ f$. Luego, el siguiente diagrama conmuta
	
	\ \ \ \ \ \ \ \ \ \ \ \ \xymatrix { V \ar[r]^{\rho_V} \ar@{->>}[rdd]_{f} & C \otimes V \ar[rr]^{C \otimes \phi} & & C \\
							& & C \otimes M \ar[ru]_{C \otimes \psi} \\
							& M \ar[ru]_{\rho_M} }

	Tomamos ahora factorización epi-mono de la flecha $(C \otimes \phi) \circ \rho_V$ y luego de la flecha $(C \otimes \psi) \circ \rho_W$. Por unicidad de la factorización epi-mono, el diagrama debe quedar

	\ \ \ \ \ \ \ \ \ \ \ \ \xymatrix { V \ar[r]^{\rho_V} \ar@{->>}[rdd]^{f} \ar@/_1pc/@{->>}[rddd]_{e} & C \otimes V \ar[rr]^{C \otimes \phi} & & C \\
							& & C \otimes M \ar[ru]^{C \otimes \psi} \\
							& M \ar[ru]^{\rho_M} \ar@{->>}[d] \\
							& A \ar@/_2pc/@{>->}[rruuu]_{m}  }

	Por lo tanto $A$ es más chico que cualquier subobjeto de $V$ que soporte a $\phi$. Para ver que $A = \ > \! \! \phi \! \! <$, tenemos que definir $\psi \in A^\lor$ tal que \linebreak $\psi(e(v)) = \phi(v)$ para todo $v \in V$. La única definición posible es la siguiente, dado $a \in A$, existe $v$ tal que $a = e(v)$, luego definimos $\psi(a) = \phi(v)$. Para ver la buena definición, debemos ver que si $e(v) = 0$, entonces $\phi(v) = 0$. Por la conmutatividad del diagrama, tenemos que $(C \otimes \phi) \circ \rho_V (v) = 0$. Si escribimos $\rho_V(v) = \displaystyle \sum_i e_i \otimes v_i$, entonces tenemos que $\displaystyle \sum_i \phi(v_i) \cdot e_i = 0$. Aplicando la counidad $\varepsilon$, obtenemos $\displaystyle \sum_i \phi(v_i) \cdot \varepsilon(e_i) = 0$. Pero como $\rho_V$ respeta la counidad, tenemos que $v = \displaystyle \sum_i \varepsilon(e_i) \cdot v_i$, luego $\phi(v) = \displaystyle \sum_i \varepsilon(e_i) \cdot \phi(v_i) = 0$.
	
	Luego, si $V = \ > \! \! \phi \! \! <$, la flecha $(C \otimes \phi) \circ \rho_V$ es inyectiva. Esta asunción la podemos realizar pues cada elemento de $End^\lor(U)$ tiene un minimal model cuya clase coincide con él. Supongamos entonces que $\tilde{\rho}(\lambda_V(v,\phi))=0$. Entonces, por el diagrama \eqref{defderho}, tenemos que si $\rho_V(v) = \displaystyle \sum_i e_i \otimes v_i$ al igual que antes, entonces $\displaystyle \sum_i e_i \otimes \phi(v_i) = 0$. Pero este valor coincide con $(C \otimes \phi) \circ \rho_V (v)$, luego $v = 0$ y entonces $\lambda_V(v,\phi)=0$. Por lo tanto $\tilde{\rho}$ es inyectiva.
\end{proof}

\begin{\te} \label{Teo2'}
	Sean $\Cat$ una categoría abeliana y enriquecida en $Vec_K$, y \linebreak $F: \Cat \rightarrow Vec_K^{<\infty}$ un $Vec_K$-funtor exacto y fiel. Entonces el levantamiento $\tilde{F}$ descripto al comienzo de esta sección es una equivalencia de categorías. 
\end{\te}

\begin{proof}

Hay que ver que $\tilde{F}$ es pleno (obviamente es fiel) y que es cuasisuryectivo en los objetos. Para ver que es pleno, tomemos \linebreak $f: \tilde{F}(A) = FA \rightarrow \tilde{F}(B) = FB$ un morfismo de $End^\lor(F)$-comódulos. Eso significa que el diagrama

	\ \ \ \ \ \ \ \ \ \ \ \ \ \ \ \ \ \ \ \xymatrix { FA \ar[d]_{\eta_A} \ar[rr]^{f} & & FB \ar[d]^{\eta_B} \\
							 End^\lor(F) \otimes FA \ar[rr]^{End^\lor(F) \otimes f} & & End^\lor(F) \otimes FB}
							
conmuta. El gráfico de $f$ es la imagen de la flecha \linebreak $G_f: FA \stackrel{\delta}{\rightarrow} FA \oplus FA \stackrel{FA \oplus f}{\rightarrow} FA \oplus FB$. Veamos que esta flecha es morfismo de $End^\lor(F)$-comódulos. En efecto, el diagrama anterior nos da la conmutatividad del siguiente diagrama

\footnotesize
	\xymatrix@C=0.1pc { FA \ar[r]^{\delta} \ar[dd]_{\eta_A} & FA \oplus FA \ar[d]^{\eta_A \oplus \eta_A} \ar[r]^{FA \oplus f} & FA \oplus FB \ar[d]^{\eta_A \oplus \eta_B} \\
																			& (End^\lor(F) \otimes FA) \oplus (End^\lor(F) \otimes FA) \ar@2{-}[d] & (End^\lor(F) \otimes FA) \oplus (End^\lor(F) \otimes FB) \ar@2{-}[d] \\
																			End^\lor(F) \otimes FA \ar[r]^{End^\lor(F) \otimes \delta} & End^\lor(F) \otimes (FA \oplus FA) \ar[r]^{End^\lor(F) \otimes f} & End^\lor(F) \otimes (FA \oplus FB) }
\normalsize

que muestra que $G_f$ es morfismo de $End^\lor(F)$-comódulos y por lo tanto el gráfico de $f$ es un subcomódulo de $FA \oplus FB \cong F(A \oplus B)$. Luego, por el lema \ref{subobjeto}, existe \xymatrix { C \hbox{  } \ar@{>->}[r]^{s} & A \oplus B } tal que el gráfico de $f$ es la imagen de $F(s)$. Tomemos 
$$ i: C \stackrel{s}{\rightarrow} A \oplus B \stackrel{p_1}{\rightarrow} A $$
Al aplicar $F$ obtenenemos
$$ F(i): FC \stackrel{F(s)}{\rightarrow} A \oplus B \stackrel{\pi_1}{\rightarrow} A $$
$$ c \mapsto (a,f(a)) \mapsto a $$
Como la imagen de $F(s)$ es el gráfico de $f$, $F(i)$ resulta un isomorfismo, luego $Ker(F(i))=Coker(F(i))=0$, pero como $F$ preserva núcleos y conúcleos y es fiel, $Ker(i)=Coker(i)=0$ y luego $i$ es un isomorfismo. Tomando entonces
$$ u: A \stackrel{i^{-1}}{\rightarrow} C \stackrel{s}{\rightarrow} A \oplus B \stackrel{p_2}{\rightarrow} B $$
Al aplicar $F$ nos queda
$$ F(u): FA \stackrel{F(i)^{-1}}{\rightarrow} FC \stackrel{F(s)}{\rightarrow} FA \oplus FB \stackrel{\pi_2}{\rightarrow} FB $$
$$ a \mapsto c \mapsto (a,f(a)) \mapsto f(a) $$
Por lo tanto $F(u)=f$.

Resta ver que cualquier objeto de $Comod_{<\infty}(End^\lor(F))$ es isomorfo a algún $\tilde{F}(C)$. Sea $\Dat$ la subcategoría plena de los comódulos que son isomorfos a algún $\tilde{F}(C)$. El lema \ref{subobjeto} nos dice que $\Dat$ es cerrada por subobjetos, un argumento dual muestra que es cerrada por cocientes, y como $F$ es exacto es cerrado por sumas directas. Luego $\Dat$ es repleta y entonces, por el lema \ref{Dat}, existe una subcoálgebra $C' \stackrel{i}{\hookrightarrow} End^\lor(F)$ tal que $\Dat = Comod_{<\infty} C'$. Esto implica que para todo $A \in \Cat$, como $FA \in \Dat$ entonces $FA$ es también un $C'$-comódulo, o sea tenemos el diagrama conmutativo

\ \ \ \ \ \ \ \ \ \ \ \ \ \ \ \ \ \ \ \xymatrix { FA \ar[dr]_{\eta_A} \ar[rr]^{\eta_A} & & C' \otimes FA \ar[dl]^{i \otimes FA} \\
							& End^\lor(F) \otimes FA }
							
Pero por la adjunción de la propiedad \ref{adjuncionpredual}, dada la estructura de comódulo $F \stackrel{\eta}{\rightarrow} C' \otimes F$ existe una única $End^\lor(F) \stackrel{\psi}{\rightarrow} C'$ tal que el siguiente diagrama conmuta

\ \ \ \ \ \ \ \ \ \ \ \ \ \ \ \ \ \ \ \xymatrix { F \ar@2{->}[dr]_{\eta} \ar@2{->}[rr]^{\eta} & & End^\lor(F) \otimes F \ar@2{->}[dl]^{i \otimes F} \\
							& C' \otimes F }
							
Luego $i \circ \psi$ hace conmutar al diagrama 

\ \ \ \ \ \ \ \ \ \ \ \ \ \ \ \ \ \ \ \xymatrix { F \ar@2{->}[dr]_{\eta} \ar@2{->}[rr]^{\eta} & & End^\lor(F) \otimes F \ar@2{->}[dl]^{(i \circ \psi) \otimes F} \\
							& End^\lor(F) \otimes F }

pero las t.n. que hacen esto también son únicas (por la misma adjunción considerando ahora la estructura de $End^\lor(F)$-comódulo para los $FA$) y la identidad también lo cumple. Luego $i \circ \psi = id_{End^\lor(F)}$ y luego como $i$ era una inclusión se deduce que $C' = End^\lor(F)$ y el funtor $F$ es cuasisuryectivo.

\end{proof}

\pagebreak
\section{Del Otro Lado de la Dualidad} \label{sec:OtroLadoDualidad}

\subsection{El caso de las álgebras simétricas}

Sea como antes $\Vat$ una categoría tensorial simétrica cocompleta con hom internos, y consideremos la subcategoría plena $Alg^{Sim}_\Vat$ de las álgebras conmutativas. Cuando miremos un objeto $A$ de $Alg^{Sim}_\Vat$ en ${Alg^{Sim}_\Vat}^{op}$ lo notaremos $Spec(A)$, $\bar{A}$, o bien $G$ pues veremos próximamente que en el caso que nos interesa es un objeto grupo.

Recordemos que, en una categoría $\Cat$ con objeto terminal $1$ y productos finitos, un objeto $G$ de $\Cat$ es un objeto grupo si se tienen flechas multiplicación $\mu: G \times G \rightarrow G$ asociativa, unidad $\eta: 1 \rightarrow G$ e inversa $\zeta: G \rightarrow G$ tales que los siguientes diagramas conmutan:

Asoc. \xymatrix{ G \times G \times G \ar[d]_{\mu \times G} \ar[r]^{G \times \mu} & G \times G \ar[d]^{\mu} \\
						G \times G \ar[r]^{\mu} & G} 
\ \ \ Unid. \xymatrix{ 1 \times G \ar[dr]_{\cong} \ar[r]^{\eta \times G} & G \times G \ar[d]^{\mu} & G \times 1 \ar[l]_{G \times \eta} \ar[dl]^{\cong} \\
							& G } 
							
Inv.
\xymatrix{ G \ar[d] \ar[r]^{\delta} & G \times G \ar@/^/[r]^{G \times \zeta} \ar@/_/[r]_{\zeta \times G} & G \times G \ar[d]^{\mu} \\
					 1 \ar[rr]^{\eta} & & G } 
\\ donde $\delta$ es la diagonal.

Como se vio en la propiedad \ref{AlgSimTens}, la categoría $Alg^{Sim}_\Vat$ hereda el producto tensorial $\otimes$ de $\Vat$, con la propiedad extra de que éste resulta un coproducto. Luego, el producto en ${Alg^{Sim}_\Vat}^{op}$ corresponderá al coproducto $\otimes$ de $Alg^{Sim}_\Vat$. Además la antípoda, que suele ser un antimorfismo de álgebras, es también un morfismo de álgebras si estas son conmutativas. Tenemos entonces las siguientes equivalencias:

\vspace{1ex}

		\begin{tabular}{|c|c|}
		\hline
			$Alg^{Sim}_\Vat$ &			${Alg^{Sim}_\Vat}^{op}$ \\
		\hline
			$I$ es objeto inicial, con flechas $I \stackrel{u}{\rightarrow} A$  & 			$\bar{I}$ es objeto terminal \\
		\hline
			una comultiplicación & una multiplicación \\
		$A \stackrel{\triangle}{\rightarrow} A \otimes A$ coasociativa & $\bar{A} \times \bar{A} \stackrel{\mu}{\rightarrow} \bar{A}$ asociativa \\
		\hline
			una counidad $A \stackrel{\varepsilon}{\rightarrow} I$ & una unidad $\bar{I} \stackrel{\eta}{\rightarrow} \bar{A}$ \\
		\hline
			$\therefore$ $B$ una biálgebra & $\bar{B}$ un monoide \\
		\hline
			una antípoda $B \stackrel{a}{\rightarrow} B$ & una inversa $\bar{B} \stackrel{\zeta}{\rightarrow} \bar{B}$ \\
		\hline
			$\therefore$ $B$ un álgebra de Hopf & $\bar{B}$ un objeto grupo \\
		\hline
		\end{tabular}

\vspace{1ex}

Se sugiere al lector mirar los tres diagramas anteriores dando vuelta las flechas para obtener los diagramas de coasociatividad, counidad y antípoda.

Notamos que, en caso clásico en el que $\Vat = Vec_K$ y $\Vat_0 = Vec_K^{<\infty}$, la categoría ${Alg^{Sim}_\Vat}^{op}$ es la de los esquemas afines, y un objeto grupo en esta categoría es un esquema grupo ("`group scheme"'). Remitimos al lector interesado a \cite{Hartshorne}, entre otras muchas opciones existentes.

\subsection{El objeto grupo $G=Spec(End^\lor(F))$}

El objetivo ahora es generalizar lo hecho en la sección anterior al caso en el que el álgebra de Hopf no es simétrica como álgebra. Si $\Cat$ es una categoría tensorial con una dualidad a derecha y $F: \Cat \rightarrow \Vat_0$ es un funtor tensorial, hemos visto que $End^\lor(F)$ es un álgebra de Hopf. Sin embargo, para probar que $End^\lor(F)$ es simétrica como álgebra necesitaríamos la hipótesis de que $\Cat$ sea simétrica y que $F$ respete esta simetría, es decir que $F(\psi_{C,D}) = \psi_{FC,FD}$, pues para que $m = m \circ \psi$ necesitamos la conmutatividad del siguiente rombo que se obtiene en ese caso por coend

\xymatrix{	& FC \otimes FD \otimes FD^\wedge \otimes FC^\wedge \ar[dr]_{\lambda_{C \otimes D}} \\
				FD \otimes FC \otimes FD^\wedge \otimes FC^\wedge \ar[ur]_{\psi \otimes FD^\wedge \otimes FC^\wedge} \ar[dr]^{FD \otimes FC \otimes \psi^\wedge} & & End^\lor(F) \\
						& FD \otimes FC \otimes FC^\wedge \otimes FD^\wedge \ar[ur]^{\lambda_{D \otimes C}} }

Lo que haremos para evitar agregar esas hipótesis es dar la siguiente definición "`à la Grothendieck"' de objeto grupo.

\begin{\de}
	Sea $\Cat$ una categoría. Consideramos el funtor de Yoneda $$ \Cat \stackrel{h}{\hookrightarrow} \Ens^{\Cat ^{op}} $$ $$ C \mapsto [-,C] $$
	Un objeto $C$ de $\Cat$ es un objeto grupo si $h(C)$ es un objeto grupo en la categoría $\Ens^{\Cat ^{op}}$.
\end{\de}

\begin{remark}
	Que $h(C)$ sea un objeto grupo en la categoría $\Ens^{\Cat ^{op}}$ es equivalente a que el funtor $h(C)$ se levante a $\Gr^{\Cat ^{op}}$, es decir que $[A,C]$ sea un grupo para todo $A$ y que $[A',C] \stackrel{f^*}{\rightarrow} [A,C]$ sea un morfismo de grupos para toda $A \stackrel{f}{\rightarrow} A'$.
\end{remark}

Tenemos entonces la siguiente propiedad

\begin{\prop}
	Sea $B$ un álgebra de Hopf en $\Vat$. Entonces $\bar{B}$ (al que también notaremos $Spec(B)$ o $G$) es un objeto grupo en $Alg_\Vat^{op}$.
\end{\prop}

\begin{proof}
	Debemos ver que $[B,A]$ es un grupo para cada $A$ álgebra en $\Vat$ y que $[B,A] \stackrel{f_*}{\rightarrow} [B,A']$ es un morfismo de grupos para cada $A \stackrel{f}{\rightarrow} A'$ morfismo de álgebras. La estructura de grupo que tomamos en $[B,A]$ es el producto de convolución $f*g = m \circ (f \otimes g) \circ \triangle$, que ya hemos visto en la sección \ref{BialgyHopf} que tiene a $B \stackrel{\epsilon}{\rightarrow} I \stackrel{u}{\rightarrow} A$ como unidad.

	El diagrama que expresa que $B \stackrel{a}{\rightarrow} B$ es antípoda, compuesto con un morfismo de álgebras $B \stackrel{f}{\rightarrow} A$ nos queda
	
\ \ \ \ \ \ \ \ \ \ \ \ \ \ \ \ \ \ \ \ \ \ \ \ \ \ 
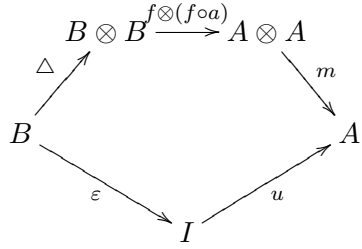
\xymatrix @C=0.5pc {	& B \otimes B \ar[rr]^{f \otimes (f \circ a)} & & A \otimes A \ar[dr]^{m} \\
						B \ar[ur]^{\triangle} \ar[drr]_{\varepsilon} & & & & A \\
						& & I \ar[urr]_{u}  }
	
	y nos dice entonces que la inversa de $f$ para $*$ es $f \circ a$.
	
	Por último, si $A \stackrel{f}{\rightarrow} A'$ es un morfismo de álgebras, que $f$ respete la unidad $u$ nos dice que $f_*$ respeta la unidad del grupo, y que $f$ respete la multiplicación $m$ nos dice que $f_*$ respeta el producto $*$.
\end{proof}

Luego, obtenemos la siguiente versión de la propiedad \ref{EndesHopf}:

\begin{\prop}
	Sea $\Cat$ una categoría tensorial con una dualidad a derecha y \linebreak $F: \Cat \rightarrow \Vat_0$ un funtor tensorial. Entonces $G=Spec(End^\lor(F))$ es un objeto grupo en la categoría ${Alg_\Vat}^{op}$.
\end{\prop}

Nos interesa ahora interpretar el levantamiento de $F$ realizado en la sección \ref{sec:levdeF}, del lado "`geométrico"' de la dualidad, es decir en $Alg_\Vat^{op}$. 

\subsection{Representaciones de un objeto grupo}

Sea $B$ un álgebra de Hopf en $\Vat$, $V$ un objeto de $\Vat_0$ (podría ser de $\Vat$ también), y $G=Spec(B)$ el objeto grupo en ${Alg_\Vat}^{op}$. Vamos a definir las representaciones (a derecha) de $G$ en $V$.

Vía el funtor de Yoneda, tenemos $G \in \Gr^{Alg_\Vat}$ que queda definido como $G=[B,-]$. Para definir las representaciones a derecha, debemos considerar al objeto grupo $G^{op}$ que difiere de $G$ en que su multiplicación viene dada por $f\bar{*}g = m \circ (g \otimes f) \circ \triangle$.



Con respecto a $V$, para cualquier álgebra $A$ en $\Vat$ tenemos la adjunción "`extensión por escalares"'

\ \ \ \ \ \ \ \ \ \ \ \ \ \ \ \ \ \ \ \ 
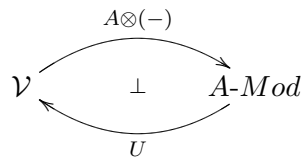
\xymatrix@C=5pc{ \Vat \ar@/^1.5pc/[r]^{A \otimes (-)} \ar@{}[r]|{\bot} & A\hbox{-}Mod \ar@/^1.5pc/[l]^{U} }

donde el producto escalar de $A \otimes V$ viene dado por $A \otimes A \otimes V \stackrel{m}{\rightarrow} A \otimes V$

En efecto, podemos ver la biyección entre las flechas

\begin{equation} \label{biyeccionAV}
\xymatrix { {}  \ar@<-4ex>@{-}[rrr] & A \otimes V \ar[r]^>>>>{f} & M \hbox{ morfismo de } A \hbox{-}Mod & {} \ar@/^2pc/[d]^{\theta}\\
						{} \ar@/^2pc/[u]^{\theta} & V \ar[r]^>>>>>>>>>>{g} & M \hbox{ flecha en } \Vat & {} }
\end{equation}

donde $\theta(f)= V \stackrel{u \otimes V}{\rightarrow} A \otimes V \stackrel{f}{\rightarrow} M$ y $\theta(g) = A \otimes V \stackrel{A \otimes g}{\rightarrow} A \otimes M \stackrel{\gamma}{\rightarrow} M$

Se verifica que $\theta \circ \theta (f) = A \otimes V \stackrel{A \otimes u \otimes V}{\rightarrow} A \otimes A \otimes V \stackrel{A \otimes f}{\rightarrow} A \otimes M \stackrel{\gamma}{\rightarrow} M$, como $\gamma$ es morfismo de $A$-módulos, es igual a 
$$A \otimes V \stackrel{A \otimes u \otimes V}{\rightarrow} A \otimes A \otimes V \stackrel{m \otimes V}{\rightarrow} A \otimes V \stackrel{f}{\rightarrow} M$$
que es igual a $f$ pues $u$ es unidad para $m$.

También tenemos $\theta \circ \theta (g) = V \stackrel{u \otimes V}{\rightarrow} A \otimes V \stackrel{A \otimes g}{\rightarrow} A \otimes M \stackrel{\gamma}{\rightarrow} M$, que por funtorialidad del producto tensorial es igual a
$V \stackrel{g}{\rightarrow} M \stackrel{u \otimes V}{\rightarrow} A \otimes M \stackrel{\gamma}{\rightarrow} M$, que es igual a $g$ pues $u$ también es unidad para $\gamma$.

En otras palabras, tenemos que un morfismo de $A$-módulos $A \otimes V \stackrel{\phi}{\rightarrow} M$ queda determinado por su restricción a $V$, es decir

\ \ \ \ \ \ \ \ \ \ \ \ \xymatrix { V \ar[r]^{u \otimes V} \ar[dr]_{\forall} & A \otimes V \ar[d]^{\phi} & A \otimes V \ar@{.>}[d]^{\exists ! \phi}_{\Leftarrow} \\
							& M & M }



Luego podemos definir $GL_V (A) = Aut_{A-Mod} (A \otimes V)$, es decir $GL_V$ en un $A$ es el $GL$ (es decir el grupo de automorfismos) de $V$ visto en los \linebreak $A$-módulos. Así obtenemos el valor en los objetos de un funtor \linebreak $GL_V \in \Ens^{Alg_\Vat}$. Veamos su definición en las flechas: sea $f: A \rightarrow A'$ un morfismo de álgebras, dado $\phi \in Aut_A(A \otimes V)$ debemos definir \linebreak $GL_V(f)(\phi) \in Aut_{A'}(A' \otimes V)$. Como estas flechas quedan definidas por su restricción a $V$, obtenemos el diagrama

\begin{equation} \label{diagramaGL}
\xymatrix @C=4pc { V \ar[r]^{u \otimes V} \ar[dr]_{u \otimes V} & A \otimes V \ar[r]^{\phi} \ar[d]^{f \otimes V} & A \otimes V \ar[d]^{f \otimes V} \\
								& A' \otimes V \ar@{.>}[r]^{\exists ! GL_V(f)(\phi)} & A' \otimes V }
\end{equation}	
							
que define $GL_V(f)(\phi)$. Notemos que $GL_V(f)(\phi)$ está definida como la única flecha que hace conmutar al diagrama exterior en \eqref{diagramaGL} (es decir el trapecio), pero a priori no tendría por qué hacer conmutar al cuadrado de la derecha. Sin embargo el siguiente lema (usado en el caso $V=W$, $\varphi=id$) nos dice que las dos condiciones son equivalentes y por lo tanto vamos a poder pensar también que $GL_V(f)(\phi)$ está definida como la única flecha que hace conmutar al cuadrado.

\begin{lemma} \label{conmutadesdeV}
	Sean $f: A \rightarrow A'$ un morfismo de álgebras en $\Vat$, \linebreak $\varphi: V \rightarrow W$ una flecha en $\Vat$, $\phi: A \otimes V \rightarrow A \otimes V$ un morfismo de \linebreak $A$-módulos y $\phi': A' \otimes W \rightarrow A' \otimes W$ un morfismo de $A'$-módulos. Entonces, en el diagrama 
	
	\ \ \ \ \ \ \ \ \ \ \ \ \ \ \ \xymatrix @C=4pc { V \ar[r]^{u \otimes V} \ar[d]_{\varphi} & A \otimes V \ar[r]^{\phi} \ar[d]^{f \otimes \varphi} & A \otimes V \ar[d]^{f \otimes \varphi} \\
								W \ar[r]^{u \otimes W} & A' \otimes W \ar[r]^{\phi'} & A' \otimes W }
	
	si el rectángulo exterior conmuta entonces el cuadrado de la derecha conmuta. 
\end{lemma}

\begin{proof}

	Teniendo en cuenta cómo construimos la biyección en \eqref{biyeccionAV}, multiplicamos al rectángulo por $A$ a izquierda y realizamos varias composiciones para obtener
	
\footnotesize
\xymatrix @C=4pc { A \otimes V \ar@/_4pc/[dd]^{f \otimes \varphi} \ar[d]^{A \otimes \varphi} \ar[r]^{A \otimes u \otimes V} \ar@/^2pc/[rrr]^{\phi} & A \otimes A \otimes V \ar[r]^{A \otimes \phi} & A \otimes A \otimes V \ar[d]^{A \otimes f \otimes \varphi} \ar[r]^{m \otimes V} \ar@/^5pc/[dd]^{f \otimes f \otimes \varphi} & A \otimes V \ar@/^1pc/[dd]^{f \otimes \varphi} \\
								A \otimes W \ar[r]^{A \otimes u \otimes W} \ar[d]^{f \otimes W} & A \otimes A' \otimes W \ar[r]^{A \otimes \phi'} & A \otimes A' \otimes W \ar[d]^{f \otimes A' \otimes W}  \\
								A' \otimes W \ar[r]^{A' \otimes u \otimes W} \ar@/_2pc/[rrr]^{\phi'} & A' \otimes A' \otimes W \ar[r]^{A' \otimes \phi'} & A' \otimes A' \otimes W \ar[r]^{m \otimes W} & A' \otimes W }
\normalsize
								
\vspace{2ex}

	La conmutatividad de todos los diagramas pequeños nos da entonces la conmutatividad del diagrama exterior, que es el cuadrado deseado.
\end{proof}

Veamos que el funtor $GL_V$ respeta las identidades y la composición, luego tendremos gratis que $GL_V(f)(\phi)$ es un automorfismo. Ambas propiedades son válidas por la unicidad de $GL_V(f)(\phi)$ para hacer conmutativo el cuadrado de \eqref{diagramaGL}, mirando los diagramas

ident: \xymatrix { A \otimes V \ar[r]^{A \otimes V} \ar[d]^{f \otimes V} & A \otimes V \ar[d]^{f \otimes V} \\
						A' \otimes V \ar@{.>}[r]^{A' \otimes V} & A' \otimes V }
\ \ \ comp: \xymatrix @C=3pc { A \otimes V \ar[rr]^{\phi} \ar[d]^{f \otimes V} & & A \otimes V \ar[d]^{f \otimes V} \\
						A' \otimes V \ar@{.>}[rr]^{GL_V(f)(\phi)} \ar[d]^{f' \otimes V} & & A' \otimes V \ar[d]^{f' \otimes V} \\
						A'' \otimes V \ar@{.>}[rr]^{GL_V(f')(GL_V(f)(\phi))} & & A'' \otimes V 	}
						
Como $GL_V (A)$ es un grupo para todo $A$ y $GL_V(f)$ es un morfismo de grupos para toda $f$, tenemos que $GL_V \in \Gr^{Alg_\Vat}$.

Observemos también que la conmutatividad del rectángulo

\ \ \ \ \ \ \ \ \ \ \ \ \xymatrix @C=4pc {A \otimes V \ar[r]^{\phi} \ar[d]_{f \otimes V} & A \otimes V \ar[d]^{f \otimes V} \ar[r]^{\phi'} & A \otimes V \ar[d]^{f \otimes V} \\
					 A' \otimes V \ar[r]^{GL_V(f)(\phi)}  & A' \otimes V \ar[r]^{GL_V(f)(\phi')} & A' \otimes V }

nos da la igualdad

\begin{equation} \label{GLV}
	GL_V(f)(\phi' \circ \phi) = GL_V(f)(\phi') \circ GL_V(f)(\phi) 
\end{equation}

Todo esto nos permite realizar la siguiente

\begin{\de}
	Sea $B$ un álgebra de Hopf en $\Vat$, $V$ un objeto de $\Vat_0$, y $G=Spec(B)$ el objeto grupo en ${Alg_\Vat}^{op}$. Una representación a derecha de $G$ en $V$ es una tranformación natural $G^{op} \stackrel{\theta}{\Rightarrow} GL_V$ en $\Gr^{Alg_\Vat}$. 
	
Para cada álgebra $A$ en $\Vat$, que $\theta_A$ sea morfismo de grupos significa que para todo par de morfismos de álgebras $f, g: B \rightarrow A$, se tiene 
$$\theta_A(B \stackrel{\varepsilon}{\rightarrow} I \stackrel{u}{\rightarrow} A) = id_{A \otimes V} \hbox {\ \ \ \ y \ \ \ \ } \theta_A(f) \circ \theta_A(g) = \theta_A(m \circ (g \otimes f) \circ \triangle)$$

\end{\de}

Pasamos ahora a definir los morfismos de representaciones. Notemos que, dadas $\theta: G \Rightarrow GL_V$ y $\theta': G \Rightarrow GL_W$ dos representaciones de $G$, para cada álgebra $A$ en $\Vat$ lo que tenemos son dos representaciones del grupo $[B,A]$ en los $A$-módulos $A \otimes V$ y $A \otimes W$. Notemos también que una flecha $\varphi: V \rightarrow W$ en $\Vat$ nos da el morfismo de $A$-módulos $A \otimes \varphi: A \otimes V \rightarrow A \otimes W$. Realizamos entonces la siguiente definición:

\begin{\de}
	Con la notación del párrafo anterior, un morfismo entre las re- presentaciones $\theta$ y $\theta'$ es una flecha $\varphi: V \rightarrow W$ tal que para cada álgebra $A$ en $\Vat$, el morfismo de $A$-módulos $A \otimes \varphi$ es un morfismo entre las representaciones $\theta_A$ y $\theta_A'$. Es decir, tal que para todo morfismo de álgebras $a: B \rightarrow A$ el siguiente cuadrado conmuta.
	
\begin{equation} \label{morfrepr}
\xymatrix { A \otimes V \ar[r]^{\theta_A(a)} \ar[d]_{A \otimes \varphi} & A \otimes V \ar[d]^{A \otimes \varphi} \\
						A \otimes W \ar[r]^{\theta_A'(a)} & A \otimes W }
\end{equation}
\end{\de}

Observemos que por el lema \ref{conmutadesdeV} (en el caso $A=A'$, $f=id$), que el diagrama \eqref{morfrepr} conmute es equivalente a que lo haga el rectángulo exterior en el diagrama

\begin{equation} \label{morfreprdesdeV}
\xymatrix @C=4pc { V \ar[r]^{u \otimes V} \ar[d]_{\varphi} & A \otimes V \ar[r]^{\theta_A(a)} \ar[d]^{A \otimes \varphi} & A \otimes V \ar[d]^{A \otimes \varphi} \\
								W \ar[r]^{u \otimes W} & A \otimes W \ar[r]^{\theta_A'(a)} & A \otimes W }
\end{equation}

\begin{\de}
	De esta forma queda definida la categoría $Rep_0(G)$ de las representaciones del objeto grupo $G$ en los objetos de $\Vat_0$. Dejamos por verificar que los morfismos de representaciones son cerrados por composición.
\end{\de}

\subsection{El segundo teorema}

La siguiente propiedad nos permite ver un paralelo entre la teoría de \linebreak Tannaka como fue expuesta en esta tesis y la teoría de Galois de Grothendieck.

\begin{\prop}
	Sea $B$ un álgebra de Hopf en $\Vat$, $V$ un objeto de $\Vat_0$, y \linebreak $G=Spec(B)$ el objeto grupo en ${Alg_\Vat}^{op}$. Es equivalente tener una estructura de $B$-comódulo a izquierda en $V$ que tener una representación a derecha de $G$ en $V$. Es equivalente también tener un morfismo entre los comódulos que entre las representaciones. Esto nos da un isomorfismo entre las categorías $Comod_0(B)$ y $Rep_0(G)$.
\end{\prop}

\begin{proof}
	Mirando una representación $\theta: G \Rightarrow GL_V$ como tranformación natural entre los funtores subyacentes $G, GL_V \in \Ens^{Alg_\Vat}$, esta equivale por Yoneda a $\theta \in GL_V(B) = Aut_B(B \otimes V)$. $\theta$ queda entonces definida por su restricción $\rho: V \rightarrow B \otimes V$. Veremos que el hecho de que la transformación natural $\theta$ respete la unidad y la multiplicación de los objetos grupo es equivalente a que el coproducto escalar $\rho$ respete la counidad y la comultiplicación de $B$.
	
	Recordemos que $\theta$ se obtiene según $\theta = \theta_B(id_B)$, y que a partir del $\theta \in GL_V(B)$ se recupera la transformación natural según \begin{equation} \label{recuptheta} \theta_A(B \stackrel{f}{\rightarrow} A) = GL_V(f)(\theta) \end{equation}
	
	Que $\theta$ preserve la unidad es que $\theta_A(B \stackrel{\varepsilon}{\rightarrow} I \stackrel{u}{\rightarrow} A) = id_{A \otimes V}$, o bien en virtud de \eqref{recuptheta} que $GL_V(B \stackrel{\varepsilon}{\rightarrow} I \stackrel{u}{\rightarrow} A)(\theta) = id_{A \otimes V}$.
	
	Por otro lado, que $\rho$ respete la counidad es que el diagrama 
	
	\ \ \ \ \ \ \ \ \ \ \ \ \xymatrix { I \otimes V & B \otimes V \ar[l]_{\varepsilon \otimes V} \\
							I \otimes V \ar[u]^{I \otimes V} & B \otimes V \ar[u]_{\theta} \ar[l]_{\varepsilon \otimes V} \\
								& V \ar[lu]^{u \otimes V = \cong} \ar[u]_{u \otimes V} \ar@/_3pc/[uu]^{\rho} }
	
	conmute, lo cual equivale a que $GL_V (B \stackrel{\varepsilon}{\rightarrow} I) = id_{I \otimes V}$. Pero como el cuadrado 
	\xymatrix@1 { I \otimes V \ar[r]^{I \otimes V} \ar[d]_{u \otimes V} & I \otimes V \ar[d]^{u \otimes V} \\
								A \otimes V \ar[r]^{A \otimes V} & A \otimes V }
	siempre conmuta, ambas condiciones son equivalentes.
	
	Que $\theta$ preserve la multiplicación es que para todo par de flechas \linebreak $f,g: B \rightarrow A$ se verifique $\theta_A(f) \circ \theta_A(g) = \theta_A(m \circ (g \otimes f) \circ \triangle)$, o bien en virtud de \eqref{recuptheta} que $GL_V(f)(\theta) \circ GL_V(g)(\theta) = GL_V(m \circ (g \otimes f) \circ \triangle) (\theta)$.
	
	Aplicado a $g = B \otimes u$ y $f = u \otimes B$ las inclusiones de $B$ en $B \otimes B$, obtenemos que $GL_V(\triangle)(\theta) = GL_V (B \stackrel{u \otimes B}{\rightarrow} B \otimes B)(\theta) \circ GL_V (B \stackrel{B \otimes u}{\rightarrow} B \otimes B)(\theta)$. Para ver quiénes son $GL_V$ de estas inclusiones, miramos los diagramas conmutativos
	
	1. \xymatrix { B \otimes V \ar[r]^{\theta} \ar[d]_{u \otimes B \otimes V} & B \otimes V \ar[d]^{u \otimes B \otimes V} \\
							B \otimes B \otimes V \ar[r]^{B \otimes \theta} & B \otimes B \otimes V }
	\ \ y
	
	2. \xymatrix @C=2.5pc { B \otimes V \ar[rrr]^{\theta} \ar[d]_{B \otimes u \otimes V} & & & B \otimes V \ar[d]^{B \otimes u \otimes V} \\
							B \otimes B \otimes V \ar[r]^{B \otimes \psi} & B \otimes V \otimes B \ar[r]^{\theta \otimes B} & B \otimes V \otimes B \ar[r]^{B \otimes \psi} & B \otimes B \otimes V }

	La conmutatividad del segundo diagrama se justifica con las igualdades
	
\begin{equation}\label{thetatilde}
\xymatrix@C=-0.2pc{ B \ar@2{-}[d] & & \ar@{-}[ld] \ar@{-}[rd] \ar@{}[d]|{u} & & V \ar@2{-}[d]&   & B \ar@{-}[d] & \ar@{}[d]|{\theta} & V \ar@{-}[d] &    & & & & B \ar@{-}[d] & & \ar@{}[d]|{\theta} & & V \ar@{-}[d]   \\
								 	  B \ar@2{-}[d] & & B \ar@2{-}[rrd]|{\psi} & & V \ar@2{-}[lld]|{\psi} 	   & \ \ \ \ = \ \ \ \ & B \ar@2{-}[d] & & V \ar@2{-}[d] & & \ar@{-}[ld] \ar@{-}[rd] \ar@{}[d]|{u} & & \ \ \ \ = \ \ \ \ & B \ar@2{-}[d] & & \ar@{-}[ld] \ar@{-}[rd] \ar@{}[d]|{u} & & V \ar@2{-}[d]  \\							 
								    B \ar@{-}[d] & \ar@{}[d]|{\theta} & V \ar@{-}[d] & & B \ar@2{-}[d] 		   &   & B \ar@2{-}[d] & & B \ar@2{-}[rrd]|{\psi} & & V \ar@2{-}[lld]|{\psi} &    & & B & & B & & V & & \\							 
								    B \ar@2{-}[d] & & V \ar@2{-}[rrd]|{\psi} & & B \ar@2{-}[lld]|{\psi} 	   &   & B & & B & & V & & & & \\							 
								    B & & B & & V 																													 &   & & & & & }
\end{equation}
	
	que valen por la naturalidad de $\psi$. Luego obtenemos el triángulo conmutativo
	
	\ \ \ \ \ \ \ \ \ \ \ \ \ \ \ \ \ \xymatrix {B \otimes B \otimes V \ar[rr]^{GL_V (\triangle)(\theta)} \ar[dr]_{\tilde{\theta} } & & B \otimes B \otimes V \\
									& B \otimes B \otimes V \ar[ur]_{B \otimes \theta} }

	donde $\tilde{\theta}= (B \otimes \psi) \circ (\theta \otimes B) \circ (B \otimes \psi)$.
	
	Luego el diagrama
	
	\ \ \ \ \ \ \ \ \ \ \ \ \ \ \ \ \ \xymatrix { V \ar[r]^{u \otimes V} \ar@/^2pc/[rrr]^{\rho} \ar[dr]_{u \otimes u \otimes B} & B \otimes V \ar[rr]^{\theta} \ar[d]^{\triangle \otimes V} & & B \otimes V \ar[d]^{\triangle \otimes V} \\
								& B \otimes B \otimes V \ar[rr]^{GL_V (\triangle)(\theta)} \ar[dr]_{\tilde{\theta}} & & B \otimes B \otimes V \\
								& & B \otimes B \otimes V \ar[ur]_{B \otimes \theta} }

	es conmutativo. Verifiquemos que la composición $$V \stackrel{u \otimes u \otimes V}{\rightarrow} B \otimes B \otimes V \stackrel{\tilde{\theta}}{\rightarrow} B \otimes B \otimes V \stackrel{B \otimes \theta}{\rightarrow} B \otimes B \otimes V$$
coincide con $(B \otimes \rho) \circ \rho$, luego obtendremos que $\rho$ respeta la comultiplicación de $B$.
	En efecto, esta composición es

\ \ \ \ \ \ \ \ \ \ \ \ \ \ \ \ \ \xymatrix@C=-0.2pc{ & \ar@{-}[ld] \ar@{-}[rd] \ar@{}[d]|{u} & & \ar@{-}[ld] \ar@{-}[rd] \ar@{}[d]|{u} & & V \ar@2{-}[d] & & 		& \ar@{-}[ld] \ar@{-}[rd] \ar@{}[d]|{u} & & & & V \ar@2{-}[d] 		 & & & & & & V \ar@{-}[lllld] \ar@{}[lld]|{\rho} \ar@{-}[d] \\
								 & B \ar@2{-}[d] & & B \ar@2{-}[rrd]|{\psi} & & V \ar@2{-}[lld]|{\psi} & \ \ \ \ = \ \ \ \ & 			& B \ar@{-}[d] & & \ar@{}[d]|{\theta} & & V \ar@{-}[d] 			& \ \ \ \ = \ \ \ \ & B \ar@2{-}[d] & & & & V \ar@{-}[lld] \ar@{}[ld]|{\rho} \ar@{-}[d] \\							 
								 & B \ar@{-}[d] & \ar@{}[d]|{\theta} & V \ar@{-}[d] & & B \ar@2{-}[d] & & 			& B \ar@2{-}[d] & & \ar@{-}[ld] \ar@{-}[rd] \ar@{}[d]|{u} & & V \ar@2{-}[d]  		& & B & & B & & V \\							 
								 & B \ar@2{-}[d] & & V \ar@2{-}[rrd]|{\psi} & & B \ar@2{-}[lld]|{\psi} & & 			& B \ar@2{-}[d] & & B \ar@{-}[d] & \ar@{}[d]|{\theta} & V \ar@{-}[d] \\							 
								 & B \ar@2{-}[d] & & B \ar@{-}[d] & \ar@{}[d]|{\theta} & V \ar@{-}[d] & & 			& B & & B & & V  \\	
								 & B & & B & & V & & 					 }

	donde la primer igualdad vale por \eqref{thetatilde} y la segunda porque $\rho$ es la restricción de $\theta$ a $V$.
	
	
	Por último, si $\rho$ respeta la comultiplicación de $B$, realizando el camino inverso llegamos a que 
	$$GL_V(\triangle)(\theta) = GL_V (B \stackrel{u \otimes B}{\rightarrow} B \otimes B)(\theta) \circ GL_V (B \stackrel{B \otimes u}{\rightarrow} B \otimes B)(\theta)$$
	Pero como para todo par de flechas $f,g: B \rightarrow A$ vale que $f = m \circ (g \otimes f) \circ (u \otimes B)$ y $g = m \circ (g \otimes f) \circ (B \otimes u)$, usando la funtorialidad de $GL_V$ y la igualdad \eqref{GLV} obtenemos
\begin{eqnarray*}
  & GL_V(f)(\theta) \circ GL_V(g)(\theta) = & \\
  & = GL_V(m \circ (g \otimes f) \circ (u \otimes B))(\theta) \circ GL_V(m \circ (g \otimes f) \circ (B \otimes u))(\theta) = & \\
	& = GL_V(m \circ (g \otimes f))(GL_V(u \otimes B)(\theta)) \circ GL_V(m \circ (g \otimes f)) (GL_V(B \otimes u)(\theta)) = & \\
	& = GL_V(m \circ (g \otimes f)) (GL_V(u \otimes B)(\theta) \circ GL_V(B \otimes u)(\theta) ) = & \\
	& = GL_V(m \circ (g \otimes f))(GL_V(\triangle)(\theta)) = GL_V (m \circ (g \otimes f) \circ \triangle) (\theta)
\end{eqnarray*}

y obtenemos que $\theta$ preserva (de forma contravariante) la multiplicación. Pasamos ahora a los morfismos. Sean $\theta: G \Rightarrow GL_V$, $\theta': G \Rightarrow GL_W$ dos representaciones; $\rho: V \rightarrow B \otimes V$, $\rho': W \rightarrow B \otimes W$ las estructuras de $B$-comódulo correspondientes y $\varphi: V \rightarrow W$ una flecha en $\Vat$. Si $f$ es un morfismo de representaciones, entonces como $\theta = \theta_B(id_B)$ y $\theta' = \theta'_B(id_B)$ en virtud de \eqref{morfreprdesdeV} el diagrama 

\ \ \ \ \ \ \ \ \ \ \ \ \ \ \ \ \ \ \ \ \ \ 
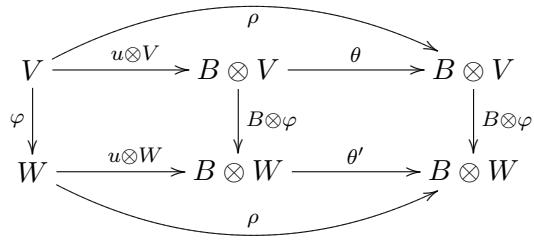
\xymatrix @C=4pc { V \ar@/^2pc/[rr]_{\rho} \ar[r]^{u \otimes V} \ar[d]_{\varphi} & B \otimes V \ar[r]^{\theta} \ar[d]^{B \otimes \varphi} & B \otimes V \ar[d]^{B \otimes \varphi} \\
								W \ar@/_2pc/[rr]^{\rho} \ar[r]^{u \otimes W} & B \otimes W \ar[r]^{\theta'} & B \otimes W }

\vspace{2ex}
								
es conmutativo y por lo tanto $f$ es morfismo de $B$-comódulos. Recíprocamente, si $f$ es morfismo de $B$-comódulos sea $a: B \rightarrow A$ un morfismo de álgebras y veamos que el rectángulo exterior conmuta en el diagrama \eqref{morfreprdesdeV}. Para ello necesitamos considerar el siguiente diagrama tridimensional

\vspace{1ex}

\ \ \ \ \ \ \ \ \ \ 
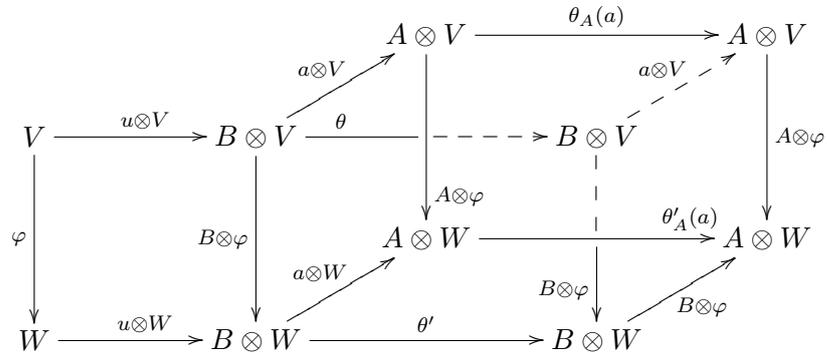
\xymatrix { & & & A \otimes V \ar[rr]^{\theta_A(a)} \ar[dd]^>>>>{A \otimes \varphi} & & A \otimes V \ar[dd]^{A \otimes \varphi} \\
						V \ar[rr]^{u \otimes V} \ar[dd]_{\varphi} & & B \otimes V \ar@{-}[r]^{\theta} \ar[ur]^{a \otimes V} \ar[dd]_{B \otimes \varphi} & \ar@{-->}[r] & B \otimes V \ar@{--}[d] \ar@{-->}[ur]^{a \otimes V} \\
						& & & A \otimes W \ar[rr]^>>>>{\theta'_A(a)} & \ar[d]_{B \otimes \varphi} & A \otimes W \\
						W \ar[rr]^{u \otimes W} & & B \otimes W \ar[rr]^{\theta'} \ar[ur]^{a \otimes W} & & B \otimes W \ar[ur]_{B \otimes \varphi} }

\vspace{1ex}

En el cubo, las caras laterales son conmutativas: la izquierda y la derecha trivialmente y las otras dos por \eqref{recuptheta}. Es conveniente numerar las flechas del diagrama para poder hacer referecia a ellas:
	
\ \ \ \ \ \ \ \ \ \ 
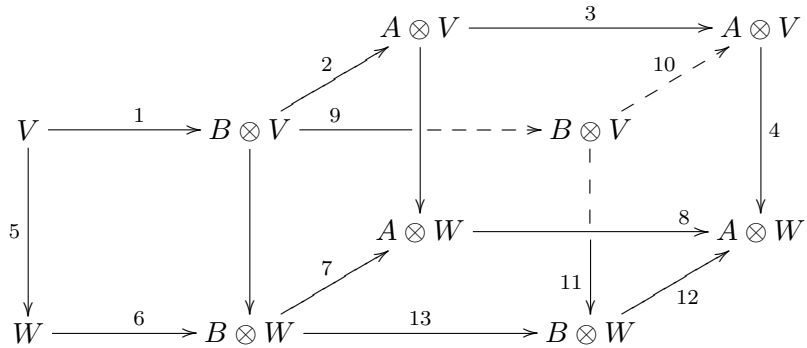
\xymatrix { & & & A \otimes V \ar[rr]^{3} \ar[dd]^>>>>{} & & A \otimes V \ar[dd]^{4} \\
						V \ar[rr]^{1} \ar[dd]_{5} & & B \otimes V \ar@{-}[r]^{9} \ar[ur]^{2} \ar[dd]_{} & \ar@{-->}[r] & B \otimes V \ar@{--}[d] \ar@{-->}[ur]^{10} \\
						& & & A \otimes W \ar[rr]^>>>>{8} & \ar[d]_{11} & A \otimes W \\
						W \ar[rr]^{6} & & B \otimes W \ar[rr]^{13} \ar[ur]^{7} & & B \otimes W \ar[ur]_{12} }

\vspace{1ex}
						
Entonces tenemos la cadena de igualdades (que se lee en el orden opuesto al usual para la composición)
$$1\hbox{-}2\hbox{-}3\hbox{-}4 \ = \ 1\hbox{-}9\hbox{-}10\hbox{-}4 \ = \ 1\hbox{-}9\hbox{-}11\hbox{-}12 \ = \ 5\hbox{-}6\hbox{-}13\hbox{-}12 \ = \ 5\hbox{-}6\hbox{-}7\hbox{-}8$$
que nos dice que en el diagrama

\vspace{1ex}

\ \ \ \ \ \ \ \ \ \ \ \ \ \ \ \ \xymatrix { V \ar@/^2pc/[rr]_{u \otimes V} \ar[r]^{u \otimes V} \ar[d]_{\varphi} & B \otimes V \ar[r]^{a \otimes V}	& A \otimes V	\ar[r]^{\theta_A(a)} & A \otimes V \ar[d]^{A \otimes \varphi} \\
								W \ar@/_2pc/[rr]^{u \otimes W} \ar[r]^{u \otimes W} & B \otimes W \ar[r]^{a \otimes W} & A \otimes W \ar[r]^{\theta_A'(a)} & B \otimes W }

\vspace{2ex}
								
el rectángulo es conmutativo. Pero como $a$ es morfimso de álgebras también los triángulos superior e inferior son conmutativos y obtenemos que el diagrama

\ \ \ \ \ \ \ \ \ \ \ \ \ \ \ \xymatrix { V \ar[r]^{u \otimes V} \ar[d]_{\varphi} & A \otimes V	\ar[r]^{\theta_A(a)} & A \otimes V \ar[d]^{A \otimes \varphi} \\
						W \ar[r]^{u \otimes W} & A \otimes W \ar[r]^{\theta_A'(a)} & B \otimes W }

conmuta y por lo tanto $f$ es morfismo de representaciones.
\end{proof}

Hemos conseguido entonces el siguiente teorema de levantamiento para $F$ equivalente al de la sección \ref{sec:levdeF}, expresado ahora de forma análoga al enunciado 2 de Galois (ver sección \ref{sec:introGal})

\begin{\te} \label{teofinal}
	Sean $\Cat$ una categoría tensorial con una dualidad a derecha, $\Vat$ una categoría tensorial simétrica cocompleta con hom internos, $\Vat_0$ una subca- \linebreak tegoría de $\Vat$ en la cual todo objeto tiene un dual, y $F: \Cat \rightarrow \Vat_0$ un funtor tensorial. Entonces se tiene el levantamiento
	
	\ \ \ \ \ \ \ \ \ \ \ \ \xymatrix {  & Rep_0(G) \ar[d]^{U} \\
							{\Cat} \ar[ur]^{\tilde{F}} \ar[r]^{F} & \Vat_0 }
							
\hspace{1ex}

	Notar que $G=Spec(End^\lor(F))$ ocupa el lugar que $Aut(F)$ ocupa en la teoría de Galois.
	
	Además, en el caso $\Vat = Vec_K$, y $\Cat$ y $F$ enriquecidos, si $\Cat$ es abeliana y $F$ es exacto y fiel entonces $\tilde{F}$ es una equivalencia de categorías.
\end{\te}


\nocite{*}

\pagebreak
\bibliographystyle{unsrt}
\bibliography{Tesis}
\end{document}